\documentclass[12pt,reqno]{amsart}
\usepackage{amsmath,amssymb,amsthm,verbatim,times,hyperref,color,txfonts}
\usepackage[letterpaper,margin=1.1in]{geometry}

\usepackage{wrapfig}
\usepackage{tikz}
\usetikzlibrary{decorations.fractals}

\hypersetup{
  pdftitle={Effective Reifenberg theorem in Hilbert and Banach spaces},
  pdfauthor={Nicholas Edelen and Aaron Naber and Daniele Valtorta},
  pdfsubject={},
  pdfkeywords={},
  pdfpagelayout=SinglePage,
  pdfpagemode=UseOutlines,
  colorlinks,
  bookmarksopen,
  linkcolor=[rgb]{0,0,0.7},
  urlcolor=[rgb]{0,0,0.4},
  citecolor=[rgb]{0.4,0.1,0}
}

\newtheorem{theorem}[subsection]{Theorem}
\newtheorem{lemma}[subsection]{Lemma}
\newtheorem{corollary}[subsection]{Corollary}
\newtheorem{prop}[subsection]{Proposition}

\theoremstyle{definition}
\newtheorem{definition}[subsection]{Definition}
\newtheorem{remark}[subsection]{Remark}
\newtheorem{example}[subsection]{Example}

\newcommand{\graph}{\mathrm{graph}}

\newcommand{\spt}{\mathrm{spt}}
\newcommand{\haus}{\mathcal{H}}

\newcommand{\vecspan}{\mathrm{span}}

\newcommand{\cS}{\mathcal{S}}

\newcommand{\cG}{\mathcal{G}}
\newcommand{\cB}{\mathcal{B}}

\newcommand{\cH}{\mathcal{H}}
\newcommand{\cJ}{\mathcal{J}}
\newcommand{\er}{\mathfrak{r}}
\newcommand{\eps}{\epsilon}

\newcommand{\lip}{\mathrm{Lip}}
\newcommand{\ind}{1}

\newcommand{\B}[2]{B_{#1}\ton{#2}}
\newcommand{\norm}[1]{\left\|#1\right\|}
\newcommand{\ps}[2]{\left\langle#1,#2\right\rangle}
\newcommand{\ton}[1]{\left(#1\right)}
\newcommand{\qua}[1]{\left[#1\right]}
\newcommand{\cur}[1]{\left\{#1\right\}}
\newcommand{\abs}[1]{\left|#1\right|}

\newcommand{\N}{\mathbb{N}}

\newcommand{\R}{\mathbb{R}}

\numberwithin{equation}{section}

\begin{document}

\title{Effective Reifenberg theorems in Hilbert and Banach spaces}
\author{Nick Edelen, Aaron Naber, and Daniele Valtorta}\thanks{The first author was supported by NSF grant DMS-1606492, the second author has been supported by NSF grant DMS-1406259, the third author has been supported by SNSF grant 200021\_159403/1}
\date{\today}

\begin{abstract}
A famous theorem by Reifenberg states that closed subsets of $\R^n$ that look sufficiently close to $k$-dimensional at all scales are actually $C^{0,\gamma}$ equivalent to $k$-dimensional subspaces.  Since then a variety of generalizations have entered the literature.  For a general measure $\mu$ in $\R^n$, one may introduce the $k$-dimensional Jone's $\beta^k$-numbers of the measure, where $\beta^k(x,r)$ quantifies on a given ball $B_r(x)$ how closely the support of the measure is to living inside a $k$-dimensional subspace.  Recently, it has been proven that if these $\beta$-numbers satisfy the uniform summability estimate $\int_0^2 \beta^k(x,r)^2 \frac{dr}{r}<M$, then $\mu$ must be rectifiable with uniform measure bounds.  Note that one only needs the {\it square} of the $\beta^k$-numbers to satisfy the summability estimate, this power gain has played an important role in the applications, for instance in the study of singular sets of geometric equations.  One may also weaken these pointwise summability bounds to bounds which are more integral in nature.\\

The aim of this article is to study these effective Reifenberg theorems for measures in a Hilbert or Banach space.  For Hilbert spaces, we see all the results from $\R^n$ continue to hold with no additional restrictions.  For a general Banach spaces we will see that the classical Reifenberg theorem holds, and that a weak version of the effective Reifenberg theorem holds in that if one assumes a summability estimate $\int_0^2 \beta^k(x,r)^1 \frac{dr}{r}<M$ {\it without power gain}, then $\mu$ must again be rectifiable with measure estimates.  Improving this estimate in order to obtain a power gain turns out to be a subtle issue.  For $k=1$ we will see for a {\it uniformly smooth} Banach space that if $\int_0^2 \beta^1(x,r)^\alpha \frac{dr}{r}<M^{\alpha/2}$, where $\alpha$ is the smoothness power of the Banach space, then $\mu$ is again rectifiable with uniform measure estimates.  
\end{abstract}

\maketitle
\tableofcontents

\section{Introduction}

A famous theorem by Reifenberg \cite{reif_orig} states that a closed subset of $\R^n$ that looks sufficiently close to a $k$-dimensional plane at all scales is $C^{0,\gamma}$-equivalent to a $k$-plane.  Easy examples show that in general H\"older cannot be improved to Lipschitz -- in fact there are examples satisfying Reifenberg's theorem that have dimension $> k$.  A set satisfying Reifenberg's theorem is often called \emph{Reifenberg flat}.

Various works have proved refinements of Reifenberg's theorem, demonstrating a Lipschitz equivalence and/or effective measure bounds assuming some summability condition on the ``$k$-dimensional excess.''  Often these theorems are called ``analyst traveling salesman'' type problems, following Jones' original work \cite{jones} concerning rectifiable curves in $\R^2$ (later extended to arbitrary codimension by \cite{okikiolu}).

Jones \cite{jones} and David-Semmes \cite{david-semmes} introduced various quantities now called \emph{Jones $\beta$-numbers}, which give an $L^p$-notion of how $k$-dimensional a measure is (in this paper we shall deal almost exclusively with the $L^2$ $\beta$-numbers).  Let us define them here: given a Borel-regular measure $\mu$ on a normed linear space $X$, the $k$-dimensional $\beta$-number in $B_r(x)$ is
\begin{gather}\label{eqn:beta-defn}
\beta^k_\mu(x, r)^2 = \inf_{p + V^k} r^{-k-2} \int_{B_r(x)} d(z, p + V)^2 d\mu(z)\, ,
\end{gather}
where the infimum is taken over all affine $k$-planes $p + V^k$.  \cite{david-semmes} used the $\beta$-numbers to demonstrate very strong structural results for ``Ahlfors-regular'' measures.

Toro \cite{toro:reifenberg} (and later David-Toro \cite{davidtoro}) gave very direct extensions of Reifenberg's theorem, where they showed that Reifenberg flat sets admitting the summability condition like
\begin{gather}\label{eqn:david-toro-hyp}
\int_0^\infty \beta^k_{\haus^k\llcorner S}(x, r)^2 \frac{dr}{r} \leq M^2\, , \quad \text{for $\haus^k$-a.e. $x \in S$}
\end{gather}
are actually bi-Lipschitz to a $k$-plane.  Notice that it suffices to assume summability of the \emph{squared} $\beta$-numbers.  This extra power gain in $\R^n$ is loosely speaking a consequence of the Pythagorean theorem.

Azzam-Tolsa \cite{azzam-tolsa} and Tolsa \cite{tolsa:jones-rect} further generalized Reifenberg's theorem to say that a measure $\mu$ is countably $k$-rectifiable if and only if
\begin{gather}
0 < \Theta^{*,k}(\mu, x) < \infty\, , \quad \int_0^\infty \beta^k_\mu(x, r)^2 \frac{dr}{r} < \infty \quad \text{ at $\mu$-a.e. $x$}\, ,
\end{gather}
here $\Theta^{*, k}$ being the upper-density. Recently in \cite{tolsa_LD}, the author shows that one can weaken the previous assumption and insist just on bounds on the \textit{lower} density $\Theta^{k}_*(\mu, x) < \infty$.

In the recent article \cite{ENV}, we 
demonstrated effective measure/packing bounds and Lipschitz structure for (possibly infinite) measures satisfying the condition
\begin{gather}
\int_0^\infty \beta^k_\mu(z, r)^2 \frac{dr}{r} \leq M^2 \quad \text{ for $\mu$-a.e. $x$}\, ,
\end{gather}
\emph{without} any additional assumption of $\mu$.  Toy examples show that, in general, one must split $\spt\, \mu$ into a ``low-density'' region of bounded measure, and a rectifiable piece of ``high-density'' which admits packing bounds.

Let us also mention the works of \cite{azzam-schul}, proving a $k$-dimensional version of the Jones' traveling salesman problem for lower-Ahlfors-regular sets; and \cite{BS1}, \cite{BS2}, characterizing $1$-dimensional measures in terms of lower-density and $\beta$-numbers.

There has been recent progress generalizing Reifenberg-type theorems to infinite-dimensional spaces.  In his thesis \cite{schul-hilbert} Schul proved a direct analogue in Hilbert spaces of Jones' original traveling salesman theorem for curves.  Ferrari-Franchi-Pajot \cite{ferrari-franchi-pajot} and Li-Schul \cite{li-schul}, \cite{li-schul-2} have demonstrated the $1$-dimensional traveling salesman theorems in the Heisenberg group, where interestingly in this case the critical power gain is $4$.  Hahlomaa \cite{hahlomaa-metric} extended Jones' theorem to metric spaces, using Menger curvature in place $\beta$ numbers.  We recommend the excellent survey article \cite{schul-survey} for a more comprehensive exposition of these and other results.

\vspace{5mm}

This paper is concerned with studying effective Reifenberg theorems on Banach spaces.  We are particularly interested in when one can expect a power gain in the summability condition, like in \eqref{eqn:david-toro-hyp}.  We shall demonstrate measure/packing bounds and Lipschitz structure for a measure $\mu$ in a Banach space $X$, under the assumption
\begin{gather}\label{eqn:intro-hyp}
\int_0^\infty \beta^k_\mu(x, r)^\alpha \frac{dr}{r} \leq M^{\alpha/2} \quad \text{ for $\mu$-a.e. $x$}\, ,
\end{gather}
where $\alpha \in [1, 2]$ is some exponent depending on $X$ and $k$. Clearly, a bigger $\alpha$ will give a stronger result.

The value of $\alpha$ is intimately tied with the existence of a Pythagorean-type theorem, and relatedly a good notion of projection.  Fundamentally, we need to be able to say that if a unit vector $v$ is pushed ``perpendicularly'' by an amount $\delta$, then the length of $v$ changes by $\approx \delta^\alpha$.  In practice this manifests itself in an improved bi-Lipschitz estimate for graphs, which says that if $f : V \to X$ is a ``graph'' over some plane $V$, with $\lip(f) \leq \eps$, then
\begin{gather}\label{eqn:teaser-power-gain}
\Big| ||(x + f(x)) - (y + f(y))||^2 - ||x - y||^2 \Big| \leq c \eps^\alpha ||x - y||^2 \quad \forall x, y \in L \, .
\end{gather}

\vspace{5mm}

We will find that in any Hilbert space $\alpha = 2$, as there are natural notions of orthogonality and the Pythagorean theorem holds. In particular, given two mutually orthogonal unit vectors $v, w$, then $\norm{v + tw}^2 \approx 1 + t^2$.  With this property, essentially the same proof of the Reifenberg theorem in \cite{ENV} carries over, although some care must be taken to ensure that the estimates depend only on $k$, and not on the dimension of the ambient space (which can be infinity). 

In a general Banach space we only have $\alpha = 1$.  One can only construct crude notions of projection, and no Pythagorean-type estimate holds. Indeed, it is easy to construct examples where the best estimate possible for unit vectors $v, w$ is $||v + tw|| \leq 1 + t$, i.e. the triangle inequality, see the example in Section \ref{sec_ex_R2}.

\vspace{5mm}
The situation becomes more interesting when $X$ is a \emph{smooth} Banach space.  In general the \emph{modulus of smoothness} attached to any Banach space, denoted $\rho_X(t)$, roughly measures the regularity of the unit sphere at scale $t$.  More precisely, 
 \begin{gather}\label{e:intro:modulus_smoothness}
  \rho_X(t)= \sup_{\norm x =1, \,  \ \ \norm y = t} \ton{\frac{\norm{x+y}+\norm{x-y}}{2}} - 1\, .
 \end{gather}

The faster $\rho_X(t)$ decays with $t \to 0$, the more regular the space.
The triangle inequality always gives the crude bound $\rho_X(t) \leq t$, while the best bound $\rho_X(t) \leq \sqrt{1 + t^2} - 1$ is achieved only by Hilbert spaces (see \cite{nordlander} and \cite[Proposition 1e2 p 61]{LT2}).  In $L^p$ spaces we have
\begin{gather}
\rho_{L^p}(t) \leq \left\{ \begin{array}{l l}  p^{-1} t^p & (1 < p \leq 2) \\ (p-1)t^2 & (2 < p < \infty) \end{array} \right . 
\end{gather}
$X$ is called \emph{smooth} if $\rho_X(t) = o(t)$.  See Section \ref{sec:smoothness} for details and references.

It turns out that when $k = 1$, and $X$ is smooth, then we have a good notion of projection, and a related Pythagorean theorem which says that when $v$, $w$ are ``orthogonal'' unit vectors, then $||v + tw|| \approx 1 + \rho_X(t)$.  In this case we can take $\alpha$ to be the \emph{power of smoothness}, which is basically the largest number for which $\rho_X(t) = O(t^\alpha)$.  The example in Section \ref{sec_ex_R2} provides a good intuition for this case.

We shall see in Example \ref{sec:no-power-gain} that even in finite dimensions the power gain of \eqref{eqn:teaser-power-gain} \emph{breaks} when $k \geq 2$ and $X$ is not Hilbert.  The lack of an improved estimate \eqref{eqn:teaser-power-gain} shows that the bi-Lipschitz bound of Theorem \ref{thm:improved-reif} fails when $k \geq 2$, and strongly suggests that the measure/packing bounds of Theorem \ref{thm:main-packing} do not admit a power gain $\alpha > 1$ for general smooth $X$ and $k \geq 2$.

\section{Main theorems}

Our main theorem is a combination measure and packing estimate for $\mu$ satisfying a summability condition like \eqref{eqn:intro-hyp}.  The theorem effectively splits $B_1(0)$ into a region of ``low-density'' with measure bounds, and a region of ``high-density'' with packing bounds.  Without further assumptions on $\mu$ easy examples show this kind of decomposition is necessary.

\begin{theorem}\label{thm:main-packing}
Let $X$ be a Banach space, and $\mu$ be a finite Borel measure with $\mu(X \setminus B_1(0)) = 0$.  Take $\cS \subset B_1(0)$ a set of full $\mu$-measure, and $r_s : \cS \to \R_+$ a radius function satisfying $0<r_s < 1$.  Assume $\mu$ satisfies
\begin{gather}\label{eqn:main-assumption}
\int_{r_s}^2 \beta^k_\mu(s, r)^\alpha \frac{dr}{r} \leq M^{\alpha/2} \quad \forall s \in \cS\, ,
\end{gather}
where $\alpha$ is the critical exponent for our problem.  Precisely:
 \begin{enumerate}
  \item[i.] if $X$ is a generic Banach space, then $\alpha=1$,
  \item[ii.] if $X$ is a Hilbert space, then $\alpha=2$,
  \item[iii.]\label{case_smooth} if $X$ is a smooth Banach space, and $k = 1$, then $\alpha$ is the smoothness power of the Banach space $X$.
 \end{enumerate}

Then there is a subcollection $\cS' \subset \cS$, so that we have the packing/measure estimate
\begin{gather}
\mu \left( B_1(0) \setminus \bigcup_{s' \in \cS'} B_{r_{s'}}(s') \right) \leq c(k,\rho_X)M\, , \quad \text{and} \quad \sum_{s' \in \cS'} r_{s'}^k \leq c(k, \rho_X)\, .
\end{gather}

\end{theorem}
\begin{remark}
 Note that by standard measure theory arguments, a finite Borel measure on a metric space is Borel-regular, see \cite[theorem II, 1.2, pag 27]{parthasarathy}.
\end{remark}

Recall from \eqref{e:intro:modulus_smoothness} the modulus of smoothness $\rho_X(t)$ and the smoothness power $\alpha$ for a Banach space $X$.  We will recall the precise definitions of these objects in Section \ref{sec:smoothness}, here we simply remind the reader that $\alpha \in [1,2]$ and its ``best'' value $\alpha = 2$ is achieved by any Hilbert space.  For a general Banach space we have $\alpha \geq 1$; and for $X = L^p$ we have $\alpha = \min\{p, 2\}$ when $1 \leq p < \infty$, and $\alpha = 1$ when $p = \infty$. 

As a corollary, when $\mu$ is discrete or has a priori density control, we obtain a measure bound directly.  Moreover, we can easily weaken the pointwise assumption \eqref{eqn:main-assumption} to an weak-$L^1$ type assumption.  Precisely, we have the following theorem.
\begin{corollary}[Discrete- and Continuous-Reifenberg]\label{cor:discrete}
Let $X$ be a Banach space, and let $\mu$ be a Borel measure with $\mu(X \setminus B_1(0)) = 0$.  Suppose $\mu$ satisfies
\begin{gather}
\mu \left( z \in B_1(0) : \int_0^2 \beta^k_\mu(z, r)^\alpha \frac{dr}{r} > M^{\alpha/2} \right) \leq \Gamma\, ,
\end{gather}
where $\alpha$ is the critical exponent for our problem as defined in Theorem \ref{thm:main-packing}.

Suppose additionally \emph{one} of the following:
\begin{enumerate}
\item[A)] $\mu$ is a packing measure of the form
\begin{gather}
\mu = \sum_{s \in \cS} a_s r_s^k \delta_{s}\, ,
\end{gather}
where $\{B_{r_s}(x_s)\}_s$ are a collection of disjoint balls centered in $B_1(0)$ with $a_s \in (0, b]$ and $0<r_s < 1$; \emph{or}

\item[B)] $\Theta_*^k(\mu, x) \leq b$ for $\mu$-a.e. $x$; \emph{or}

\item[C)] $\mu \leq b \haus^k \llcorner S$ for some subset $S$.
\end{enumerate}

Then
\begin{gather}
\mu(B_1(0)) \leq c(k, \rho_X)(M + b) + \Gamma\, .
\end{gather}
\end{corollary}

\begin{remark}
Notice that no a priori finiteness of $\mu$ is necessary in Corollary \ref{cor:discrete}. 
\end{remark}

Similarly to the Euclidean setting, our methods give not just measure/packing bounds but also a rectifiable structure.  Let us first recall the notion of rectifiability in a general metric space.
\begin{definition}
Let $\mu$ be a Borel-regular measure in a metric space $X$.  We say $\mu$ is countably $k$-rectifiable if there are Lipschitz mappings $\{ f_i : B_1(0) \subset \R^k \to X \}_{i=1}^\infty$ so that
\begin{gather}
\mu\left( X \setminus \bigcup_{i=1}^\infty f_i(B_1(0)) \right) = 0\, ,
\end{gather}
and $\mu$ is absolute continuous w.r.t. $\haus^k$.  We say a subset $S$ of $X$ is countably $k$-rectifiable if $\haus^k \llcorner S$ is countably $k$-rectifiable.
\end{definition}

I changed the definition to ``countably $k$-rectifiable,'' and commented out the remark.  if that's the standard definition it seems silly to use a different one.

We obtain the following analogue of \cite[theorem 1.1]{azzam-tolsa} and \cite{naber-valtorta:harmonic} in the Hilbert-Banach space setting, see also the recent preprint \cite{tolsa_LD}.

\begin{theorem}\label{thm:rect}
Let $X$ be a Banach space, and let $\mu$ be a Borel measure in $X$ with $\mu(X \setminus B_1(0)) = 0$.
Suppose for $\mu$-a.e. $x$ we have the bounds
\begin{gather}\label{eqn:rect-hyp}
\int_0^2 \beta^k_\mu(x, r)^\alpha \frac{dr}{r} < \infty\, , \quad \Theta^k_*(\mu, x) < \infty\, , \quad \Theta^{*, k}(\mu, x) > 0\, ,
\end{gather}
where $\alpha$ is the critical exponent as in Theorem \ref{thm:main-packing}.  Then $\mu$ is countably $k$-rectifiable.
\end{theorem}
In particular, we have the corollary
\begin{corollary}\label{cor:rect}
Let $X$ be a Banach space, and $S \subset B_1(0)$.  Suppose we have
\begin{gather}
\int_0^2 \beta^k_{\haus^k\llcorner S}(x, r)^\alpha \frac{dr}{r} < \infty \quad \text{ for $\haus^k$-a.e. $x \in S$}\, ,
\end{gather}
where $\alpha$ is the critical exponent as in Theorem \ref{thm:main-packing}.  Then $S$ is countably $k$-rectifiable.
\end{corollary}

\vspace{.25cm}

\subsection{Reifenberg-flat sets}
Finally, let us consider the special case when $S$ is a Reifenberg-flat set. In $n$-dimensional Euclidean ambient spaces, this problem has been extensively studied in literature. The main references for this are \cite{toro:reifenberg,davidtoro} for generic $k$. For $k=1$, this problem is closely related to the analyst's traveling salesman problem, and has been studied in \cite{jones,okikiolu}. A nice generalization of this last result in Hilbert spaces has been recently obtained in \cite{schul-hilbert}. As mentioned in the introduction, some results on this are available also in the Heisenberg group setting, see \cite{ferrari-franchi-pajot,li-schul,li-schul-2}, and in the metric space setting, see \cite{hahlomaa-metric}. A recent survey on these results is available in \cite{schul-survey}. Our aim is to extend these results, and in particular \cite[main theorem]{toro:reifenberg} to the general Hilbert-Banach space setting.

\vspace{3mm}

For Reifenberg flat sets, as in Reifenberg's original theorem, we can gain topological information on $S$.  Let us recall that a set $S \subset X$ is called $(k, \delta)$-Reifenberg flat on $\B 1 p$ the following holds:
\begin{gather}
\inf_{V^k} d_H(S \cap B_r(x), (x + V) \cap B_r(x)) \leq \delta r \quad \forall x \in S \cap \B 2 p \text{ and } \forall 0 < r \leq 2\, ,
\end{gather}
where the infimum is taken over all $k$-dimensional linear subspaces $V^k\subset X$.

Let us further define the $\beta_\infty$ numbers, which in the case of Reifenberg-flat sets are perhaps more natural to work with than the $L^2$-$\beta$ numbers above. We set
\begin{gather}
\beta^k_{S, \infty}(x, r) = \inf_{V^k} \{ \delta : S \cap B_r(x) \subset B_{\delta r}(x + V) \}\, .
\end{gather}

When $S$ is sufficiently Reifenberg flat in a Banach space, as a corollary to the proof of Theorem \ref{thm:main-packing}, we can deduce the $S$ is bi-H\"older to a $k$-disk.  If we additionally assume a summability condition on the $\beta_\infty$-numbers like \eqref{eqn:intro-hyp}, then $S$ is bi-Lipschitz to a $k$-disk.

\begin{prop}\label{thm:improved-reif}
Let $X$ be a Banach space, and take $\gamma \in (0, 1)$.  There is a constant $\delta_1(k, \rho_X, \gamma)>0$ so that the following holds.  Let $S$ be a closed, $(k, \delta)$-Reifenberg-flat subset of $X$, with $0 \in S$, and $\delta \leq \delta_1$.  Then we can find a $k$-plane $V^k\subset X$, and mapping $\phi : V^k \to X$, so that $\phi \equiv id$ outside $B_{3/2}(0)$, and $S \cap B_1(0) = \phi(V) \cap B_1(0)$, and $\phi$ has the bi-H\"older bound
\begin{gather}
(1-c(k)\delta)||x - y||^{1/\gamma} \leq ||\phi(x) - \phi(y)|| \leq (1 + c(k)\delta) ||x - y||^\gamma \quad \forall x, y \in V\, .
\end{gather}

If additionally we have a bound of the form
\begin{gather}\label{eqn:reif-hyp}
\int_0^2 \beta_{S, \infty}^k(x, r)^\alpha \frac{dr}{r} \leq Q^\alpha \quad \forall x \in S\, ,
\end{gather}
where $\alpha$ is the critical exponent for our problem as in Theorem \ref{thm:main-packing}, then $\phi$ is a bi-Lipschitz equivalence with 
\begin{gather}\label{eq_2lip_power_gain}
e^{-c(k,\rho_X)Q^\alpha} ||x - y|| \leq ||\phi(x) - \phi(y)|| \leq e^{c(k,\rho_X)Q^\alpha} ||x - y||\, .
\end{gather}
\end{prop}

\vspace{.5cm}

\section{Preliminaries}\label{sec_preliminaries}

In this section, we collect some basic preliminary estimates that will be useful for our main construction. Throughout this paper $X$ will always denote a Banach space.  Any additional properties that we may assume will be made explicit.

We make repeated use of the following elementary principle.  If $f : A \subset X \to X$ is a mapping satisfying 
\begin{gather}\label{eqn:lipeo}
\ell^{-1} ||x - y|| \leq ||f(x) - f(y)|| \leq \ell ||x - y|| \quad \forall x, y \in A
\end{gather}
for some $\ell\geq 1$, then $f$ is a bijection onto its image, with Lipschitz inverse.  We will refer to $f$ satisfying \eqref{eqn:lipeo} as a \emph{bi-Lipschitz equivalence}, with bi-Lipschitz constant bounded by $\ell$.  We note that, trivially, \eqref{eqn:lipeo} is implied by the much stronger condition
\begin{gather}
|| (f(x) - x) - (f(y) - y) || \leq \eps ||x - y|| \quad (\eps < 1)\, .
\end{gather}
Given a Lipschitz function $f : A \subset X \to X$, we write
\begin{gather}
\lip(f) = \sup_{x \neq y \in A} \frac{||f(x) - f(y)||}{||x - y||}\, .
\end{gather}

Typically script letters like $\cG$, $\cB$, $\cS$, etc. will denote collections of ball centers.  We will generally denote elements of such a $\cG$ by the corresponding lower-case letter $g$, and write $r_g$ for the radius function.  So, e.g. $\{B_{r_s}(s) \}_{s \in \cS}$ will be the balls indexed by $\cS$.

We will reserve $\chi < 1$ for the scale parameter, and we shall write $\er_i = \chi^i$ for shorthand.

\vspace{5mm}

We require the following truncated partition of unity.  Its construction is standard but for the reader's convenience we detail it here.
\begin{lemma}\label{lem:pou}
There is an absolute constant $\gamma$ so that the following holds.  Let $\{B_{3r}(x_i)\}_{i \in I}$ be a collection of balls in $X$ with overlap bounded by $\Gamma$, i.e., so that for all $x\in X$:
\begin{gather}
\#\{ i \in I : x \in B_{3r}(x_i) \} \leq \Gamma\, .
\end{gather}

Then there exist Lipschitz functions $\phi_i : X \to [0, 1]$ satisfying:
\begin{gather}
\spt\,  \phi_i \subset B_{3r}(x_i)\, , \quad \sum_i \phi_i = 1 \text{ on } \bigcup_i B_{2.5r}(x_i)\, , \quad \lip(\phi_i) \leq \gamma \Gamma/r\, .
\end{gather}
We may call $\phi_i$ the truncated partition of unity subordinate to $\{B_{3r}(x_i)\}_i$.
\end{lemma}

\begin{proof}
Let $b : \R_+ \to \R_+$ be the piece-wise-linear function $b(t) = (3-t)_+$, and define the Lipschitz functions 
\begin{gather}
\psi_i(x) = b( ||x - x_i||/r )\, .
\end{gather}
$\psi_i$ have the following properties: 
\begin{gather}
\spt\, \psi_i \subset B_{3r}(x_i)\, , \quad 0 \leq \psi_i \leq 3\, , \quad \lip(\psi_i) \leq 1/r\, , \quad \psi_i \geq 1/2 \text{ on } B_{2.5r}(x_i)\, .
\end{gather}
These are our local cutoff functions.

Let
\begin{gather}
s(x) = \sum_i \psi_i(x)\, .
\end{gather}
By the finiteness assumption, $s$ is well-defined and Lipschitz, and satisfies
\begin{gather}
\spt\, s \subset \cup_i B_{3r}(x_i)\, , \quad 0 \leq s \leq 3\Gamma\, , \quad \lip(s) \leq \Gamma/r\, , \quad s \geq 1/2 \text{ on } \cup_i B_{2.5r}(x_i)\, .
\end{gather}

We define the global cut-off.  Let $h : \R_+ \to \R_+$ be the piece-wise linear function
\begin{gather}
h(t) = \left\{ \begin{array}{l l} 0 & t \in [0, 1/4] \\ 4t - 1 & t \in [1/4, 1/2] \\ 1 & t \in [1/2, \infty) \end{array} \right.\,  ,
\end{gather}
so that if we set $f(x) = h(s(x))$, then $f$ satisfies:
\begin{gather}
\spt\,  f \subset \{ s \geq 1/4 \}, \quad 0 \leq f \leq 1\, , \quad \lip(f) \leq 4 \Gamma/r\, , \quad f \equiv 1 \text{ on } \cup_i B_{5r/2}(x_i)\, .
\end{gather}

For each $i$ we now define
\begin{gather}
\phi_i(x) = f(x) \frac{\psi_i(x)}{s(x)}\, .
\end{gather}
Since $\spt\,  (f \psi_i/s ) \subset \{ s \geq 1/4 \}$ one can verify directly this satisfies the required estimates.
\end{proof}
\vspace{.25cm}

\subsection{Beta numbers}

We first recall some standard properties of the $\beta$ numbers.
\begin{lemma}\label{lem:basic-beta}
$\beta$ is monotone wrt $\mu$, in the sense that if $\mu' \leq \mu$, then
\begin{gather}
 \beta^k_{\mu'}(x, r) \leq \beta^k_\mu(x, r)\, . 
\end{gather}
Moreover, from the definition it follows immediately that if $B_r(x) \subset B_R(y)$, then 
\begin{gather}\label{eq_beta_yx}
 \beta^k_\mu(x, r) \leq (R/r)^{k+2}\beta^k_\mu(y, R)\, .
\end{gather}
As an immediate corollary, we have the inequalities
\begin{gather}\label{eq_beta_pws}
\beta^k_\mu(x, r) \leq c(k)\fint_{B_r(x)} \beta^k_\mu(y, 2r) d\mu(y), \quad\text{ and } \quad \beta^k_\mu(x, r) \leq c(k) \int_r^{2r} \beta^k_\mu(x, s) \frac{ds}{s}\, .
\end{gather}
In particular, if $\mu(X \setminus B_1(0)) = 0$, then
\begin{gather}
\int_0^\infty \beta^k_\mu(x, r) \frac{dr}{r} \leq c(k) \int_0^2 \beta^k_\mu(x, r) \frac{dr}{r} \quad \forall x \in B_1(0)\, .
\end{gather}
Finally, we point out that $\beta$ is scale-invariant in the following sense.  If we set $\mu_{x, r}= r^{-k} \mu(x + rA)$, then $\beta^k_{\mu_{x, r}}(0, 1) = \beta^k_\mu(x, r)$.

\end{lemma}

\vspace{3mm}

We also record this easy measure-theoretical lemma about integral bounds on beta number vs pointwise bounds.
\begin{lemma}\label{lemma_technical}
Let $\mu$ be a Borel measure with $\mu\ton{X\setminus \B 1 0}=0$ and with upper Ahlfors bounds
\begin{gather}
\mu(B_r(x)) \leq \Gamma r^k \quad \forall x \in B_1(0), 0 < r < 1\, .
\end{gather}
For all $\delta_1,\delta_2>0$ fixed, if 
\begin{gather}
\int_{B_1(0)} \int_0^2 \beta_{\mu}^k (z, s)^\alpha \frac{dr}{r} d\mu(z) < \infty\, ,
\end{gather}
then for $\mu$-a.e. $x \in B_1(0)$, there exists $R_x > 0$ such that
\begin{equation}\label{eq_technical}
\mu\left\{ z \in B_r(x) : \int_0^{2 r} \beta_{\mu}^k (z, s)^\alpha \frac{ds}{s} > \delta_1 \right\} \leq \delta_2 r^k \quad \forall 0 < r < R_x\, .
\end{equation}
\end{lemma}

\begin{proof}
Let $F$ be the set of points for which \eqref{eq_technical} does not hold. Fix any $0 < R < 1/4$ arbitrarily small.  By definition, for all $x \in F$, there exists some positive $s_x < R$ such that
\begin{gather}\label{eq_countable}
 s_x^k < \frac{1}{\delta_2} \mu \cur{ z \in B_{s_x}(x) : \int_0^{2 s_x} \beta_{\mu}^k (z, s)^\alpha > \delta_1 } \, .
\end{gather}
Choose a Vitali subcovering $\{B_{s_i}(x_i)\}_i$ of $\{B_{s_x}(x)\}_{x \in F}$, so that $\{B_{s_i}(x_i)\}_i$ are pairwise disjoint and
\begin{gather}
 F\subseteq \bigcup_{i} \B {5 s_i}{x_i}\, .
\end{gather}
Notice that by \eqref{eq_countable} and the finiteness of $\mu$, this covering is at most countable. Then we calculate
\begin{align}
\mu(F)
&\leq \sum_i \mu(B_{5s_i}(x_i))\leq 5^k \Gamma \sum_i s_i^k \notag\\
&\leq c(k) \frac{\Gamma}{\delta_2} \sum_i \frac{1}{\delta_1} \int_{B_{s_i}(x_i)} \int_0^{2s_i} \beta_{\mu}^k (z, s)^\alpha \frac{dr}{r} d\mu(z) \\
&\leq c(k) \frac{\Gamma}{\delta_2 \delta_1} \int_{B_1(0)} \int_0^{2R} \beta_{\mu}^k (z, s)^\alpha \frac{dr}{r} d\mu(z) \, .\notag
\end{align}
By dominated convergence, and since $R$ is arbitrarily small, $\mu(F) = 0$.
\end{proof}

\vspace{5mm}

\subsection{General position}
A concept that will be essential for us is the concept of points/vectors in general position. This definition, in one form or another, is already present in literature, but we recall it here for the reader's convenience.

Given a set of vectors $\cur{v_1,\cdots,v_k}$, these vectors are linearly independent if and only if for all $i$, $v_i\neq 0$ and $v_i \not \in \operatorname{span}(v_1,\cdots,v_{i-1})$. Here we recall a quantitatively stable notion of linear independence that will have two main applications: one is to provide us with a notion of ``basis with estimates'' in a Banach space, something resembling an orthonormal basis in the Hilbert case. One other important application will be given in the definition of good and bad balls in Section \ref{sec_tilting}.  

\begin{definition}
 We say that $v_1, \ldots, v_k$ are in $\tau$-general position if for each $i$ we have $\tau \leq \norm{v_i}\leq \tau^{-1}$, and
 \begin{gather}
  v_{i+1}\not \in B_\tau(\operatorname{span}(v_1,\cdots,v_i))\, .
 \end{gather}
Equivalently
\begin{gather}
 d(v_{i+1},\operatorname{span}(v_1,\cdots,v_i))\geq \tau\, .
\end{gather}

\end{definition}

In the following lemma, we see that a choice of basis in general position for a finite dimensional (sub)space $V\subset X$ induces a linear isomorphism between $V$ and $\R^k$ \emph{with uniform estimates}.
\begin{lemma}\label{lemma_GP_bounds}
 Let $v_1, \ldots, v_k$ be vectors in $\tau$-general position in $X$, and let $V$ be its $k$-dimensional span.  Then for any $v \in V$, we can write (uniquely)
\begin{gather}
v = \sum_i \lambda_i v_i,
\end{gather}
where
\begin{gather}\label{eq_GP_est}
c_1(k, \tau)^{-1} ||v|| \leq \sum_i |\lambda_i| \leq c_1(k,\tau) ||v||\, .
\end{gather}
\end{lemma}

\begin{remark}
In this lemma, we are basically saying that if we identify $V$ with $\R^k$ via the basis $v_i$, then the $l^1$ norm in this base is equivalent to the original norm $\norm \cdot$. It is clear that, up to enlarging the constant by another $c(k)$, this statement is true also for all $l^p$ norms in $\R^k$. To me more precise, there is a constant $c(k,\tau)$ so that if $V$ is identified with $\R^k$ via the basis $v_i$, then for any $p \in [1, \infty]$ we have
\begin{gather}
c(k,\tau)^{-1} ||v||_{\ell^p} \leq ||v||_X \leq c(k,\tau)||v||_{\ell^p} \quad \forall v \in V \cong \R^k\, .
\end{gather}
\end{remark}

\begin{proof}
The bound $||v|| \leq \tau^{-1} \sum_i |\lambda_i|$ follows trivially from the triangle inequality.  We prove the other bound. 

We proceed by induction.  The Lemma is obvious for $k = 1$.  Suppose now the Lemma holds for $k-1$, and take $v \in V$ with
\begin{gather}
v = \sum_{i=1}^k \lambda_i v_i\, , 
\end{gather}
and without any loss of generality we can assume $||v|| = 1$.

We claim that $|\lambda_k| \leq 2/\tau$.  Otherwise, we could write
\begin{gather}
v_k = \frac{1}{\lambda_k} v - \sum_{i=1}^{k-1} \frac{\lambda_i}{\lambda_k} v_i \equiv \frac{1}{\lambda_k} v + w\, ,
\end{gather}
for $w \in \vecspan(v_1, \ldots, v_{k-1})$.  In particular, we would have
\begin{gather}
d(v_k, \vecspan(v_1, \ldots, v_{k-1})) \leq ||v_k - w|| \leq \frac{||v||}{|\lambda_k|} \leq \tau/2\, ,
\end{gather}
contradicting $\tau$-general position of the $v_i$.

Therefore, $|\lambda_k| \leq 2/\tau$, and we can write
\begin{gather}
v - \lambda_k v_k = \sum_{i=1}^{k-1} \lambda_i v_i\, ,
\end{gather}
where $||v - \lambda_k v_k|| \leq 1 + 2\tau^{-2}$.  By our inductive hypothesis, we have
\begin{gather}
\sum_{i=1}^{k-1} |\lambda_i| \leq c(k, \tau) (1 + 2\tau^{-2})\, ,
\end{gather}
which proves the Lemma for $k$.  In fact, the inductive argument shows that $|\lambda_i| \leq (1+2\tau^{-2})^{k-i+1} ||v||$.
\end{proof}
\vspace{.2cm}

Lemma \ref{lemma_GP_bounds} implies the following crucial fact: \emph{up to linear transformation with uniform estimates}, any two norms on a finite-dimensional space are equivalent with a constant depending only on dimension.
\begin{lemma}\label{lemma_Banach_equiv}
Let $V$ be a $k$-dimensional plane in a Banach space $X$.  Then for any $\tau \in (0, 1)$, we can find unit vectors $v_i \in V$ lying in $\tau$-general position.  In particular, if we take $\tau = 2/3$, and define the linear map $\phi : (V, ||\cdot||) \to (\R^k, ||\cdot||_2)$ by
\begin{gather}
\phi(v) = (\lambda_i)_i, \quad v = \sum_i \lambda_i v_i,
\end{gather}
(so that $\phi$ identifies $V$ with $\R^k$ via the basis $v_i$), then $\phi$ is a bi-Lipschitz equivalence, with $||\phi|| + ||\phi^{-1}|| \leq c(k)$.
\end{lemma}

\begin{proof}
We construct the $v_i$.  For $v_1$ take any vector in $V$ of length $1$.  By inductive hypothesis, suppose we have constructed $v_1, \ldots, v_i$.  Now by Riesz lemma we can pick $v_{i+1} \in V$ with $||v_{i+1}|| = 1$, and $d(v_{i+1}, \mathrm{span}(v_1, \ldots, v_i)) > \tau$.  By induction we obtain the required $v_i$.  The Lipschitz bound on $\phi$ follows immediately from Lemma \ref{lemma_GP_bounds}.
\end{proof}

Here are some important corollaries of this equivalence.  First, almost-disjoint balls lying close to a $k$-plane in a Banach space admit a $k$-dimensional packing bound.
\begin{lemma}\label{lemma_Banach_packing}
 Let $p+V$ be an affine $k$-dimensional plane in a Banach space $X$, and $\cur{\B {r_i}{x_i}}_{i\in I}$ be a family of pairwise disjoint balls with $r_i \leq R$, $x_i \in B_R(p)$, and $d(x_i, p + V) < r_i/2$. Then
 \begin{gather}
  \sum_i r_i^k\leq c_2(k) R^k\, .
 \end{gather}
\end{lemma}
\begin{proof}
 We can suppose for convenience that $p=0$. For each $i$, choose $x'_i \in V$ with $||x'_i - x_i|| < r_i/2$.  Then $B_{r_i/2}(x'_i) \cap V \subset B_{r_i}(x_i)$.   Take $\phi$ as in the previous Lemma \ref{lemma_Banach_equiv}.  Then we get
 \begin{gather}
  B_{r_i/c(k)}\ton{\phi(x_i') \in \R^k} \subseteq \phi\ton{\B {r_i/2}{x_i' \in V}}\subseteq B_{c(k) r_i}\ton{\phi(x_i') \in \R^k}\, ,
 \end{gather}
and the Euclidean balls $\{ B_{r_i/c(k)}(\phi(x_i') \in \R^k) \}_{i\in I}$ are pairwise disjoint and all contained in the ball $B_{c(k) R}\ton{0 \in \R^k}$. The estimate now follows from standard Euclidean volume arguments.
\end{proof}

 Second, balls in a $k$-plane in $X$ admit uniform upper and lower Hausdorff bounds.
\begin{lemma}\label{lemma_Hk_balls}
 Let $V$ be a $k$-dimensional plane in some Banach space $X$. Then for all $x\in V$,
 \begin{gather}
  c(k)^{-1}r^k \leq \cH^k(\B r x \cap V) \leq c(k) r^k\, .
 \end{gather}
\end{lemma}

\begin{proof}
Direct from the existence of $\phi$ in Lemma \ref{lemma_Banach_equiv}, and the behavior of Hausdorff measure under Lipschitz mappings.
\end{proof}

Third, disjoint balls close to a $k$-plane, and clustered reasonably near a $(k-1)$-plane, admit a $(k-1)$-dimensional packing bound.
\begin{lemma}\label{lem:L-packing}
Let $V$ be a $k$-plane in the Banach space $X$, and take $L$ a $(k-1)$-plane in $V$.  Let $\{ x_i\}_{i \in I}$ be a $2 \chi r/5$-separated set in
\begin{gather}
B_r(0) \cap B_{\chi r/10}(V^k) \cap B_{10 \chi r}(L^{k-1})\, .
\end{gather}
Then for $\chi\leq 1$ we have that $\# I \leq c_B(k) \chi^{1-k}$.
\end{lemma}

\begin{proof}
For each $i$ choose $x_i' \in V$ with $||x_i - x_i'|| < \chi r/10$.  Then the balls $\{B_{\chi r/10}(x_i')\}_i$ are disjoint, and contained in $V \cap B_{2r}(0) \cap B_{11 \chi r}(L)$.  Take $\phi : V \to \R^k$ as in Lemma \ref{lemma_Banach_equiv}.  By the same logic as in the proof of Lemma \ref{lemma_Banach_packing}, we get that the balls
\begin{gather}
\left\{ B_{\chi r/c(k)}\ton{\phi(x_i') \in \R^k}\right\}_i 
\end{gather}
are pairwise disjoint, and contained in set $B_{c(k) r}\ton{0 \in \R^k} \cap B_{c(k)\chi r}\ton{\phi(L) \subset \R^k}$.  The result follows by a standard volume argument.
\end{proof}

We close this section by observing the following stability property for vectors in $\tau$-general position.
\begin{lemma}\label{lemma_GP_stable}
 Suppose $v_1, \ldots, v_k$ are vectors in $\tau$-general position, and vectors $w_i$ are chosen so that
 \begin{gather}
  \norm{w_i-v_i} < \epsilon\, ,
 \end{gather}
then $w_i$ are in $(\tau-c(k,\tau)\epsilon)$-general position.  

Similarly, if $x_0, \ldots, x_k$ are points so that $\{x_i - x_0\}_{i=1}^k$ are in $\tau$-general position, and $y_i$ are chosen so that $||x_i - y_i|| < \eps$, then the vectors $\{y_i - y_0\}_{i=1}^k$ are in $(\tau - 2c(k, \tau)\eps)$-general position.
\end{lemma}
\begin{proof}
We need to show that
\begin{gather}
\norm{w_{i+1} - \sum_{j=1}^i \lambda_j w_j} \geq \tau - c(k,\tau)\eps\,  ,
\end{gather}
for any collection $\lambda_1, \ldots, \lambda_i$ of real numbers.

There is no loss in assuming $\eps \leq 2^{-1}c_1(k, \tau)^{-1}$ ($c_1$ being the constant from Lemma \ref{lemma_GP_bounds}), by requiring $c \geq 2c_1$.  First, suppose $\sum_j |\lambda_j| \geq 2c_1(\tau + \tau^{-1})$.  Then we have by Lemma \ref{lemma_GP_bounds} and our hypothesis:
\begin{gather}
\norm{w_{i+1} - \sum_{j=1}^i \lambda_j w_j} \geq \norm{ \sum_{j=1}^k \lambda_j v_j } - \norm{ \sum_{j=1}^k \lambda_j (v_j - w_j) } - ||w_{i+1}||  \geq ( c_1^{-1} - \eps)  \sum_{j=1}^i |\lambda_j| - \tau^{-1} \geq \tau\, .
\end{gather}
Now suppose $\sum_j |\lambda_j| \leq 2c_1(\tau + \tau^{-1})$.  Then using our hypothesis we obtain
\begin{align}
\norm{w_{i+1} - \sum_{j=1}^i \lambda_j w_j}
&\geq \norm{v_{i+1} - \sum_{j=1}^i \lambda_j v_j} - ||w_{i+1} - v_{i+1}|| - \norm{\sum_{j=1}^i \lambda_j(v_j - w_j)} \geq \tau - \eps(1 + 2c_1(\tau^{-1} + \tau))\, .
\end{align}
This establishes the required bound.  The second assertion follows directly.
\end{proof}

\vspace{.25cm}

\subsection{Distance to subspaces}

Here we recall the notion of Hausdorff distance between sets and Grassmannian distance between linear subspaces, and prove some basic estimates on these two.  We shall see how effective bases give us good estimates over nearby spaces.

\begin{definition}
 Given two sets $A,B\subset X$, the Hausdorff distance between $d_H(A,B)$ is defined as
 \begin{gather}
  d_H(A,B) = \inf \cur{\delta\geq 0 \ \ s.t. \ \ A\subseteq \B \delta B \ \ \text{and } \ \ B \subseteq \B \delta A}\, .
 \end{gather}
 Note that $d_H(A,B)= d_H(\overline A, \overline B)$, and in particular the Hausdorff distance is a full-blown distance only between closed sets. 
\end{definition}

It is clear that the Hausdorff distance by itself cannot provide a reasonable notion of distance between linear subspaces. Indeed, $d_H(V,W)\neq \infty$ only if $V= W$. For this reason, we introduce the Grassmanian distance in the next definition. 
\begin{definition}
 Given two linear subspaces $L,V\subseteq X$, we define the Grassmannian distance between these two as
 \begin{gather}
  d_G(L,V)= d_H(L\cap \B 1 0, V\cap \B 1 0 ) \equiv d_H (L\cap \overline{\B 1 0}, V\cap \overline{\B 1 0} )\, ,
 \end{gather}
Note that if $\dim(L)\neq \dim(V)$, then $d_G(L,V)=1$.
\end{definition}

\vspace{.3 cm}

In the next lemma, we recall a basic fact about linear and affine subspaces. While for two general sets it is highly non true that $A\subseteq \B{\delta}{B}$ implies $B\subseteq \B{c\delta}{A}$, for affine subspaces of the same dimension something similar to that is true.
\begin{lemma}\label{lemma_hdv}
Let $p +V$ and $q + W$ be $k$-dimensional affine subspaces in $X$, with $(p + V) \cap B_{1/2}(0) \neq \emptyset$.  Suppose
\begin{gather}
(p + V) \cap B_1(0) \subset B_{\delta}(q + W)\, .
\end{gather}
Then we have
\begin{gather}
d_H( (p + V) \cap B_1(0), (q + W) \cap B_1(0)) \leq c(k) \delta\, ,
\end{gather}
and in particular, $d_G(V, W) \leq c(k) \delta$.
\end{lemma}
\begin{proof}
 As it is self-evident, the requirement that $V$ and $W$ have the same dimension is crucial for the lemma. This suggests that the proof is based on some argument involving affine basis for $p+V$ and $q+W$ and comparisons between the two.  We can take $\delta \leq \delta_0(k)$ by ensuring $c(k) \geq \delta_0^{-1}$.
 
 Let $p_0\in p+V$ be a point of minimal distance from the origin, so that $\norm{p_0}\leq 1/2$. Take $p_1,\cdots,p_k\in (p+V)\cap B_{9/10}(0)$ a sequence of points such that
 \begin{gather}
  \norm{p_i-p_0}= 1/3\,  \quad \text{ and } \quad p_i\not \in p_0 + \B {2/9}{\operatorname{span}(p_1-p_0,\cdots,p_{i-1}-p_0)}\, .
 \end{gather}
One can find the $p_i$ using the Riesz lemma as in Lemma \ref{lemma_Banach_equiv}.  In particular, this implies that $\cur{\ton{p_i-p_0}}_{i=1}^k$ are vectors in $2/9$-general position in $V$. By hypothesis, we can pick $q_i\in q+V$ such that $\norm{q_i-p_i}\leq 2\delta$.  From Lemma \ref{lemma_GP_stable} the $\{q_i - q_0\}_{i=1}^k$ are in $1/9$-general position provided $\delta_0(k)$ is sufficiently small.

Take some $y \in (q + W) \cap B_1(0)$.  Then by Lemma \ref{lemma_GP_bounds} there are numbers $\alpha_i = \alpha_i(y)$ so that
\begin{gather}
y = q_0 + \sum_{i=1}^k \alpha_i(q_i - q_0), \quad |\alpha_i| \leq c(k)\, .
\end{gather}
If we let $x = x(y) \in p + V$ be the point defined by
\begin{gather}
x = p_0 + \sum_{i=1}^k \alpha_i(p_i - p_0)\, ,
\end{gather}
then
\begin{gather}
||y - x|| \leq ||q_0 - p_0|| + \sum_{i=1}^k |\alpha_i| \left[ ||p_i - q_i|| + ||p_0 - q_0|| \right] \leq c(k)\delta\, .
\end{gather}
Therefore we have $(q + W) \cap B_1(0) \subset B_{c(k)\delta}(p + V)$.

One can easily check that, since $\max\{||p||, ||q||\} \leq 3/4$, we have
\begin{gather}
B_{\eps}((p + V) \cap B_{1+\eps}) \subset B_{10\eps}((p + V) \cap B_1(0))\, ,
\end{gather}
and the same for $q + W$.  The lemma now follows directly.
\end{proof}
\vspace{.25cm }

\subsection{Almost-projections, Graphs}\label{sec:graphs}
In this section, we recall some basic definition and properties of linear bounded projections in Banach spaces, and use this notion to define graphs over finite dimensional subsets. Before beginning, we mention the fact that bounded linear projections over Banach spaces behave differently than in Hilbert spaces. In a Hilbert space $H$, all closed subspaces $V$ have a linear projection $\pi_V$ of norm $1$ and such that $V\oplus \pi_V^{-1}(0)=H$. In Banach spaces norm-one linear projections are very rare objects. Indeed, if a Banach space $X$ of dimension $\geq 3$ admits a norm-one linear projection for all of its two dimensional subspaces, then $X$ is a Hilbert space. This is a classical result in Banach spaces, see the recent survey \cite[section 3]{beata}.

In order to distinguish the nice Hilbert space projections from their rougher Banach counterparts, we are going to call a linear projection on a Banach space with norm bounded (but not by $1$) ``almost projections''.

We start by recalling an easy consequence of Hahn-Banach theorem.
\begin{lemma}\label{lemma_op_ext}
 Let $L:A\to V$ be a continuous linear operator from a linear subspace $A\subset X$ to a $k$-dimensional Banach space $V$. Then there exists a bounded linear extension $\tilde L:X\to V$ satisfying
 \begin{gather}
  ||\tilde L||\leq c(k) ||L||\, .
 \end{gather}
\end{lemma}
\begin{proof}
 Let $\cur{w_i}$ be a unit basis for $V$ lying in $2/3$ general position, see Lemma \ref{lemma_Banach_equiv}, and identify $V$ with $\R^k$ via this basis. By Lemma \ref{lemma_GP_bounds}, we know that $\norm{\cdot}_{L^\infty(\R^k)}$ is uniformly equivalent to the original Banach norm on $V$. In other words, for all $w\in V$, we have that the components $\phi_i:V\to \R$ given by $\phi_i(w)=\phi_i \ton{\sum_i \lambda_i w_i}=\lambda_i $ are uniformly bounded linear maps with
 \begin{gather}
  \norm{\phi_i}\leq c(k)\, .
 \end{gather}
Define $\psi_i : V\to \R$ by setting 
\begin{gather}
 \psi_i(v)= \phi_i(L(v))\, .
\end{gather}
Then we have $\norm{\psi_i}\leq c(k) \norm{L}$, and by Hahn-Banach, for each $i$ there exists a norm preserving extension $\tilde \psi_i$ of $\psi_i$ to the whole space $X$. Set 
\begin{gather}
 \tilde L (x) = \sum_i \tilde \psi_i (x) w_i\, .
\end{gather}
We have
\begin{gather}
 ||\tilde L(x)||_{V} \leq c(k) \norm{L(x)}_{L^\infty(\R^k)} \leq c(k) \norm{L}_V ||x||_V\, .
\end{gather}
\end{proof}

We can use this lemma to give a trivial proof of the following:
\begin{lemma}\label{lem:almost-proj}
For any linear $k$-space $V$, there is a linear map $\pi_V : X \to V$ satisfying the following:

A) $\pi_V(v) = v$ for all $v \in V$, 

B) $||\pi_V|| \leq c_3(k)$,

C) given any $W$ with $d_G(V, W) < \eps$, then $||\pi_V(w) - w|| \leq c_3(k) \eps ||w||$.
\end{lemma}

\begin{proof}
If we let $L : V \to V$ be the identity operator, then take $\pi_V$ to be the linear extension of $L$ from Lemma \ref{lemma_op_ext}.

Now given $w \in W$, where $d_G(W, V) < \eps$, we can by assumption find a $v \in V$ with $||w - v|| \leq \eps ||w||$.  Then we have
\begin{align}
||\pi_V(w) - w|| 
&\leq ||\pi_V(w) - \pi_V(v)|| + ||\pi_V(v) - v|| + ||v - w|| \\
&\leq (1 + c(k))||v - w|| \\
&\leq c(k) \eps ||w||\, .
\end{align}
\end{proof}

\begin{definition}\label{def:almost-proj}
We shall call any linear map $\pi : X \to V$ satisfying the conditions A)-B)-C) of Lemma \ref{lem:almost-proj} an \emph{almost-projection} for $V$.  Given any almost-projection $\pi$, we abuse notation and write $\pi^\perp := Id - \pi$.

Given an affine $k$-space $p + V$, we define $\pi_V$ in terms of the associated linear space $V$.  An affine space $p + V$ admits a notion of \emph{almost-affine-projection}
\begin{gather}
\Pi(x) := p + \pi_V(x - p) \equiv \pi_V^\perp(p) + \pi_V(x)\, ,
\end{gather}
which is independent of choice of $p \in p + V$.
\end{definition}

An important but easy consequence of the definition of almost projections is the following.
\begin{prop}\label{prop:almost-proj}
Let $V$, $W$ be linear $k$-spaces, with almost-projections $\pi_V$, $\pi_W$.  Suppose $d_G(V, W) \leq \delta$.  Then
\begin{gather}
||\pi_V^\perp(\pi_W(x))|| \leq c_3(k)^2 \delta ||x||.
\end{gather}
\end{prop}

\begin{proof}
We have by Lemma \ref{lem:almost-proj} part C):
\begin{align}
||\pi_V^\perp(\pi_W(x))|| = ||\pi_W(x) - \pi_V(\pi_W(x))|| \leq c_3 \delta ||\pi_W(x)|| \leq c_3^2 \delta ||x||\, .
\end{align}
\end{proof}



Almost projections allow us to define a tractable notion of graph in a Banach space.  
\begin{definition}\label{def:graph}
Given an affine $k$-space $p + V$, and almost-projection $\pi_V$, we say a set $G$ is a graph over $(V, \pi_V)$ if there is a domain $\Omega \subset p + V$, and function $g : \Omega \to X$, so that
\begin{gather}
G = \{ x + g(x) : x \in \Omega\}, \quad \text{ and }\quad  \pi_V(g(x)) \equiv 0\, .
\end{gather}

For short we will often write $G = \graph_{\Omega, \pi_V}(g)$.
\end{definition}
\begin{remark}
Lemma \ref{lem:regraph} demonstrates that graphicality is ``well-defined,'' in the sense that whenever $G$ is a (small) graph with respect to some almost-projection $\pi$, then $G$ is a graph with respect to any other almost-projection (although with a slightly worse bound).
 
\end{remark}
\vspace{.25cm}

\subsection{Modulus of smoothness}\label{sec:smoothness}

The norm in a Banach space is evidently a Lipschitz function, but in general nothing more can be said. For example in $L^\infty[0,1]$ it is easy to see that the sup norm is not $C^1$. 

The modulus of smoothness of a Banach space $(X,\norm\cdot)$ measures in a quantitative way how smooth the norm of this space is. Here we briefly recall its definition and main properties, for more on this topic we refer the reader to some standard reference for Banach spaces (see \cite{LT1,LT2}), and to some specific important articles related to this subject (see \cite{clarkson,alber_gen_proj,alber_pyth,hanner}). Needless to say, since this topic has been extensively studied in literature, these references are not exhaustive. Moreover, this notion is intimately related via duality to the perhaps more standard notion of modulus of convexity.

\begin{definition}\label{def:modulus-of-smoothness}
 Given a Banach space $X$, we set $\rho_X:[0,\infty)\to [0,\infty)$ to be its modulus of smoothness, defined by
 \begin{gather}
  \rho_X(t)= \sup_{\norm x =1 \,  \ \ \norm y = t} \ton{\frac{\norm{x+y}+\norm{x-y}}{2}} - 1\, .
 \end{gather}
We say that $X$ is uniformly smooth if $\lim_{t\to 0}t^{-1} \rho_X(t)=0$, and we say that is of smoothness power-type $\alpha\in [1,2]$ if
\begin{gather}\label{eq_deph_power}
 \limsup_{t\to 0}t^{-\alpha} \rho_X(t)<\infty\, .
\end{gather}
\end{definition}

An easy consequence of the convexity of $\norm \cdot$ is that $\rho_X$ is a convex function.  
\begin{remark}
Note that $0\leq \rho(t)\leq t$ by the triangle inequality, and that for any Hilbert space $H$, $\rho_H(t)=\sqrt{1+t^2}-1$.  In fact Hilbert spaces are the ``smoothest'' possible Banach spaces, in the sense that for any Banach space $X$, we have $\rho_X(t)\geq \rho_H(t)$ (see \cite{LT2,nordlander}).
\end{remark}
\begin{example}
 As a first example, we recall that when $X =L ^p$, we have
\begin{gather}\label{eq_rho_lp}
\rho_{L^p}(t) \leq \left\{ \begin{array}{l l}  p^{-1} t^p+o(t^p) & (1 < p \leq 2) \\ (p-1)t^2 +o(t^2)& (2 < p < \infty) \end{array} \right. \, .
\end{gather}
This follows from Hanner's inequality (see \cite{hanner}, \cite[pag 63]{LT2}).
\end{example}


If $X$ is smooth, then its norm is continuously differentiable away from the origin. In such a space, the gradient of $\norm x^2/2$ is equal to the functional $J(x)$, where $J$ is the normalized duality mapping between $X$ and its dual $X^*$. Since this mapping is going to play an important role in the following, we recall its definition and some of its properties here.
\begin{definition}
 Given any Banach space $X$, let $X^*$ be its dual. A normalized duality mapping $J:X\to X^*$ is a mapping satisfying
 \begin{gather}
  \norm{J(x)}_{X^*}= \norm{x}_{X}\, , \quad \ps{J(x)}{x} = \norm{x}^2\, ,
 \end{gather}
where $\ps{\phi}{x}=\phi(x)$ is the natural pairing between a functional $\phi\in X^*$ and an element $x\in X$.
\end{definition}
\begin{example}\label{ex_J_p}
 The easiest example of mapping $J$ is given in a Hilbert space $(H,\ps{\cdot}{\cdot}_H)$, where the Riesz representation theorem states that $J(x)=\ps{x}{\cdot}_H$ is a normalized duality mapping, and actually it is the unique map with these properties.
 
 For the reader's convenience, we also recall what the mapping $J$ is in the real Banach spaces $l_p$. For $p\in (1,\infty)$, there exists a unique $J$ determined by
 \begin{gather}\label{eq_Jp}
  J\ton{\cur{x_i}_{i=1}^\infty} = \ton {\sum_{i=1}^\infty \abs{x_i}^p }^{\frac{2-p}{p}} \cur{\abs{x_i}^{p-2} x_i }_{i=1}^\infty   \ \ \in l_p^* = l_q\, .
 \end{gather}
On $l_1$, we can write 
\begin{gather}
 J\ton{\cur{x_i}_{i=1}^\infty} = \ton {\sum_{i=1}^\infty \abs{x_i} } \cur{\operatorname{sign}(x_i)}_{i=1}^\infty   \ \ \in l_1^* = l_\infty\, ,
\end{gather}
where $\operatorname{sign}(x)$ is the sign function for $x\in \R \setminus \{0\}$, and it can be any number in $[-1,1]$ if $x=0$ (thus $J$ is not uniquely determined on $l_1$).  
\end{example}

The most important property of $J$ for us is the following effective continuity:
\begin{lemma}[{\cite[equation 7.7]{alber_gen_proj}}]\label{lemma_pi}
If $X$ is a uniformly smooth Banach space, then
\begin{gather}
||J(x) - J(y)||_{X^*} \leq 8 R \frac{\rho_X(4 ||x - y||/R)}{4 ||x - y||/R},
\end{gather}
where $R = \sqrt{ (||x||^2 + ||y||^2)/2}$.
\end{lemma}

As a direct Corollary of Lemma \ref{lemma_pi}, and the definition of $J$, we obtain the following Pythagorean-type theorems (similar to \cite[theorem 7.5]{alber_gen_proj}, \cite[theorem 2.11]{alber_pyth})
\begin{lemma}\label{lemma_pyth}
Let $X$ be a uniformly smooth Banach space, then
\begin{gather}\label{eqn:ip-type}
\Big| || x + y||^2 - ||x||^2 \Big| \leq 2 |\ps{Jx}{y}| + 4 (||x||^2 + ||y||^2) \rho_X \left( \frac{4||y||}{\sqrt{||x||^2 + ||y||^2}} \right).
\end{gather}

In particular, we mark two special cases.  Let $V^k$ be a $k$-dimensional space in $X$.  If $\pi(x)$ is an almost-projection to $V$, then for every $x$:
\begin{gather}\label{eqn:ip-type-pi}
\Big| ||x||^2 - ||\pi(x)||^2 \Big| \leq 2 |\ps{J \pi(x)}{\pi^\perp(x)}| + 8 c_3(k)^2 ||x||^2 \rho_X(||\pi^\perp(x)||/||x||).
\end{gather}
If $f : V \to X$ is a Lipschitz mapping, with $\lip(f) \leq \eps \leq 1$, then
\begin{gather}\label{eqn:ip-type-f}
\Big| ||(x + f(x)) - (y + f(y))||^2 - ||x - y||^2 \Big| \leq |\ps{J(x - y)}{f(x) - f(y)}| + 8 \rho_X(4\eps) ||x - y||^2 ,
\end{gather}
for every $x, y \in V$.
\end{lemma}

\begin{proof}
Let $\gamma(t) = x + ty$.  Then we compute
\begin{align}
\Big| ||x + y||^2 - ||x||^2 \Big| 
&= \left| \int_0^1 2 \ps{J \gamma(t)}{ \gamma'(t)} dt \right| \\
&\leq 2|\ps{Jx}{y}| +  \int_0^1 2|\ps{J(x + ty) - Jx}{ y}| dt.
\end{align}

If we define
\begin{gather}
R(t) = \sqrt{(||x + ty||^2 + ||x||^2)/2} \leq \sqrt{ ||x||^2 + ||y||^2},
\end{gather}
then using Lemma \ref{lemma_pi} and the convexity of $\rho_X$, we bound
\begin{align}
2|\ps{J(x + ty) - Jx}{ y}| 
&\leq 2 \cdot 8 R(t) \frac{\rho(4 ||y||/R(t))}{4||y||/R(t)} ||y|| \\
&= 4 R(t)^2 \rho(4 ||y||/R(t) ) \\
&\leq 4 R(t) \sqrt{||x||^2 + ||y||^2} \rho(4||y||/\sqrt{||x||^2 + ||y||^2}) \\
&\leq 4 (||x||^2 + ||y||^2) \rho(4 ||y||/\sqrt{||x||^2 + ||y||^2}).
\end{align}
This establishes \eqref{eqn:ip-type}.

To prove \eqref{eqn:ip-type-pi} replace $x$ with $\pi(x)$ and $y$ with $\pi^\perp(x)$ in \eqref{eqn:ip-type}, and use the bound $||\pi|| \leq c_3(k)$.  To prove \eqref{eqn:ip-type-f}, replace $x$ with $x + f(x)$, and $y$ with $y + f(y)$.
\end{proof}

\begin{remark}
Notice that if $\ps{J|_V}{\pi^\perp} \equiv 0$ (so that $\pi$ is an ``orthogonal'' projection), then \eqref{eqn:ip-type-pi} becomes
\begin{gather}
\Big| ||x||^2 - ||\pi(x)||^2 \Big| \leq c(k) ||x||^2 \rho_X(||\pi^\perp(x)||/||x||).
\end{gather}
\end{remark}

\vspace{.25cm}

\subsection{Canonical projections}

In certain cases we have a canonical notion of projection, which admits better bounds than a generic almost-projection.

\subsubsection{Hilbert spaces}

If $X$ is a Hilbert space, and $V$ is a $k$-plane, then $V$ admits a unique orthogonal projection $\pi_V : X \to V$, with the property that
\begin{gather}\label{eqn:classic-pyth}
||\pi_V(x)||^2 + ||x - \pi_V(x)||^2 = ||x||^2\, .
\end{gather}
Correspondingly, in a Hilbert space we have a canonical notion of orthogonal complement $V^\perp = \ker(\pi_V)$, for which $\pi_V^\perp \equiv \pi_{V^\perp}$, in the notation of Definition \ref{def:almost-proj}.  Moreover, from the Pythagorean relation \eqref{eqn:classic-pyth}, 
\begin{gather}
d(x, V) = ||\pi_{V^\perp}(x)|| \equiv ||\pi_V^\perp(x)||\, .
\end{gather}
Finally, let us remark that trivially, the orthogonal projection is an almost-projection.

\vspace{5mm}
The fact that projections in Hilbert spaces are canonical allow us to give a different definition of distance between subspaces. In particular, given $V,W$ to linear subspaces in $H$, we could define a distance between $V$ and $W$ by taking the operator norm $\norm{\pi_V - \pi_W}$. As it is not difficult to see, this notion is equivalent to $d_G(V,W)$. Here we recall a standard lemma needed to show this equivalence, that will be stated in more generality later on in Lemma \ref{lem:operator-diff}.
\begin{lemma}\label{lem:d-perp}
Let $V, W$ be linear subspaces of a Hilbert space.  Then $d_G(V, W)=d_G(V^\perp, W^\perp)$. 
\end{lemma}

\begin{proof}
By symmetry, it is sufficient to prove that $d_G\ton{V^\perp,W^\perp}\leq d_G\ton{V,W}$. 

Take $x\in V^\perp$ such that $\norm x=1$, and consider that $d(x,W^\perp\cap \B 1 0)=\norm{\pi_W(x)}$. Let $z=\pi_W(x)$ and $y=\pi_V(z)$. We want to show that if $d_G(V,W)\leq \epsilon<1$, then $\norm{z}\leq \epsilon$. We can limit our study to the space spanned by $x,y,z$, and assume WLOG that $x=(1,0,0)$, $y=(0,b,0)$ and $z=(a,b,c)$. By orthogonality between $z$ and $z-x$, we have
\begin{gather}
 a^2+b^2+c^2 +(1-a)^2+b^2+c^2=1 \, \quad \implies \quad a=a^2+b^2+c^2 \, ,
\end{gather}
and since $z\in W$, we also have $\norm{z-y} \leq \epsilon \norm{z}$, which implies
\begin{gather}
 a^2+c^2\leq \epsilon^2 \ton{a^2+b^2+c^2} \, \quad \Longrightarrow \quad a^2+c^2 \leq \frac{\epsilon^2}{1-\epsilon^2} b^2\, .
\end{gather}
Since the function $f(x)=x^2/(1-x^2)$ is monotone increasing for $x\geq 0$, we can define $\alpha\geq 0$ in such a way that
\begin{gather}
 a^2+c^2 = \frac{\alpha^2}{1-\alpha^2} b^2\, , \quad a=a^2+b^2+c^2 = \frac{1}{1-\alpha^2} b^2\, .
\end{gather}
Note that necessarily we will have $\alpha\leq \epsilon$. Now we have
\begin{align}
 &\frac{1}{(1-\alpha^2)^2} b^4 =a^2\leq \frac{\alpha^2}{1-\alpha^2} b^2 \\
 &\implies b^2\leq \alpha^2\ton{1-\alpha^2} \\
 &\implies \norm{z}^2 = a^2+b^2+c^2\leq \alpha^2\leq \epsilon^2\, .
\end{align}
This proves that $V^\perp\cap\B 1 0 \subset \B{\epsilon}{W^\perp\cap \B 1 0}$. In a similar way, one proves the opposite direction.
\end{proof}

With this easy lemma, we can show as promised that
\begin{lemma}\label{lemma_piVSdG}
Let $V,W$ be linear subspaces of a Hilbert space $H$. Then for every $x\in H$, 
 \begin{gather}\label{eq_piH1}
  \norm{\pi_V(x) - \pi_W(x)}\leq d_G(V,W) ||x||\, .
 \end{gather}
In the converse direction we have
\begin{gather}\label{eq_piH2}
 d_G(V,W)\leq \sup_{\norm x=1 } \cur{\norm{\pi_V(x)-\pi_W(x)}}\, .
\end{gather}
\end{lemma}
 \begin{proof}
Let $x$ be such that $\norm{x}=1$, and set $x=\pi_V(x)+\pi_{V^\perp}(x):=y+z$. Then
\begin{gather}
 \norm{\pi_V(x)-\pi_W(x)}^2=\norm{y-\pi_W(y)-\pi_W(z)}^2 = \norm{y-\pi_W(y)}^2+\norm{z-\pi_{W^\perp}(z)}^2 = d(y,W)^2+d(z,W^\perp)^2\, .
\end{gather}
Since $y\in V$, then $d(y,W)\leq \norm y d_G(V,W)$, and similarly $d(z,W^\perp)=\norm z d(V^\perp,W^\perp)$. Since $\norm y^2+\norm z^2=1$, by the previous lemma we get
\begin{gather}
 \norm{\pi_V(x)-\pi_W(x)}^2\leq \norm{y}^2 d_G(V,W)^2+\norm{z}^2 d_G(V^\perp,W^\perp)^2=d_G(V,W)^2\, .
\end{gather}
This proves \eqref{eq_piH1}. \eqref{eq_piH2} is an easy consequence of the definition of $d_G(V,W)$. 
\end{proof}

\vspace{5mm}

\subsubsection{Curves in smooth Banach spaces}

If $X$ is uniformly smooth, then the normalized duality mapping between $X$ and $X^*$ provides us with a canonical (norm one) projection onto one dimensional subspaces, as described in the next Definition. Moreover, thanks to the results \cite[theorem 7.5]{alber_gen_proj}, \cite[theorem 2.11]{alber_pyth} we have a generalized Pythagorean theorem in uniformly smooth Banach spaces that is going to be crucial for the power gain in the Reifenberg theorem.
\begin{definition}\label{deph_Jproj}
 Given a $1$-dimensional subspace $V$ of $X$, spanned by the unit vector $v$, we call the map $\pi_V : X \to V$ defined by $\pi_V(x) = \ps{J(v)}{x} v$ the \emph{J-projection}, or canonical projection, onto $V$.
\end{definition}
Of course any $J$-projection is trivially an almost-projection, and it is easy to see that in a Hilbert space this coincides with the orthogonal projection onto $V$. Moreover, it is easy to see that this almost projection has operator norm $1$, since $\abs{\ps{J(v)}{x}}\leq \norm x$ for all $x$ and $\ps{J(v)}{v}=1$.

\vspace{5mm}

\subsubsection{Summary}

Let us summarize the two key properties we need of orthogonal and J-projections.
\begin{lemma}\label{lem:operator-diff}
Let $V, W$ be two $k$-spaces in $X$, with associated almost-projections $\pi_V, \pi_W$.  Suppose \emph{either} $X$ is Hilbert, and $\pi_V, \pi_W$ are orthogonal; \emph{or} $X$ is uniformly smooth, $k = 1$, and $\pi_V$, $\pi_W$ are J-projections.

If $d_G(V, W) < \delta$, then
\begin{gather}
||\pi_V - \pi_W|| \leq 2\rho_X(4\delta)/\delta\, .
\end{gather}
\end{lemma}

\begin{proof}
If $X$ is Hilbert, this is a corollary of Lemma \ref{lemma_piVSdG}.  Suppose now that $X$ is a uniformly smooth Banach space, $k = 1$, and that $\pi_V, \pi_W$ are J-projections.  We can choose unit vectors $v, w$ spanning $V, W$, with $||v - w|| < \delta$, and then $\pi_V(x) = \ps{J(v)}{x}v$, and $\pi_W(x) = \ps{J(w)}{x}w$.  We estimate therefore that
\begin{gather}
||\pi_V(x) - \pi_W(x)|| \leq (||J(v) - J(w)||+ ||v - w||) ||x|| \leq (2\rho_X(4\delta)/\delta + \delta) ||x||\, .
\end{gather}
In the last inequality we also used the convexity of $\rho_X(t)$.
\end{proof}

\begin{lemma}\label{lem:pythag}
Take $V$ a $k$-plane in $X$.  If \emph{either} $X$ is Hilbert, and $\pi_V$ is the orthogonal projection, \emph{or} $X$ is uniformly smooth, $k = 1$, and $\pi_V$ is the J-projection, then we have the following improvements on \eqref{eqn:ip-type-pi}, \eqref{eqn:ip-type-f}: for any $x$, 
\begin{gather}\label{eqn:pythag}
\Big| ||x||^2 - ||\pi(x)||^2 \Big| \leq 8 ||x||^2 \rho_X(||\pi^\perp(x)||/||x||).
\end{gather}
If $f : V \to X$ is a Lipschitz mapping, with $\lip(f) \leq \eps \leq 1$, then for every $x, y \in V$, 
\begin{gather}\label{eqn:pythag-f}
\Big| ||(x + f(x)) - (y + f(y))||^2 - ||x - y||^2 \Big| \leq 2||x - y||||\pi(f(x) - f(y)|| + 8 \rho_X(4\eps)||x - y||^2 .
\end{gather}
\end{lemma}

Of course these estimates are far from sharp when $X$ is Hilbert.

\begin{proof}
By Lemma \ref{lemma_pyth}, it suffices to show that $<J|_V, \pi_V^\perp> = 0$.  When $X$ is Hilbert, this follows immediately from the fact $J(x)=\ps{x}{\cdot}_H$.  When $X$ is uniformly smooth, and $k = 1$, we can verify: given unit vector $v$ spanning $V$, then
\begin{gather}
\ps{J(v)}{ \pi_V^\perp(x)} = \ps{J(v)}{ x - \ps{J(v)}{ x} v} = 0.
\end{gather}
\end{proof}

Improved orthogonality estimates like \eqref{eqn:pythag} give improved Lipschitz bounds on graph projections, which at a very basic level is why we can expect improved estimates on the Reifenberg maps.
\begin{prop}\label{prop:graph-est}
Take $V$ a $k$-plane in $X$.  Suppose \emph{either} $X$ is Hilbert, and $\pi_V$ is the orthogonal projection, \emph{or} $X$ is uniformly smooth, $k = 1$, and $\pi_V$ is the J-projection.  Let
\begin{gather}
G = \graph_{\Omega, \pi_V}(g), \quad \lip(g) \leq \eps \leq 1, \quad \Omega \subset V\, .
\end{gather}

Then we have the estimate
\begin{gather}
\Big| ||(x + g(x)) - (y + g(y))||^2 - ||x - y||^2 \Big| \leq 8 \rho_X(4 \eps) ||x - y||^2 \quad \forall x, y \in \Omega\, .
\end{gather}
In particular, $\pi_V : G \to V$ is a bi-Lipschitz equivalence, with Lipschitz constant bounded by $1 + 8\rho_X(4\eps)$.
\end{prop}

\begin{proof}
Immediate from Lemma \ref{lem:pythag} and the definition of graph.
\end{proof}

\subsection{Tilting control}\label{sec_tilting}

We study the tilting between best planes at different scales, and try to control the tilting using the $\beta$ numbers. 

First of all, we give a definition of ``approximate best subspace'' for the measure $\mu$ on any ball in $X$. 
\begin{definition}\label{deph_V}
 Given a finite measure $\mu$ in a Banach space $X$, and given a ball $\B r x$, we set $p(x, r) + V(x, r)$ to be an affine $k$-dimensional subspace (with $p(x, r) \in B_r(x)$) such that
 \begin{gather}
  r^{-k-2}\int_{\B r x} d(y, p(x, r) + V(x, r))^2 d\mu(y)\leq 2 \beta_{\mu}^k (x,r)^2\, .
 \end{gather}
\end{definition}
The definition if obviously well-posed if $\beta(x,r)>0$. For $\beta=0$, we have the following easy lemma.
\begin{lemma}\label{lem:beta-zero}
 Let $\beta^k_{\mu}(x,r)=0$, then there exist a $k$-dimensional affine subspace $p + V$ such that the $\mu(B_r(x) \setminus (p + V)) = 0$.
\end{lemma}
\begin{remark}
 Note that we don't claim simply that the support of the measure $\mu$ is contained in $p+V$. Although this is equivalent to our claim when $X$ is separable (and thus it has a countable base for the topology), our claim is a priori stronger in general Banach spaces. 
\end{remark}

\begin{proof}

In infinite dimensional Banach spaces, we don't have compactness for the Grassmannian of $k$-dimensional affine subspaces, thus we need a different argument. For convenience, we assume that $x=0$ and $r=1$ and that $\mu(X\setminus \B 1 0)=0$ (otherwise we replace $\mu$ with $\mu\llcorner \B 1 0$). Consider for all $i\in \N$ a sequence of affine subspaces $p_i+V_i$ such that
\begin{gather}
 \int_{\B 1 0} d(y, p_i+V_i)^2 d\mu(y)\leq 3^{-i}\, ,
\end{gather}
so that by Chebyshev inequality
\begin{gather}
\mu(X \setminus B_{i^{-1}}(p_i + V_i)) \leq 2^{-i}\, ,
\end{gather}
Thus we get that for all $j$:
\begin{gather}
 \mu\ton{X \setminus \bigcap_{j\geq i} B_{j^{-1}}(p_j+V_j) } \leq 2^{-i+1}\, ,
\end{gather}
and in turn
\begin{gather}
 \mu\ton{X \setminus \bigcup_i \bigcap_{j \geq i} B_{j^{-1}}(p_j+V_j) }=0\, ,
\end{gather}

We claim that there is an affine $k$-space $q + W$ so that
\begin{gather}
\bigcap_{j \geq i} B_{j^{-1}}(p_j + V_j) \subset q + W \quad \forall i\, .
\end{gather}
This would clearly finish the proof. 

Now obviously $\bigcap_{j\geq i} B_{j^{-1}}(p_j+V_j)$ is a convex set.  Take $x_0\in \bigcap_{j\geq i} B_{j^{-1}}(p_j+V_j)$ be any point (if no such $x_0$ exists then we have nothing to prove), and assume by contradiction that there exist $x_1,\cdots,x_{k+1}\in \bigcap_{j\geq i} B_{j^{-1}}(p_j+V_j)$ such that $\cur{x_1-x_0,\cdots,x_{k+1}-x_0}$ are linearly independent. Fix a $\tau>0$ so these points are in $\tau$-general position.

By Lemma \ref{lemma_GP_stable}, there exists $j$ sufficiently large such that if $\cur{y_i}_{i=0}^{k+1}$ are such that $||y_i-x_i ||\leq j^{-1}$, then $\{y_i-y_0\}_{i=1}^{k+1}$ lie in $\tau/2$ general position, and in particular are linearly independent.  Thus we can find a $k+1$ dimensional affine subspace that is contained in the $k$-dimensional affine subspace $p_j+V_j$, for $j$ sufficiently large, and we reach our contradiction.
\end{proof}
%

\vspace{5mm}

Now that we have a definition for $V(0,1)$, we turn to the tilting control. The idea is the following: given two balls one containing the other, say for example $\B 1 0$ and $\B {1/10}{0}$, we want to be able to say that $\beta(0,1)$ controls the distance between $V(0,1)$ and $V(0,1/10)$. The following example shows that in general this is not possible. 
\begin{example}
 Let $k=1$ and $\mu$ be the sum of $5$ Dirac masses in the Euclidean $\R^2$
 \begin{gather}
  \mu= \delta_0 + \delta_{(1,0)}+\delta_{(-1,0)}+\delta_{(0,t)}+\delta_{(0,-t)}\, .
 \end{gather}
For $0<t\leq 1/10$, it is easy to see that $V(0,1)$ is the $x$-axis, while $V(0,1/10)$ is the $y$-axis, and this is independent on the choice of $t$.

Moreover, we have
\begin{gather}
 \beta_\mu(0,1/10)^2=0\, , \quad \beta_\mu(0,1)^2= 2t^2\, . 
\end{gather}
As $t$ approaches $0$, the beta numbers clearly don't control the distance between $V(0,1)$ and $V(0,1/10)$ (which is constant in $t$ and equal to $1$).  So the geometry of the measure $\mu$ is essential to obtain the bound we want.
\end{example}

We will see in the following that we have ``tilting control'' as long as $\mu$ is sufficiently spread over something $k$-dimensional on the small ball. In order to be more precise, we give the following definition of ``good balls''.
\begin{definition}\label{deph_bad_balls}
Take $\mu$ a finite Borel-regular measure, and $\chi \in (0, 1/10)$.  We say a ball $B_r(x)$ is a \emph{good ball} w.r.t the measure $\mu$ and parameter $\chi$ if for any affine subspace $q + L$ of dimension $\leq k-1$, there exists a point $z$ such that
\begin{enumerate}
\item[i)] $\mu(B_{\chi r}(z) \cap B_r(x)) \geq 10^{-1} c_2^{-1} (\chi r)^k$, and
\item[ii)] $z\not \in \B {7\chi r}{q + L}$
\end{enumerate}
Here $c_2(k)$ is the constant from Lemma \ref{lemma_Banach_packing}.  If $k = 0$ then good is simply the requirement that some $z$ exists satisfying i).  If $B_r(x)$ is not good, we say $B_r(x)$ is a \emph{bad ball} w.r.t. $\mu$ and $\chi$.
\end{definition}

The next lemma shows that on good balls we have good tilting control.

\begin{lemma}\label{lem:good-ball-tilting}
Let $\mu$ be a finite Borel-regular measure, and consider $B_{r}(x) \subset B_{1/2}(0)$.  If $B_r(x)$ is a good ball w.r.t. $\mu$ and $\chi$, then we have
\begin{gather}
d_H( \left[ p(x, r) + V(x, r) \right] \cap B_1(0), \left[ p(0, 1) + V(0, 1) \right] \cap B_1(0)) \leq c(k, r, \chi) \beta^k_\mu(0, 1)\,  ,
\end{gather}
and in particular
\begin{gather}
d_G(V(x, r), V(0, 1)) \leq c(k, r, \chi) \beta^k_\mu(0, 1)\, .
\end{gather}
\end{lemma}

An immediate corollary is the following comparability between \emph{any} two good balls.
\begin{lemma}\label{lem:good-ball-tilting-2}
Suppose $B_{r'}(x')$ and $B_r(x)$ are good balls w.r.t. $\mu$ and $\chi$.  If we have $B_{r'}(x) \cup B_r(x) \subset B_{R/2}(y)$, then
\begin{gather}
d_H ( \left[ p(x, r) + V(x, r) \right] \cap B_R(y), \left[ p(x', r') + V(x', r') \right] \cap B_R(y)) \leq c(k, r/R, r'/R, \chi) \beta(y, R) R\, ,
\end{gather}
and
\begin{gather}
d_G( V(x, r), V(x', r')) \leq c(k, r'/R, r/R, \chi) \beta(y, R).
\end{gather}
\end{lemma}

\begin{proof}[Proof of Lemma \ref{lem:good-ball-tilting}]
We assume $y=0$ and $R=1$ for simplicity. By enlarging $c$ as necessary we can also assume wlog that
\begin{gather}
\beta(0, 1) \leq \delta_0(k, r,\chi)\, .
\end{gather}
In the following $c$ denote a generic constant depending only on $k, r, \chi$.

We claim we can inductively find points $\hat x_0, \ldots, \hat x_k \in B_r(x)$ such that
\begin{enumerate}
 \item the vectors $\cur{\hat x_i-\hat x_0}_{i=1}^k$ are in $5\chi r$-general position
 \item we have the estimates
 \begin{gather}
 d( \hat x_i, p(x, r) + V(x, r))^2 \leq c \beta(0, 1)^2, \quad d(\hat x_i, p(0, 1) + V(0, 1))^2 \leq  c \beta(0, 1)^2\, .
\end{gather}
\end{enumerate}

Let us see how this claim completes the proof.  Choose $y_i$ in $p(x,r)+V(x,r)$ with $\norm{\hat x_i-y_i}\leq c \beta(0, 1)$.  By the triangle inequality, $d(y_i,p(0,1)+V(0,1))\leq c \beta(0, 1)$ as well.

Provided $\delta_0(k, r,\chi)$ is sufficiently small, by Lemma \ref{lemma_GP_stable} the vectors $\cur{y_i-y_0}_{i=1}^k$ lie in $3\chi r$-general position.  Now given any $y \in (p(x, r) + V(x, r)) \cap B_1(0)$, we write by Lemma \ref{lemma_GP_bounds} 
\begin{gather}
y = y_0 + \sum_{i=1}^k \alpha_i (y_i - y_0), \quad ||\alpha_i|| \leq c\, ,
\end{gather}
and thereby deduce
\begin{gather}
d(y, p(0, 1) + V(0, 1)) \leq c \beta(0, 1)\, .
\end{gather}
The proof of Lemma \ref{lem:good-ball-tilting} is completed by an application of Lemma \ref{lemma_hdv}.

We are left to prove the inductive claim.  To construct our base case $\hat x_0$, in the following let us set $j = -1$ and interpret $q + L_{-1} = \emptyset$.  Otherwise, suppose by induction that we have a collection $\cur{\hat x_i}_{i=0}^j$ with the desired properties for some $j\leq k-1$, and let $q + L_j$ be the $j$ dimensional affine subspace given by
\begin{gather}
 q + L_j = \hat x_0 + \operatorname{span}\{\hat x_1-\hat x_0,\cdots, \hat x_j-\hat x_0\}\, .
\end{gather}

By assumption, there exists a point $x_{j+1} \not \in B_{7\chi r}(q + L_j)$ such that
\begin{gather}\label{eqn:good-lower-mass}
\mu(B_{\chi r}(x_{j+1})\cap B_r(x) ) \geq 10^{-1} c_2^{-1} (\chi r)^k \, .
\end{gather}
Set for simplicity $\bar \mu=\mu\llcorner (B_{\chi r}(x_{j+1})\cap B_r(x))$, and define for $z\in X$ and $s>0$ the set
\begin{gather}\label{eq_cheby2}
 Q_{z,s}= \cur{y\in X \ \ s.t. \ \ d(y,p(z,s)+V(z,s)))^2 \leq 3 \fint d(w,p(z,s)+V(z,s))^2 d\bar \mu(w)}\, .
\end{gather}
By Chebyshev inequality, we have trivially that
\begin{gather}
 \bar \mu(Q_{x,r})\geq \frac 2 3 \bar \mu(X)\, , \quad \bar \mu(Q_{0,1})\geq \frac 2 3 \bar \mu(X)\, .
\end{gather}
Thus there exists a point $\hat x_{j+1}\in \spt\,  \bar\mu \cap Q_{x,r}\cap Q_{0,1}$. Since $\spt\, \bar\mu \subseteq B_{\chi r}(x_{j+1})$, by the triangle inequality we get $\hat x_{j+1}\not \in B_{5\chi r}(q + L_j)$. Moreover, we have by \eqref{eqn:good-lower-mass} and the inclusion $B_{r}(x) \subset B_1(0)$ that
\begin{align}
&d( \hat x_{j+1}, p(x, r) + V(x, r))^2 \leq 3\mu(B_r(y) \cap B_{\chi r}(x_{j+1}))^{-1} r^{k+2}\beta(x, r)^2 \leq c \beta(0, 1)^2\, , \\
&d(\hat x_{j+1}, p(0, 1) + V(0, 1))^2 \leq  3\mu(B_r(y) \cap B_{\chi r}(x_{j+1}))^{-1} \beta(0, 1)^2\leq c \beta(0, 1)^2\, .
\end{align}
This complete the proof of the inductive claim, and in turn the proof of the lemma.
\end{proof}

\vspace{1cm}

\section{Reifenberg estimates in Banach spaces}\label{sec_sigma}

Our fundamental tool is the Reifenberg map $\sigma$, which is essentially an interpolation of projection mappings.  We shall use the Reifenberg maps to construct approximating manifolds by ``gluing'' together nearby planes.  This section establishes important basic estimates on these maps.  We are \emph{not} defining the actual Reifenberg maps we use at this stage; the estimates require only the basic structure.

In this section we shall suppose we have a fixed $k$-plane $V$, with almost-projection $\pi$ as in Definition \ref{def:almost-proj}, and a point $p$. 
Let $x_i$ be a $2r/5$-separated set in $X$, and take $p_i \in B_r(x_i)$ with $k$-planes $V_i$ and associated almost-projections $\pi_i$.

Assume the following tilting and closeness control:
\begin{gather}\label{eq_sigma_tilt}
d(0, p + V) < r/10, \quad d(x_i,  p + V) < r/10, \quad d(p_i, p + V) < \delta r, \quad d_G(V_i, V) < \delta\, .
\end{gather}

\subsection{The map \texorpdfstring{$\sigma$}{sigma}}

Suppose $\sigma : B_{3r} \to X$ is a mapping of the form
\begin{gather}\label{eq_deph_sigma}
\sigma(x) = x - \sum_i \phi_i(x) \pi^\perp_i(x - p_i)\, ,
\end{gather}
where $\phi_i$ is the truncated partition of unity subordinate to the $B_{3r}(x_i)$, as per Lemma \ref{lem:pou}.  Notice that, by Lemma \ref{lemma_Banach_packing} and our hypothesis on $x_i$, the overlap of the $\{B_{3r}(x_i)\}_i$ is bounded by some uniform constant $c(k)$.  So in particular the $\phi_i$ satisfy:
\begin{gather}
0 \leq \phi_i \leq 1, \quad \spt\, \phi_i \subset B_{3r}(x_i), \quad \lip(\phi_i) \leq c(k)/r\, .
\end{gather}

We are ready to state and prove the main lemma (the ``Banach squash lemma'') regarding the properties of the map $\sigma$.  This Lemma proves that, provided a set $G$ is reasonably well-behaved to start with (i.e. is a graph with small Lipschitz norm), then $\sigma|_G$ has good Lipschitz bounds (parts A, D), and the image $\sigma(G)$ has good graphical properties (parts B, C).

\vspace{5mm}

There are two subtle points.  First, where $\sum_i \phi_i = 1$ the map $\sigma$ is entirely an interpolation of affine projections, and in these regions the resulting graph geometry of $\sigma(G)$ depends only on the geometry of affine the planes $p + V$, $p_i + V_i$ (and not on $G$!).  Here is a baby example for illustration: take $X = \R^n$, the planes $p_i + V_i$ to be a single $p_1 + V_1$, and for simplicity set $\phi_1 \equiv 1$.  Then $\sigma$ becomes to the affine projection onto $p_1 + V_1$, and $\sigma(G) = p_1 + V_1$ for any graph $G$ over $p+ V$.  In general, in part C) we show that wherever $\sum_i \phi_i=1$, $\sigma(G)$ has graphical bounds independent of $G$.

Outside the region where $\sum_i \phi_i=1$, the map $\sigma$ starts to ``remember'' the geometry of $G$.  For example, in the extreme, when $\sum_i \phi_i = 0$, the $\sigma$ is simply the identity, and $\sigma(G) = G$ in there regime.  In part B) we show the graphical bounds on $\sigma(G)$ will generally depend both on bounds for $G$, and the tilting between the various planes $p+V$, $p_i + V_i$.

\vspace{5mm}

Second, when we have some reasonable notion of orthogonality (e.g. when $X$ is Hilbert, or $k = 1$ and $X$ is smooth), we get improved estimates on $\sigma|_G$.  This is because $\sigma$ is pushes $G$ ``almost orthogonally'' to $G$'s plane of graphicality.  Part D of Lemma \ref{lem:squash} shows a power gain in the ``tangential'' movement and Lipschitz bounds of $\sigma$.

Various forms of this lemma are present in literature, for example in \cite{toro:reifenberg,davidtoro,ENV}. Up to technical details, the proof of this lemma is standard. However, since this lemma is crucial for our estimates and we are going to use special properties of the $J$-projections on Banach spaces, we write a complete proof of this lemma. 

\begin{lemma}[Banach Squash Lemma]\label{lem:squash}
There are constants $\eps_1(k)$, $c_4(k)$ so that the following holds.  In the notation above, and with the assumptions \eqref{eq_sigma_tilt}, let $G$ be a closed set so that
\begin{gather}
G \cap B_{3r} = \graph_{\Omega, \pi_V}(g), \quad r^{-1}||g|| + \lip(g) \leq \eps, \quad B_{5r/2} \cap (p + V) \subset \Omega\, .
\end{gather}

Then provided $\delta + \eps \leq \eps_1(k)$, we have 
\begin{enumerate}
\item[A)] For $x, y \in G \cap B_{3r}$, $\sigma$ is a bi-Lipschitz equivalence between $G \cap B_{3r}$ and $\sigma(G \cap B_{3r})$, satisfying the estimates
\begin{gather}\label{eq_2lip_X}
r^{-1} ||\sigma(x) - x|| \leq c_4(\delta + \eps), \quad \text{and} \quad ||(\sigma(x) - \sigma(y)) - (x - y)|| \leq c_4(\delta + \eps) ||x - y||\, .
\end{gather}

\item[B)] We have
\begin{gather}
\sigma(G) \cap B_{2r} = \graph_{\tilde\Omega, \pi_V}(\tilde g), \quad r^{-1} ||\tilde g|| + \lip(\tilde g) \leq c_4 (\delta + \eps), \quad B_{3r/2} \cap (p + V) \subset \tilde \Omega, 
\end{gather}

\item[C)] If $\sum_i \phi_i = 1$ on $B_{(2 + c_4(\delta + \eps))r}$, then in part B) we in fact have the bound 
\begin{gather}\label{eqn:improved-squash-est}
r^{-1}||\tilde g|| + \lip(\tilde g) \leq c_4 \delta\, .
\end{gather}

\item[D)] Suppose either of the following scenarios: $X$ is a Hilbert space, and each $\pi$, $\pi_i$ are orthogonal; or $X$ is a uniformly smooth Banach space, $k = 1$, and each $\pi$, $\pi_i$ is a J-projection.  Then we have the improved estimates: for all $x, y \in G \cap B_3$, 
\begin{gather}\label{eq_2lip_+}
r^{-1} ||\pi(\sigma(x) - x)|| \leq c_4 \rho_X(c_4(\delta + \eps)), \quad \text{ and } \quad 
\Big| ||\sigma(x) - \sigma(y)||^2 - ||x - y||^2 \Big| \leq c_4\rho_X(c_4(\delta + \eps))  ||x - y||^2
\end{gather}
\end{enumerate}
\end{lemma}

\begin{proof}
In the following $c$ denotes a generic constant depending only on $k$.  We will assume $\eps_1(k)$ is chosen sufficiently small so that we always have $c \eps_1 \leq 1/100$.  By scaling we can assume $r = 1$. 

Given $x, y \in p + V$, for ease of notation we shall write $x^+ = x + g(x)$, and $y^+ = y + g(x)$.  For each $i$, choose $\tilde p_i \in V$ so that $||\tilde p_i - p_i|| < \delta$.  We can without loss of generality assume $||p|| < \delta$.

If $\phi_i(x^+) > 0$ and $x^+ \in B_3$, then $||x^+ - x_i|| < 3$, and therefore $||x_i|| < 6$, and $||p_i|| < 7$.  For such an $i$, we have
\begin{align}
||\pi^\perp_i(x + g(x) - p_i)||
&\leq ||\pi^\perp_i(x - \tilde p_i)|| + ||\pi_i^\perp(g(x))|| + ||\pi_i^\perp(p_i - \tilde p_i)||\\
&\leq c(k) \delta ||x - \tilde p_i|| + c(k) ||g(x)|| + c(k) \delta \\
&\leq c(k) (\delta + \eps).
\end{align}
Remember that $x - \tilde p_i \in V$.

Since $\#\{ i : x_i \in B_{6} \} \leq c(k)$ by Lemma \ref{lemma_Banach_packing}, we obtain
\begin{gather}
||\sigma(x^+) - x^+|| \leq \sum_i |\phi_i(x^+)| ||\pi^\perp_i(x^+ - p_i)|| \leq c(k) (\delta + \eps)\, .
\end{gather}

Similarly, we have
\begin{align}
||(\sigma(x^+) - x^+) - (\sigma(y^+) - y^+)||
&\leq \sum_i |\phi_i(x^+) - \phi_i(y^+)| ||\pi_i^\perp(x^+ - p_i)|| + \sum_i |\phi_i(y^+)| ||\pi^\perp_i(x^+ - y^+)||\, , 
\end{align}
where the first term on the right is bounded by
\begin{align}
\sum_i |\phi_i(x^+) - \phi_i(y^+)| ||\pi_i^\perp(x^+ - p_i)|| &\leq c (||x - y|| + ||g(x) - g(y)||) c(k) (\delta + \eps) \\
&\leq c(k) (\delta + \eps) (1 + \eps) ||x - y||\, ,
\end{align}
and the second term is bounded by
\begin{align}
\sum_i |\phi_i(y^+)| ||\pi^\perp_i(x^+ - y^+)|| 
&\leq c(k) ||\pi^\perp_i(x - y)|| + c(k) ||g(x) - g(y)|| \\
&\leq c(k) \delta ||x - y|| + c(k) \eps ||x - y||\, .
\end{align}
This proves part A).

In order to prove B), we write
\begin{align}
\sigma(x + g(x)) 
&= x + \pi(\sigma(x + g(x)) - x) + \pi^\perp(\sigma(x + g(x))) \\
&=: x + e(x) + \pi^\perp(\sigma(x + g(x))) \label{eqn:x-plus-g-is}
\end{align}
where we define $e : B_{5/2} \cap (p + V) \to V$ by
\begin{gather}
e(x) = \pi(\sigma(x + g(x)) - x) + \pi^\perp(p) \equiv \pi(\sigma(x^+) - x^+)
\end{gather}
Recall that $x^+ = x + g(x)$.  Moreover, since $\eps_1(k) < 1/10$ we have
\begin{gather}\label{eqn:x-plus-g-inside}
\{x + g(x) : x \in B_{5/2} \cap (p + V)\} \subset G \cap B_3\, .
\end{gather}

By part A) and \eqref{eqn:x-plus-g-inside} we have for any $x \in B_{5/2} \cap (p + V)$, 
\begin{gather}\label{eqn:est-e}
||e(x)|| \leq c(k)(\delta + \eps), \quad ||e(x) - e(y)|| \leq c(k)(\delta + \eps) ||x^+ - y^+|| \leq c(k)(\delta + \eps)||x - y||\, .
\end{gather}
Therefore, provided $\delta + \eps \leq \eps_1(k)$, we deduce the map
\begin{gather}
x \mapsto x + e(x) : B_{5/2} \cap (p + V) \to U\, , 
\end{gather}
is a bi-Lipschitz equivalence, with Lipschitz inverse
\begin{gather}\label{eqn:est-Q}
Q : U \to B_{5/2} \cap (p + V), \quad ||Q(x) - x|| \leq c(k)(\delta + \eps), \quad \lip(Q) \leq 2\, .
\end{gather}
Moreover, from our bounds \eqref{eqn:est-e} on $e$, we have $U \supset B_2(0) \cap (p + V)$ provided $\eps_1(k)$ is sufficiently small.

If we define
\begin{gather}
\tilde g(y) = \pi^\perp(\sigma(Q(y) + g(Q(y))))\, ,
\end{gather}
then from \eqref{eqn:x-plus-g-is} and the definition of $Q$ we have
\begin{gather}
\sigma(Q(y) + g(Q(y))) = y + \tilde g(y)\, .
\end{gather}
And so
\begin{gather}
\sigma(\{ x + g(x) : x \in B_{5/2} \cap V \}) = \graph_{U, \pi}(\tilde g), \quad U \supset B_2(0) \cap (p + V)\, .
\end{gather}

Since
\begin{gather}
\tilde g(y) = \pi^\perp \qua{ \sigma( Q(y) + g(Q(y))) - \left[ Q(y) + g(Q(y)) \right] } + g(Q(y))\, ,
\end{gather}
we have from part A) the bounds
\begin{gather}
||\tilde g(y)|| \leq c(k)\delta ||Q(y) + g(Q(y))|| + ||g(Q(y))|| + c(k)\delta \leq c(k) (\delta + \eps),
\end{gather}
and
\begin{align}
||\tilde g(y) - \tilde g(z)|| 
&\leq c(k)(\delta + \eps)|| (Q(y) - g(Q(y))) - (Q(z) - g(Q(z))) || + ||g(Q(y)) - g(Q(z))|| \\
&\leq c(k)(\delta + \eps) ||y - z||\, ,
\end{align}
for any $y, z \in U$.

To finish proving B), it remains to show that
\begin{gather}\label{eqn:reif-inclusions}
\sigma(\{x + g(x) : x \in B_{5/2} \cap V\}) \supset \sigma(G) \cap B_2(0) \supset \{ y + \tilde g(y) : y \in B_{3/2} \cap V \}\, .
\end{gather}

First, suppose $\sigma(x + g(x)) \in B_2(0)$.  Then
\begin{gather}
||x|| \leq ||(x + g(x)) - \sigma(x + g(x))|| + ||\sigma(x + g(x))|| + ||g(x)|| \leq c(k)(\eps + \delta) + 2 < 5/2\, .
\end{gather}
Conversely, if $y \in B_{3/2} \cap V$, then
\begin{gather}
||\sigma(Q(y) + g(Q(y)))|| = ||y + \tilde g(y)|| \leq 3/2 + c(k)(\eps + \delta) < 2\, ,
\end{gather}
again provided $\eps_1(k)$ is small.  This completes the proof of part B).

Let us prove part C).  For ease of notation write $x^{++} = Q(x) + g(Q(x))$, and $y^{++} = Q(y) + g(Q(y))$.  By estimates \eqref{eqn:est-Q} and part B), $x^{++} \in B_{2 + c(k)(\delta + \eps)}$ whenever $x \in \tilde\Omega$.  Therefore, we can write
\begin{align}
\tilde g(x) &= -\pi^\perp(p) + \sum_i \phi_i(x^{++})\pi^\perp(p_i + \pi_i(x^{++} - p_i)) \, .
\end{align}

For any $x$ with $\phi_i(x^{++}) > 0$, we can estimate using Proposition \ref{prop:almost-proj}:
\begin{align}
||\pi^\perp( \pi_i^\perp(p_i) + \pi_i(x^{++}) ) ||
&\leq c(k) ||\pi_i^\perp(p_i - \tilde p_i)|| + c(k) ||\pi^\perp_i(\tilde p_i)|| + ||\pi^\perp(\pi_i(x^{++}))|| \\
&\leq c(k) \delta + c(k) \delta ||\tilde p_i|| + c(k) \delta ||x^{++}|| \\
&\leq c(k)\delta\, .
\end{align}
Using Lemma \ref{lemma_Banach_packing}, and the definition of $x_i$, we deduce that
\begin{gather}
||\tilde g(x)|| \leq c(k)\delta \quad \text{ for } x \in \tilde\Omega\, .
\end{gather}

Similarly, we can estimate
\begin{align}
||\tilde g(x) - \tilde g(y)||
&\leq \sum_i |\phi_i(x^{++}) - \phi_i(y^{++})| ||\pi^\perp(\pi_i^\perp(p_i) + \pi_i(x^{++}))|| + \sum_i |\phi_i(y^{++})| ||\pi^\perp(\pi_i(x^{++} - y^{++}))|| \\
&\leq c(k) \delta ||x^{++} - y^{++}|| + c(k)\delta ||x^{++} - y^{++}||\\
&= c(k) \delta ||Q(x) - Q(y) + g(Q(x)) - g(Q(y))|| \\
&\leq c(k) \delta ||x - y||\, ,
\end{align}
using the estimates \eqref{eqn:est-Q}.  This completes the proof of part C).

Finally, we show D). From part A), we have the coarse bounds
\begin{gather}
\frac{1}{2} ||x - y|| \leq ||\sigma(x^+) - \sigma(y^+)|| \leq 2||x - y||, \quad \frac{1}{2} ||x - y|| \leq ||x^+ - y^+|| \leq 2 ||x - y||\, ,
\end{gather}
and
\begin{gather}
||\pi^\perp(\sigma(x^+) - \sigma(y^+))|| \leq c(\delta + \eps)||x - y||\, .
\end{gather}
We claim that
\begin{gather}\label{eqn:small-tangential}
||\pi(\sigma(x^+) - \sigma(y^+)) - (x - y)|| \leq c (\rho(c(\delta + \eps)) + (\delta + \eps)^2) ||x - y||\, .
\end{gather}
To see this, write
\begin{gather}
\pi(\sigma(x^+) - \sigma(y^+)) = (x - y) + \sum_i (\phi_i(x^+) - \phi_i(y^+))\pi(\pi_i^\perp(x^+ - p_i)) + \sum_i \phi_i(y^+) \pi(\pi_i^\perp(x^+ - y^+))
\end{gather}
Then, similar to part A), but making use of Lemma \ref{lem:operator-diff}, we can estimate
\begin{align}
\norm{\sum_i (\phi_i(x^+) - \phi_i(y^+))\pi(\pi_i^\perp(x^+ - p_i)) }
&\leq c(k) ||x^+ - y^+|| \left( \sup_{x_i \in B_6(0)} ||\pi - \pi_i|| ||\pi_i^\perp(x^+ - p_i)||\right) \\
&\leq c(k) \ton{\frac{\rho(c\delta)}{\delta} + \delta}(\delta + \eps) ||x - y|| \\
&\leq c(k) (\rho(c(\delta + \eps)) + (\delta + \eps)^2) ||x - y||\, ,
\end{align}
where in the last inequality we used the convexity of $\rho_X$.  Similarly,
\begin{gather}
\norm{\sum_i \phi_i(y^+) \pi(\pi_i^\perp(x^+ - y^+))} \leq c(k) \left( \sup_{x_i \in B_6(0)} ||\pi - \pi_i|| ||\pi^\perp_i(x^+ - y^+)|| \right) \leq c(k) (\rho(c(\delta + \eps)) + (\delta + \eps)^2) ||x - y||.
\end{gather}
This establishes our claim.  By an essentially verbatim proof, we have also
\begin{gather}
||\pi(\sigma(x^+) - x^+)|| \leq c\rho(c(\delta + \eps))||x - y||\, .
\end{gather}

Using \eqref{eqn:small-tangential}, and \eqref{eqn:pythag-f} with the bounds of part A), we get
\begin{align}
\Big ||\sigma(x^+) - \sigma(y^+)||^2 - ||x^+ - y^+||^2 \Big| 
&\leq 4\rho_X(c(\delta + \eps)) ||x^+ - y^+||^2 + ||\pi( \sigma(x^+) - \sigma(y^+) - (x^+ - y^+))||||x - y|| \\
&\leq c \rho_X( c(\delta + \eps)) ||x^+ - y^+||^2.
\end{align}
\end{proof}

\subsection{Regraphing}

We demonstrate that graphs in the sense of Definition \ref{def:graph} (with small norm) over a given affine plane $p + V$, can be written as graphs over slightly tilted/shifted affine planes $q + W$, with small norm also. This lemma is very intuitive in Euclidean spaces, although its proof is not so short. Here we present a Banach space version.

\begin{lemma}\label{lem:regraph}
Let $V$, $W$ be $k$-spaces, with almost-projections $\pi_V$, $\pi_W$, and take points $p, q \in B_{2r}$.  Suppose we know
\begin{gather}\label{eqn:regraph-cond}
d(q, p + V) < \delta r, \quad d_G(V, W) < \delta\, .
\end{gather}

Suppose $G$ is such that
\begin{gather}
G \cap B_{2r} = \graph_{\Omega, \pi_V}(g), \quad r^{-1}||g|| + \lip(g) < \eps, \quad B_{7r/5} \cap (p + V) \subset \Omega \subset (p + V)\, .
\end{gather}
Then provided $\delta + \eps \leq \eps_2(k)$, we have a region $U \subset q + W$, and Lipschitz $g : U \to X$, so that
\begin{gather}
G \cap B_r = \graph_{U, \pi_W}(h), \quad r^{-1}||h|| + \lip(h) \leq c(k) (\eps + \delta), \quad B_{3r/5} \cap (q + W) \subset U \subset (q + W)\, .
\end{gather}
\end{lemma}


\begin{remark}
If $W = V$ and $p = q$, then this demonstrates the ``well-definition'' of graphicality in the sense of Definition \ref{def:graph}: if $G$ is a sufficiently small graph with respect to some almost-projection, then it is a graph with respect to \emph{any} almost-projection.  Unfortunately, in a general Banach space, regraphing $G$ over a different almost-projection will always pick up a factor of $c(k)$, even in the special case of $W = V$.
\end{remark}

\begin{proof}
In the following we denote by $c$ a generic constant depending only on $k$, and always assume $\eps_2(k)$ is chosen so that $\eps_2 c \leq \frac{1}{100}$.  Again by scaling we can assume $r = 1$.

First, there is no loss in assuming $||p - q|| < \delta$.  This follows because we can choose $\tilde p \in V$ with $||q - \tilde p|| < \delta$, and then $p + V = (p - \tilde p) + V$.  Let 
\begin{gather}
\Pi_{V}(x) \equiv \pi^\perp_V(p) + \pi_V(x), \quad \Pi_{W}(x) \equiv \pi^\perp_W(q) + \pi_W(x)
\end{gather}
be the associated almost-affine projections to $p + V$, $q + W$ (recall that $\Pi_V$ is independent of choice of $p \in p + V$).  We have
\begin{gather}
||\Pi_V|| \leq c(k), \quad ||\Pi_W || \leq c(k)\, .
\end{gather}

Observe that $\Pi_V : (q + W) \to (p + V)$ is a bi-Lipschitz equivalence, with estimates
\begin{gather}
||\Pi_V(y) - y|| \leq c \delta (1 + ||y||), \quad ||(\Pi_V(y) - y) - (\Pi_V(z) - z)|| \leq c \delta ||y - z||\, ,
\end{gather}
whenever $y, z \in q + W$.  This follows because, using Proposition \ref{lem:almost-proj}, 
\begin{gather}
||\Pi_V(y) - y|| = ||\pi_V^\perp(p - q) + \pi_V^\perp(y - q)|| \leq c||p - q| + c \delta ||y - q||\, .
\end{gather}
Similarly, we have
\begin{gather}
||(\Pi_V(y) - y) - (\Pi_V(z) - z)|| = ||\pi_V^\perp(y - z)|| \leq c\delta||y - z||\, .
\end{gather}

Define the map $f : B_{6/5} \cap (q + W) \to (q + W)$ by
\begin{gather}
f(y) = \Pi_W( \Pi_V(y) + g(\Pi_V(y)))\, .
\end{gather}
Since $\Pi_V(B_{6/5} \cap (q + W)) \subset B_{6/5+c\delta} \cap V$, we see that $f$ is well-defined and Lipschitz.

We estimate, for $y, z \in B_{6/5}\cap (q + W)$, 
\begin{gather}
||f(y) - y|| = ||\Pi_W(\Pi_V(y) - y + g(\Pi_V(y)))|| \leq c(\delta + \eps)\, ,
\end{gather}
and
\begin{gather}
||(f(y) - y) - (f(z) - z)|| = ||\Pi_W(\pi_V^\perp(y - z) + g(\Pi_V(y)) - g(\Pi_V(y)))|| \leq c (\delta + \eps) ||y - z||\, .
\end{gather}
Therefore, by our restriction on $\eps_2(k)$, $f$ has a Lipschitz inverse
\begin{gather}
f^{-1} : U \subset (q + W) \to B_{6/5} \cap (q + W)\, ,
\end{gather}
with $||f^{-1}|| + \lip(f^{-1}) \leq 3$.

Let us define $\tilde g : U \to X$ by 
\begin{gather}
\tilde g(y) = \pi_W^\perp(\Pi_V(f^{-1}(y)) + g(\Pi_V(f^{-1}(y)))) - \pi^\perp_W(q)\, .
\end{gather}
Then, for $y \in U$, we have
\begin{gather}
\Pi_V(f^{-1}(y)) + g(\Pi_V(f^{-1}(y))) = y + \tilde g(y)\, ,
\end{gather}
and so
\begin{gather}
\{ x + g(x) : x \in \Pi_V(B_{6/5} \cap (q + W)) \} = \graph_{U, \pi_W}(\tilde g)\, .
\end{gather}

Let us demonstrate the correct estimates on $\tilde g$.  For $y, z \in U$, we have
\begin{gather}
||\tilde g(y)|| \leq c||(\pi_V(f^{-1}(y) - q))|| + c ||q|| + c||p|| + c||g(\Pi_V(f^{-1}(y)))|| \leq c(\delta + \eps)\, ,
\end{gather}
and
\begin{align}
||\tilde g(y) - \tilde g(z)|| 
&\leq c||\pi_V(f^{-1}(y) - f^{-1}(z))|| + c||g(\Pi_V(f^{-1}(y))) - g(\Pi_V(f^{-1}(z)))|| \\
&\leq c(\delta + \eps) ||f^{-1}(y) - f^{-1}(z)|| \\
&\leq c(\delta + \eps) ||y - z||\, .
\end{align}

Therefore, it remains only to show
\begin{gather}
\{ x + g(x) : x \in \Pi_V(B_{6/5} \cap (q + W)) \} \supset G \cap B_1(0) \supset \{ y + \tilde g(y) : y \in B_{3/5} \cap (q + W) \}\, .
\end{gather}
On the one hand, if $x + g(x) \in B_1(0)$, then writing $\Pi_V^{-1} : V \to W$ we have
\begin{gather}
||\Pi_V^{-1}(x)|| \leq (1 + c\delta)||x + g(x) - g(x)|| < 1 + c\delta + c\eps < 6/5\, .
\end{gather}
On the other hand, if $y \in B_{3/5} \cap (q + W)$, then
\begin{gather}
||y + \tilde g(y)|| < 3/5 + c(\delta + \eps) < 1\, .
\end{gather}
This completes the proof of Lemma \ref{lem:regraph}.
\end{proof}

\vspace{1cm}

\section{Power gain: examples}

Before moving to the proof in general, we show here some examples illustrating the behavior we can and cannot expect. In particular, we want to see what kind of estimates on the bi-Lipschitz constant we can expect in equation \eqref{eq_2lip_+} (or equivalently \eqref{eqn:teaser-power-gain}). The examples that follow illustrate two phenomena: the first is that we cannot improve \eqref{eq_2lip_+} to
\begin{gather}
 \Big| ||\sigma(x) - \sigma(y)||^2 - ||x - y||^2 \Big| \leq c_4f(\epsilon) ||x - y||^2
\end{gather}
for any $f(\epsilon)\leq c\epsilon^{\alpha'}$ with $\alpha'<\alpha$, where $\alpha$ is the power type of the ambient Banach space $X$ defined in \eqref{eq_deph_power}. 

The second is that in a general Banach space $X$ and for $k\geq 2$, the improved bi-Lipschitz estimate of \eqref{eq_2lip_+} is wrong, and the best one can hope for is \eqref{eq_2lip_X}. 

\subsection{Power gain in \texorpdfstring{$\R^2$}{R2} with Banach norms}\label{sec_ex_R2}
Our first example is an easy example of a curve in $\R^2$ equipped with different $l^p$ norms for $1\leq p\leq 2$. Recall that the $l^p$ norm on $\R^2$ is defined by
\begin{gather}
 \norm{(x,y)}_p = \begin{cases}
                   \ton{\abs x^p+\abs y^p}^{1/p} & \text{for } p\in [1,\infty)\, ,\\
                   \max\cur{\abs x,\abs y} & \text{for } p=\infty\, .
                  \end{cases}
\end{gather}
We will denote by $e_1,e_2$ the standard vector basis of $\R^2$.

Let $\gamma_1:[0,1]\to \R^2$ be the curve given by $\gamma_1(t)=te_1$. For all $p$, this curve has a well-defined length, which is
\begin{gather}
 \int_0^1 \norm{\dot \gamma_1}_{l^p} dt = 1\, .
\end{gather}
For all $\abs{\epsilon}\leq 1$, define the curve $\gamma_2:[0,1]\to \R^2$ by
\begin{gather}
 \gamma_2(t)= \begin{cases}
               te_1 & \text{for} \ \ t\in [0,1/3]\, ,\\
               te_1 +\ton{t-1/3} \epsilon e_2 & \text{for} \ \ t\in [1/3,1/2]\, ,\\
               te_1 +\ton{2/3-t} \epsilon e_2 & \text{for} \ \ t\in [1/2,2/3]\, ,\\
               te_1 & \text{for} \ \ t\in [2/3,1]\, .
              \end{cases}
\end{gather}
For those familiar with fractals, this curve is the first step of a snowflake construction with step $\epsilon$. Clearly $\gamma_2$ is a Lipschitz curve which is $C^1$ away from the points $(1/3,2/3)$. Its speed as a function of $p$ is given by
\begin{gather}
 \norm{\dot \gamma_2(t)}_{l^p} = \begin{cases}
                                  1 & \text{for} \ t\in [0,1/3)\, ,\\
                                  \ton{1+\abs{\epsilon}^p}^{1/p} & \text{for} \ t\in (1/3,2/3)\, ,\\
                                  1 & \text{for} \ t\in (2/3,1]\, .
                                 \end{cases}
\end{gather}
Consider the projection map $\pi:\R^2\to \R^2$ given by 
\begin{gather}
 \pi(x,y)=(x,0)\, .
\end{gather}
This is the standard orthogonal projection in $\R^2$, and it is easy to verify that for all $1\leq p \leq 2$ this is a generalized projection with $\norm \pi =1$. Moreover, for $1<p\leq 2$ this is the $J$-projection (recall Definition \ref{deph_Jproj} and \eqref{eq_Jp}) onto the subspace $V=\operatorname{span}(e_1)$.

Clearly, for all $1\leq p \leq 2$, the curve $\gamma_2$ is a generalized graph (recall Definition \ref{def:graph}) over the subspace $V$ with projection $\pi$, and this projection $\pi(\gamma_2(t))=\gamma_1(t)$ is a bi-Lipschitz equivalence with bi-Lipschitz constant $(1+\abs {\epsilon})$ for all $1\leq p\leq 2$.

However, for $1<p\leq 2$, the bi-Lipschitz constant can be improved to
\begin{gather}
 \ton{1+\abs{\epsilon}^p}^{1/p}\leq 1+c\rho_{(\R^2,l^p)}(\epsilon)\sim_{\epsilon\to 0} 1+\frac 1 p \abs{\epsilon}^p\, ,
\end{gather}
where we used the estimate \eqref{eq_rho_lp} for the modulus of smoothness $\rho_{(\R^2,l^p)}$. In particular, this implies that for all points $z,w\in \gamma_2$, we have
\begin{gather}
 \abs{\norm{\pi(z)-\pi(w)}^2-\norm{z-w}^2}\leq c\rho_{(\R^2,l^p)}(\epsilon)\norm{z-w}^2\, .
\end{gather}

\vspace{2mm}

In the language of the Banach Squash Lemma \ref{lem:squash}, we can rephrase this example in the following terms. We consider the Banach space $X=(\R^2,l^p)$ and the mapping $\sigma=\pi$. In other words, we have a single $1$-dimensional affine space $V=\operatorname{span}(e_1)$ and a single projection $\pi$ onto this subspace, thus we do not need any partition of unity $\cur{\lambda_i}$ to define the map $\sigma$.

$G=\gamma_2$ is a generalized graph over the segment $\ton{[0,1]\times \cur{0}}\subset V$, and the graphing function $g$ satisfies 
\begin{gather}
 \norm{g}_\infty=\epsilon\, , \quad \operatorname{Lip}(g)=\epsilon\, .
\end{gather}
The projection map $\sigma=\pi$ is an explicit bi-Lipschitz equivalence between $G$ and $\ton{[0,1]\times \cur{0}}\subset V$, with bi-Lipschitz constant equal to $(1+\abs{\epsilon}^p)^{1/p}$, which shows that we cannot improve \eqref{eq_2lip_+} to
\begin{gather}
 \Big| ||\sigma(x) - \sigma(y)||^2 - ||x - y||^2 \Big| \leq c_4f(\epsilon) ||x - y||^2
\end{gather}
for any $f(\epsilon)\leq c\epsilon^\alpha$ for $\alpha<p$. 

\vspace{.5cm}

\subsection{Infinite dimensional snowflake}\label{sec:snowflake-Linfty} 
An instructive example to look at is the classical example of the snowflake. In particular, we recall the following standard construction in $\R^2$ (see for example \cite[Exercise 10.16]{BP}). 

\begin{wrapfigure}{l}{37mm}
\begin{center}
 \begin{tikzpicture}[]
  \draw decorate{
       (0,0) - +(0:3)  };     
\end{tikzpicture}

\vspace{2mm}
\begin{tikzpicture}[decoration=Koch snowflake]
  \draw decorate{
       (0,0) -- ++(0:3)  };     
\end{tikzpicture}
\vspace{2mm}

\begin{tikzpicture}[decoration=Koch snowflake]
  \draw decorate{ decorate{
       (0,0) -- ++(0:3)  }};
\end{tikzpicture}
\vspace{2mm}

\begin{tikzpicture}[decoration=Koch snowflake]
  \draw decorate{decorate{ decorate{
       (0,0) -- ++(0:3)  }}};
\end{tikzpicture}
\end{center}
\end{wrapfigure}

\noindent The construction of a snowflake of parameter $\eta>0$ is well known (see for example \cite[section 4.13]{mattila}). Take the unit segment $[0,1]\times\{0\}\subseteq \R^2$, and replace the middle part $[1/3,2/3]\times\{0\}$ with the top part of the isosceles triangle with base $[1/3,2/3]\times\{0\}$ and of height $\eta \cdot \operatorname{lenght}([1/3,2/3]\times \{0\})$. In other words, you are replacing the segment $[1/3,2/3]\times\{0\}$ with the two segments joining $(1/3,0)$ to $(1/2,\eta/3)$, and $(1/2,\eta/3)$ to $(2/3,0)$. Then repeat this construction inductively on each of the $4$ straight segments in the new set. Here on the left hand side you can see the very classical picture of the first three steps in the construction of the standard snowflake, with $\eta=\sqrt 3 /2 $.

It is clear that the length of the curve at step $i$ is equal to the length at step $i-1$ times $2/3+\sqrt{1+\eta^2}/3$, so the length of the snowflake will be infinity for any $\eta>0$. This is a simple application of the Pythagorean theorem, and the extra square power on $\eta$ comes from the fact that at each step we are adding some length $\eta$ to the curve, but in a direction perpendicular to it.

However, if we replace the fixed parameter $\eta$ with a variable parameter $\eta_i$, we see immediately that the length of the limit curve will be finite if and only if $\sum \eta_i^2<\infty$. This suggests that in $\R^2$ a curve $\gamma$ is of finite length if for all $x\in \gamma$, $\int_0^1 \beta^2_1(x,r)\frac{dr}{r}<\infty$.

\vspace{2mm}

\textbf{Snowflake in $L^p$ spaces} Here we try to produce a similar example in infinite dimensions, and we will see that the finiteness of the length of the curve depends on the summability of $\sum \eta_i^\alpha$, where $\alpha$ depends on the space.

Consider the space $L^\infty[0,1]$, and let $e_i\in L^\infty [0,1]$ be the Rademacher's functions. In other words, we set $e_1=\ind_{[0,1]}=1$, $e_2= \ind_{[0,1/2]}-\ind_{(1/2,1]}$, $e_3= \ind_{[0,1/4]}-\ind_{(1/4,1/2]}+\ind_{(1/2,3/4]}-\ind_{(3/4,1]}$, ... $e_i(t)= \operatorname{sign}[\sin(2\pi i t)]$.

Now consider the curve $\gamma_1:[0,1]\to L^\infty([0,1])$ given by $\gamma_1(t)=te_1$. This curve has a well-defined length, which is
\begin{gather}
 \int_0^1 \norm{\dot \gamma_1}_{L^\infty} dt = 1\, .
\end{gather}
We build a sequence of curves $\gamma_n$ similar to snowflakes with parameter $\eta_n$, but developed over an infinite dimensional space instead of $\R^2$. In particular, take $\gamma_1$, split it into $3$ pieces of equal length, and modify the middle piece by ``bumping'' it in the direction of $e_2$. In particular:
\begin{gather}
 \gamma_2(t)= \begin{cases}
               te_1 & \text{for} \ \ t\in [0,1/3]\, ,\\
               te_1 +\ton{t-1/3} {\eta_1} e_2 & \text{for} \ \ t\in [1/3,1/2]\, ,\\
               te_1 +\ton{2/3-t} {\eta_1} e_2 & \text{for} \ \ t\in [1/2,2/3]\, ,\\
               te_1 & \text{for} \ \ t\in [2/3,1]\, .
              \end{cases}
\end{gather}
Then we repeat this process inductively on $i$, and apply the previous construction on each of the straight segment in $\gamma_i$ by bumping it in the direction of $e_i$. 

For each $i$, $\gamma_i:[0,1]\to L^\infty[0,1]$ is a Lipschitz function which is $C^1$ away from the points $k\cdot 3^{-i}$. The speed of $\gamma_2$ and $\gamma_3$ is given by
\begin{gather}
 \norm{\dot \gamma_2(t)}_{L^\infty} = \begin{cases}
                                  1 & \text{for} \ t\in [0,1/3)\, ,\\
                                  1+\abs{\eta_1} & \text{for} \ t\in [1/3,2/3)\, ,\\
                                  1 & \text{for} \ t\in [2/3,1]\, .
                                 \end{cases}\\
 \norm{\dot \gamma_3(t)}_{L^\infty} = \begin{cases}
                                  \begin{cases} 
                                  1 & \text{for} \ t\in [0,1/9)\, ,\\
                                  1+\abs{\eta_2} & \text{for} \ t\in [1/9,2/9)\, ,\\
                                  1 & \text{for} \ t\in [2/9,1/3)\, ,
                                  \end{cases}\\
                                  \begin{cases} 
                                  1 +\abs{\eta_1}& \text{for} \ t\in [1/3,4/9)\, ,\\
                                  1+\abs{\eta_1}+\abs{\eta_2} & \text{for} \ t\in [4/9,5/9)\, ,\\
                                  1 +\abs{\eta_1} & \text{for} \ t\in [5/9,2/3)\, ,
                                  \end{cases}\\
                                  \begin{cases} 
                                  1 & \text{for} \ t\in [2/3,7/9)\, ,\\
                                  1+\abs{\eta_2} & \text{for} \ t\in [7/9,8/9)\, ,\\
                                  1 & \text{for} \ t\in [8/9,1]\, ,
                                  \end{cases}
                                 \end{cases}
\end{gather}
It is easy to see that for a generic $i$ the length of the curve obtained in this fashion is then
\begin{gather}
 L(\gamma_i)=\int_0^1 \norm{\dot \gamma_i}_{L^\infty} = 1+\frac 1 3 \sum_{k=1}^{i-1} \abs{\eta_k}\, .
\end{gather}
This implies that the pointwise limit $\gamma_\infty=\lim_i \gamma_i$ is a curve of finite length if and only if $\sum_{k=1}^{\infty} \abs{\eta_k}<\infty$.

\vspace{3mm}
Notice that the same family of curves $\gamma_i$ seen as curves in $L^2([0,1])$ behaves in a different way. Indeed, in order to compute the speed $\norm{\dot \gamma(t)}$ notice that in $L^2$ we have the identity
\begin{gather}
 \norm{e_1+ \sum_{i\geq 2} \eta_i e_i }^2 = 1+\sum_i \eta_i^2\, ,
\end{gather}
since $e_i$ are orthonormal vectors in $L^2$. Thus it is easy to see that as curves in $L^2$, $\gamma_i$ have uniformly bounded length if and only if $\sum_i \eta_i^2<\infty$. Thus, there is a strong difference in behaviour between $L^2$ and $L^\infty$ from this point of view. 

\vspace{2mm}

Similar computations can be carried out in $L^p[0,1]$, and using the standard inequalities for $L^p$ norms (see Hanner inequality, \cite[theorem 1]{hanner}), it is possible to prove the following lemma.
\begin{lemma}\label{lemma_Lp_snow}
The curves in the family $\gamma_i:[0,1]\to L^p[0,1]$ have uniformly bounded length if $\sup_i \abs{\eta_i}\leq 1/10$ and 
 \begin{gather}\label{eq_Lpex}
    \begin{cases}
   \sum_i \abs{\eta_i}^p <\infty & \text{  for  } \ \ 1\leq p\leq 2\, ,\\
   \sum_i \abs{\eta_i}^2 <\infty & \text{  for  } \ \ 2\leq p<\infty\, ,\\
   \sum_i \abs{\eta_i} <\infty & \text{  for  } \ \ p=\infty\, .
  \end{cases}
 \end{gather}
 Note that for $p\in [2,\infty)$ fixed, the lengths of $\gamma_i$ are uniformly bounded if $\sum_i \abs{\eta_i}^2<\infty$, but this bound is not uniform in $p$.
\end{lemma}

\begin{remark}
 Given the bounds on the modulus of smoothness for $L^p$ given by \eqref{eq_rho_lp}, this behavior suggests a link between the modulus of smoothness of the space $X$ and the Reifenberg theorem. 
\end{remark}

\subsection{Failure of sharp bi-Lipschitz bound}\label{sec:no-power-gain}

We give an example demonstrating the failure of the improved bound \eqref{eqn:teaser-power-gain} when $k \geq 2$, and $X$ is not Hilbert.
In particular, we show that if $X$ is a Banach space, even if its modulus of smoothness of this space is of power type $\alpha>1$, then a Lipschitz graph over some $2$ dimensional space $L$ with Lipschitz constant $\epsilon$ need not be $(1+c\epsilon^\alpha)$ bi-Lipschitz equivalent to its base.

We consider the space $X = \R^3$ with the $\ell^4$ norm
\begin{gather}
|| (x^1, x^2, x^3) ||_{\ell^4} = ( |x^1|^4 + |x^2|^4 + |x^3|^4 )^{1/4}.
\end{gather}
This space is smooth, with modulus of smoothness $\alpha = 2$, so the improved estimate \eqref{eqn:teaser-power-gain} would imply
\begin{gather}\label{eqn:would-be-power-gain}
\Big| ||(x + f(x)) - (y + f(y))||^2 - ||x - y||^2 \Big| \leq c \eps^2 ||x - y||^2 \quad \forall x, y \in L \cap B_1(0) \, ,
\end{gather}
for every $2$-plane $L^2 \subset X$, and every $\eps$-Lipschitz graph function $f : L  \cap B_1(0) \to X$.

However, we shall demonstrate the following failure, precluding \eqref{eqn:would-be-power-gain} for \emph{any} notion of graph.
\begin{prop}\label{prop:no-power-gain}
Let $X = (\R^3, \norm{\cdot}_4)$.  There is a $2$-plane $L^2 \subset X$, and absolute constants $c$, $\eps_0$, with the following property:  Given any function $f : L \cap B_1(0) \to X$, with $\lip(f) = \eps \leq \eps_0$, then we can find a pair $x, y \in L \cap B_1(0)$ admitting a lower bound
\begin{gather}\label{eqn:no-power-gain}
\Big| ||(x + f(x)) - (y + f(y))||^2 - ||x - y||^2 \Big| \geq \eps/c ||x - y||^2\, .
\end{gather}
\end{prop}

\begin{remark}
In fact, the proof shows \eqref{eqn:no-power-gain} for an open neighborhood of $2$-planes.  So this failure is generic, in the sense that you cannot just ``choose a better plane'' or ``choose a better notion of graph.''
\end{remark}

\begin{remark}
Any finite, $n$-dimensional Banach space is $c(n)$-equivalent to a Hilbert space, and any Hilbert structure \emph{does} admit an improved bound \eqref{eqn:would-be-power-gain}.  However, in passing between Banach and Hilbert norms you lose the sharpness of the inequality (i.e. $1+c\eps^2$ would become $c(n)(1+c\eps^2)$).  Moreover, and more importantly, the comparability between norms depends on the ambient dimension $n$, so even for non-sharp estimates like those in \ref{thm:main-packing}, one cannot hope to use a ``comparable'' Hilbert structure to gain a power.
\end{remark}

The failure of the improved estimate \eqref{eqn:would-be-power-gain} is fundamentally a consequence of the \emph{non-linearity} of $J : L \to X^*$.  We explain.  Consider momentarily a general uniformly smooth Banach space $X$, with modulus of smoothness $\alpha$, a $k$-space $L^k$, and an $\eps$-Lipschitz map $f : L^k \to X$.  By the same argument as Lemma \ref{lemma_pyth}, we have
\begin{gather}
||(x + f(x)) - (y + f(y))||^2 - ||x - y||^2 = \ps{J(x - y)}{ f(x) - f(y)} + O(\eps^\alpha) ||x - y||^2 \quad \forall x, y \in L.
\end{gather}
The obstacle to obtaining an improved bi-Lipschitz estimate like \eqref{eqn:would-be-power-gain} is then the quantity
\begin{gather}\label{eqn:power-gain-J-term}
\ps{J(x - y) }{ f(x) - f(y)} .
\end{gather}

When $J|_L$ is linear, then $L$ admits an ``orthogonal complement'' $L^\perp$ satisfying
\begin{gather}\label{eqn:ortho-complement}
L \oplus L^\perp = X, \quad \text{and} \quad \ps{J|_{L} }{ L^\perp} = 0.
\end{gather}
For example, if $\{v_i\}_i$ is a basis for $L$, then take $L^\perp = \cap_i \ker J(v_i)$.  When $L^\perp$ exists, we can define graphs over $L$ to be maps into $L^\perp$, and then \eqref{eqn:power-gain-J-term} vanishes for all such graphs $f : L \to L^\perp$.  This is the origin of the improved bi-Lipschitz estimate \eqref{eqn:teaser-power-gain}.

In both exceptional cases (when $X$ is Hilbert or $k = 1$), $J|_{L}$ is linear, and we correspondingly get both a natural notion of graph and an improved bi-Lipschitz estimate.  When $X$ is Hilbert, the inner product structure gives a natural isomorphism $X \cong X^*$,  and so $J : X \to X^* \cong X$ is the just the identity mapping.  When $k = 1$, $J$ is trivially linear on $1$-spaces since $J$ is always $1$-homogenous.

In fact, these are the \emph{only} cases when $J|_L$ is linear.  A deep theorem of Banach spaces (see \cite[theorem 3.8]{hudzik}) says that $X$ is Hilbert if and only if every closed subspace admits an orthogonal complement $L^\perp$ satisfying \eqref{eqn:ortho-complement}.  If $J|_{L^2}$ were linear for every $2$-plane, then $J$ would be linear on $X$, and by the argument above we could thereby find an orthogonal complement to every closed $L$.

We mention a related, equally remarkable classification, which says that Banach space is $X$ is Hilbert if and only if every $2$-dimensional space admits a norm-one projection (see for example the recent survey \cite[section 3]{beata}).  However we point out that nowhere in our paper do we ever explicitly use that a projection has norm-one.  The improved estimate is more directly a consequence of the existence of orthogonal complements.

\vspace{3mm}

In our example space $X = (\R^3, \ell^4)$, $J$ can be written explicitly as
\begin{gather}\label{eqn:J-l4}
J(x) = J(x^1, x^2, x^3) = ||x||_{\ell^4}^{-2} ((x^1)^3, (x^2)^3, (x^3)^3)\, ,
\end{gather}
where we identify $X^*$ with $(\R^3, \ell^{4/3})$ via the Euclidean inner product $\ps{ \cdot}{ \cdot }$.  On any $2$-space $L$, $J|_L$ is non-linear.  Notice that since $1$-homogenous functions are linear on $1$-spaces, $k \geq 2$ is necessary to see the non-linearity.

When $J|_L$ is non-linear, attempting to satisfy $\ps{J(x - y)}{ f(x) - f(y)} = 0$ for all $x, y \in L \cap B_1(0)$ should impose ``too many'' conditions on a non-constant $f$.  Given $N+1$ points $\{x_i\}_{i=0}^N \subset L \cap B_1(0)$, then
\begin{gather}\label{eqn:J-conds-explain}
\ps{J(x_i - x_j)}{f(x_i) - f(x_j)} = 0 \quad 0 \leq i < j \leq N 
\end{gather}
represents $N(N+1)/2$ linear conditions on only $N+1$ vectors $\{ f(x_i) \}_{i=0}^N$.

We will find that for a generic choice of $L$ and $x_i$, and after fixing the value of $f(x_0)$, the conditions \eqref{eqn:J-conds-explain} are linearly independent, and so force $f$ to be constant.  (Some special $2$-planes, like the coordinate planes $x_i = 0$, admit orthogonal complements in the sense of \eqref{eqn:ortho-complement}, and for these planes conditions \eqref{eqn:J-conds-explain} are degenerate).

We make this precise and quantitative in the following Lemma, which is the key to proving Proposition \ref{prop:no-power-gain}.
\begin{lemma}\label{lem:J-bound}
Let $X = (\R^3, ||\cdot||_4)$.  There is a $2$-plane $L^2 \subset X$, and an absolute constant $c$, with the following property:  Given any Lipschitz $f : L \cap B_1(0) \to X$, with $\lip(f) \leq 1$, then we can find a pair $x, y \in L \cap B_1(0)$, so that
\begin{gather}\label{eq_Jf}
c \left| \ps{J(x - y)}{f(x) - f(y)} \right| \geq \lip(f) ||x - y||^2\, .
\end{gather}
\end{lemma}

The idea behind Lemma \ref{lem:J-bound} is the following.  If we take $N = 5$, and fix $f(x_0) = 0$, then the numbers
\begin{gather}\label{eqn:J-numbers-explain}
\ps{J(x_i - x_j)/||x_i - x_j||}{ f(x_i) - f(x_j)} \quad 0 \leq i < j \leq 5
\end{gather}
represent $15$ separate linear combinations of the $15$ ( = $N \times n$) various coordinates of $f(x_1), \ldots, f(x_5)$.  So the numbers \eqref{eqn:J-numbers-explain} can be expressed as a square matrix $M$ times the vector $(f(x_1), \ldots, f(x_5))$.  We will show that for a good choice of $L^2$, and ``most'' $x_i$, this matrix is invertible, and so lower bounds on the differences $||f(x_i) - f(x_j)||$ pass to lower bounds on the numbers \eqref{eqn:J-numbers-explain}.

First we show how Proposition \ref{prop:no-power-gain} follows from Lemma \ref{lem:J-bound}, then we shall prove Lemma \ref{lem:J-bound}.
\begin{proof}[Proof of Proposition \ref{prop:no-power-gain} given Lemma \ref{lem:J-bound}]
We claim to have the following inequality, for any $x, y \in B_1(0) \cap L$:
\begin{gather}\label{eqn:lip-lower}
\Big| ||(x + f(x)) - (y + f(y))||^2 - ||x - y||^2 \Big| \geq 2|\ps{J(x - y)}{f(x) - f(y)} | - c \eps^2 ||x - y||^2\, ,
\end{gather}
for some absolute constant $c$ independent of $f$. It is clear that this inequality and \eqref{eq_Jf} prove Proposition \ref{prop:no-power-gain}.

In order to prove this claim, recall that since $X$ is smooth, we have $J(x) = \mathrm{grad} ||x||^2/2$.  Define the curve
\begin{gather}
\gamma(t) = x - y + t(f(x) - f(y))\, .
\end{gather}
We can compute
\begin{align}
& || (x + f(x)) - (y + f(y)) ||^2 - ||x - y||^2 \\
&= \int_0^1 2 \ps{J(\gamma(t))}{f(x) - f(y)} dt \notag\\
&= 2\ps{J(x - y)}{f(x) - f(y)} + \int_0^1 2\ps{J(\gamma(t) - \gamma(0))}{f(x) - f(y)} dt \, .\label{eq_Jf1}
\end{align}
Using Lemma \ref{lem:operator-diff} and the estimate \eqref{eq_rho_lp} for the modulus of smoothness of $(\R^2,\norm{\cdot}_4)$, we have the bound
\begin{align}\label{eq_Jf2}
| \ps{J(\gamma(t) - \gamma(0))}{f(x) - f(y)}|
\leq c ||f(x) - f(y)||^2 
\leq c\eps^2 ||x - y||^2\, .
\end{align}
This establishes \eqref{eqn:lip-lower}.
\end{proof}

\begin{proof}[Proof of Lemma \ref{lem:J-bound}]
Take six points $x_0, x_1, \ldots, x_5 \in L^2 \cap B_1(0)$, off-putting for the moment our specific choice of $L$.  Since \eqref{eq_Jf} is invariant under translations $f \mapsto f + \mathrm{const}$, we can and shall assume $f(x_0) = 0$.

Let us define $X \in \R^{15}$ to be the vector of components $(f(x_1), \ldots, f(x_5))$, and define $Y \in \R^{15}$ to be the vector with entries
\begin{align}
\ps{J(x_i - x_j)/||x_i - x_j||}{f(x_i) - f(x_j)} \quad 0 \leq i < j \leq 5.
\end{align}
(remember that $f(x_0) = 0$!)

We can write each component $Y_a$ as the matrix product $Y_a = \sum_{b=1}^{15} M_{ab} X_b$, where $M_{ab}$ is a $15 \times 15$ matrix.  Each entry of $M_{ab}$ is some component of $\pm J(x_i - x_j)/||x_i - x_j||$, and permuting the $x_i$ has the effect of permuting rows of $M$.  Moreover, observe that $M_{ab}$ depends only on the \emph{differences} $x_i - x_j$, and hence there is no loss in assuming $x_0 = 0$ when calculating $\det(M)$.

Fix some choice of norm $|| \cdot ||$ on $\R^{15}$.  Since each entry $|M_{ab}| \leq 1$, we have
\begin{gather}\label{eqn:J-det}
||Y|| \geq (|\det(M)|/c) ||X||\, ,
\end{gather}
for some absolute constant $c$.  We wish to pick a good selection of $x_i$, so that: $\det(M)$ is bounded away from $0$; $||f(x_0) - f(x_1)|| \approx \lip(f) ||x_0 - x_1||$; and $||x_i - x_j|| \approx ||x_0 - x_1||$ for every $i < j$.  These properties, combined with \eqref{eqn:J-det} and our definition of $X$, $Y$, will establish the Lemma.

\vspace{3mm}

Towards this goal, we first verify that $\det(M)$ is bounded away from $0$ for ``most'' choice of $x_i$, in a particular $2$-plane.  From the formula \eqref{eqn:J-l4}, and taking $x_0 = 0$, we see that $\det(M)$ is a $0$-homogenous function, and can be written
\begin{gather}
\det(M) = \frac{D(x_1, \ldots, x_5)}{Q(x_1, \ldots, x_5)}
\end{gather}
where $D$ is a $45$-homogeneous polynomial in the entries of each $x_i$, and $Q$ is an analytic function which vanishes only when some $x_i = x_j$.  Up to sign, each $D$, $Q$ is symmetric under permutations of the $x_i$.

Fix $L$ to be the plane spanned by $v_1 = (1, 1, 0)$, and $v_2 = (0, 1, 1)$.  We claim $D|_L$ is not the zero polynomial.  This follows by a straightforward but tedious calculation.  If we let
\begin{align}
&x_0 = 0, \quad x_1 = v_1+v_2, \quad x_2 = 2v_1 + 3v_2\, , \\
&\quad x_3 = 3v_1 + 4v_2, \quad x_4 = 2v_1 - v_2, \quad x_5 = -v_1 + 3v_2\, ,
\end{align}
then one can compute directly that $D(x_1, \ldots, x_5) \neq 0$.

By writing out
\begin{gather}
D|_L = \sum_\alpha  p_\alpha(x_1) q_\alpha(x_2, \ldots, x_n),
\end{gather}
where the $p_\alpha$ are polynomials in the coordinates of $x_1 \in L$, and $q_\alpha$ are polynomials in coordinates of $x_2, \ldots, x_n \in L$, we see that
\begin{gather}
\{ x_1 \in L : D(x_1, \cdot)|_L = 0 \} = \{ x_1 \in L : p_\alpha(x_1) = 0 \text{ for every } \alpha \} 
\end{gather}
is a dilation-invariant algebraic variety in $L$, and hence is a finite union of lines through the origin.  Repeating this, for $x_2, x_3, \ldots, x_5$, we arrive at the following statement: for all but finitely many $x_1 \in S^1 \subset L$, we can find an $x_2\ldots, x_5 \in S^1$ all distinct so that $D(x_1, \ldots, x_5) \neq 0$, and $||x_1 - x_i|| < 1/100$.

An obvious argument then gives the following.  There is an absolute constant $c$ (depending only on our choice of plane $L$), and a finite, $1/100$-dense subset $I$ of $S^1 \subset L$ (in the sense that any point in $S^1$ is within distance $1/100$ of $I$), so that for every $x_1 \in I$, we can find $x_2, \ldots, x_5 \in S^1 \subset L$, satisfying:
\begin{gather}
\det(M)(x_1, \ldots, x_5) \geq 1/c, \quad \text{and} \quad ||x_1 - x_i|| < 1/100, \quad \forall 2 \leq i \leq 5. 
\end{gather}

For ease of notation write $\eps = \lip(f)$.  Choose $p, q$ so that $||f(p) - f(q)|| > (\eps/2)||p - q||$.  By replacing $\eps/2$ with $\eps/10$, we can assume that
\begin{gather}
B_{||p - q||/5}(q) \subset B_1(0).
\end{gather}
Set $x_0 = p$, and choose some $x_1 \in B_{||p - q||/50}(q) \cap (p + ||p - q||I)$.  We then obtain $x_2, \ldots, x_5 \in B_1(0) \cap \partial B_{||p - q||}(p)$, so that:
\begin{align}
\det(M)(x_0, x_1, \ldots, x_5) \equiv \det(M)(x_1 - x_0, \ldots, x_5 - x_0) \geq 1/c . 
\end{align}

For this choice of $x_0, \ldots, x_5$, we can form the vectors $X$, $Y$ as at the start of the proof, and we get
\begin{align}
||Y|| \geq (1/c) ||X|| \geq (1/c) ||f(x_1)|| = (1/c)||f(x_1) - f(x_0)|| \geq (\eps/c) ||p - q|| .
\end{align}
Therefore, for some $0 \leq i < j \leq 5$, we must have have
\begin{gather}
\left| \ps{J(x_i - x_j)/||x_i - x_j||}{x_i - x_j} \right| \geq (\eps/c) ||p - q|| \geq (\eps/c) ||x_i - x_j||,
\end{gather}
which is the desired conclusion.
\end{proof}

\vspace{3mm}

\section{Covering Lemma}

In this section we present the main covering Lemma of this paper and use it to prove the main Theorem, while we postpone the proof of this covering Lemma to Section \ref{sec:proof-covering}.  Before stating the Lemma, we provide some intuition behind its statement and proof (see Section \ref{sec_sketch_cov} for a more detailed outline of the proof).

\vspace{2mm}

\subsection{Intuition for the Covering Lemma}
We will consider a finite, Borel measure satisfying the following Dini bound:
\begin{gather}\label{eqn:covering-intuition-hyp}
\int_{r_s}^1 \beta^k(s, r)^\alpha \frac{dr}{r} \leq \delta^\alpha, \quad \forall s \in \cS,
\end{gather}
for some $\delta$ small.  A simple scaling argument allows us to reduce to this case.  The $\cS$ is a set of full $\mu$-measure, and the $r_s : \cS \to [0, 1)$ is a radius function, which can be unrelated to $\mu$.  We break up $\cS = \cS_z \cup \cS_+$, where $r_s|_{\cS_z} = 0$, and $r_s|_{\cS_+} > 0$.  One should think of $r_s$ and $\cS$ as a generalized partial covering of $B_1(0)$, consisting of open balls $\{B_{r_s}(s)\}_{s \in \cS_+}$ and a set $\cS_z$, with the property that $\mu(B_1(0) \setminus \cS) = 0$.  \footnote{We mention that since $X$ is not assumed to be separable, the complement of the support of $\mu$ may not have measure zero, so it's better to talk about sets of full measure rather than supports.}

The objective of the Covering Lemma is to build a new partial covering of $B_1(0)$, of the form
\begin{gather}
 F = \cS_z' \cup \bigcup_{s' \in \cS'_+} B_{r_{s'}}(s') \cup \bigcup_{b \in \cB} B_{r_b}(b) \, ,
\end{gather}
where $\cS'_z$ and $\cS'_+$ are suitable subsets of $\cS_z$ and $\cS_+$ respectively, and the new extra balls in the covering ${B_{r_b}(b)}_{b\in \cB}$ are carefully chosen ``bad balls'' according to Definition \ref{deph_bad_balls}. 

Since our final goal is to control the measure $\mu$ away from balls with packing estimates, and obtain rectifiability information for this measure, we require our new covering to have the following properties:
\begin{enumerate}
 \item $F$ need not have full measure, but the discrepancy is controlled:
 \begin{gather}
  \mu(\B 1 0 \setminus F) \leq c\delta^\alpha\, .
 \end{gather}
 \item[(2a)] The balls in the covering, $\{B_{r_{s'}}(s')\}_{s'\in \cS_+'}$ and $\{B_{r_{b}}(b)\}_{b\in \cB}$, admit a uniform $k$-dimensional packing bound 
 \begin{gather}
  \sum_{s' \in \cS'_+} r_{s'}^k + \sum_{b \in \cB} r_b^k \leq c_5(k)\, .
 \end{gather}
 \item[(2b)] The set $\cS'_z$ is contained in the image of a $(1+c\delta^\alpha)$-Lipschitz map $\tau:V\to X$, where $V$ is a $k$-dimensional subspace.  We will take $\cS'_z=\cS_z\cap \tau(V)$, and shall construct the map $\tau$ during the proof (see also the outline in Section \ref{sec_sketch_cov}).
 \item The balls $\cur{\B{r_b}{b}}_{b\in \cB}$ are bad according to Definition \ref{deph_bad_balls}.
\end{enumerate}

The reason $F$ takes this structure is the following.  Recall that good balls (according to Definition \ref{deph_bad_balls}) had ``big'' measure spread out around a $k$-plane, and this allowed us to control the tilting of $L^2$-approximate-best-planes between nearby good balls via the $\beta$-numbers (Section \ref{sec_tilting}).  The vague strategy behind the Covering Lemma is to use this tilting control, and our Dini condition \eqref{eqn:covering-intuition-hyp}, to construct inductively on smaller and smaller scales a sequence of Lipschitz manifolds that approximate the collection of good balls at a given scale.  Regions which are ``far away'' from the approximating manifolds have controlled measure (item 1).  We can iterate on smaller and smaller scales, but must stop if we hit a ball $B_{r_s}(s)$, or some bad ball $B_{r_b}(b)$ (item 3) -- in either case we loose tilting control.  On these balls we get packing estimates (items 2a).  If in certain regions we can iterate infinitely far down, we end up with a Lipschitz manifold covering a piece of $\cS_z$ (item 2b).

Notice the packing estimates in item 2a are not small, regardless of $\delta$.  This is best illustrated in the example when $\mu$ is supported entirely on a $k$-plane $V$: then $\delta = 0$, but we have no control over $\mu \llcorner V$.  In general, the set $F$ forms a cover of the ``limiting'' Lipschitz manifold, which is bi-Lipschitz to a disk, and lives near $L^2$-approximate-best-planes.  $F$ inherits good $k$-dimensional packing/measure bounds, but the $\beta$-numbers give us no control over $\mu$ in this limiting manifold.

\vspace{2mm}

To obtain our Main Theorem \ref{thm:main-packing}, we must refine our cover inside the bad balls $\{B_{r_b}(b)\}_{b \in \cB}$.  By definition of bad balls we know that, up to a set of small measure, $\mu$ inside a given bad ball $B_{r_b}(b)$ is concentrated around some $k-1$ dimensional subspace. Thus we can cover most of $\mu \llcorner B_{r_b}(b)$ with a family of balls 
\begin{gather}\label{eq_b'}
\{\B{\chi r_{b}}{b'}\}_{b'} \quad \text{ with } \quad \#\{b'\} \leq c(k) \chi^{1-k},
\end{gather}
where $\chi$ is chosen small.  Thus we have \textit{small} $k$-dimensional packing estimate on the balls $\{B_{\chi r_b}(b')\}_{b'}$.  On each of these new balls $B_{\chi r_b}(b')$ we can then apply the Covering Lemma again in an inductive fashion until we reach our final goal (that is, a covering not involving bad balls).  The smallness of the $k$-dimensional packing bounds in \eqref{eq_b'} ensures that the \emph{global} $k$-dimensional packing estimate of the new covering obtained in this fashion will remain uniformly controlled in each step of our inductive refinement (for the details, see Section \ref{sec:iterative-packing}).

We remark that the inductive application of the Covering Lemma is the reason we must in \eqref{eq_measure_control} consider the restriction of bad balls $B_{r_b}(b)$ to $\{ s \in \cS : r_s < r_b\}$.  We need to ensure that in every new application of the Covering Lemma at some scale $R$ (occurring inside a bad ball produced from a previous application of the Lemma), we only see $\cS$ with $r_s < R$.

\vspace{3mm}

\subsection{Covering lemma}
Now we state precisely the main covering lemma.

\begin{lemma}[Reifenberg covering]\label{lem:main-covering}
There are constants $\delta_0(k, \rho_X, \chi)$ and $c_5(k)$, so that the following holds.  Let $\mu$ be a finite Borel-regular measure, and $\cS = \cS_z \cup \cS_+$ a set of full $\mu$-measure.  Take $r_s : \cS \to \R_+$ a nonnegative radius function satisfying $r_s < 1$, $r_s|_{\cS_z}=0$ and $r_s|_{\cS_+}>0$.  Assume that $\mu$ satisfies
\begin{gather}\label{eqn:covering-hyp}
\int_{r_s}^\infty \beta(s, r)^\alpha \frac{dr}{r} \leq \delta^\alpha \quad \forall s \in \cS \, ,
\end{gather}
where $\alpha = \alpha(X)$ is the power of smoothness of $X$.

Then provided $\delta \leq \delta_0$, there is a subcollection $\cS'_+ \subset \cS_+$, a collection of ``bad-balls'' $\{B_{r_b}(b) \}_{b \in \cB}$, and a mapping $\tau : p(0,1)+V(0,1) \to X$ which is bi-Lipschitz onto its image, so that the following holds:

A) measure control: if we let 
\begin{gather}\label{eq_measure_control}
F = \left[ \cS_z \cap \tau(B_3(0) \cap (p(0, 1) + V(0, 1))) \right] \cup \bigcup_{s' \in \cS'_+} B_{r_{s'}}(s') \cup \bigcup_{b \in \cB} \left[ B_{r_b}(b) \cap \{ s \in \cS : r_s < r_b \} \right]\, ,
\end{gather}
then
\begin{gather}
\mu ( B_1(0) \setminus F) \leq c(k,\chi) \delta^\alpha\, , 
\end{gather}

B) packing control: $\tau$ is a $(1+c(k,\rho_X, \chi)\delta^\alpha)$-bi-Lipschitz equivalence, and we have 
\begin{gather}\label{eq_packing}
 \sum_{s' \in \cS'_+} r_{s'}^k + \sum_{b \in \cB} r_b^k \leq c_5(k)\, ,
\end{gather}


C) bad ball structure: for each $b \in \cB$, the ball $B_{r_b}(b)$ is bad in the sense of Definition \ref{deph_bad_balls} with respect to $\mu$, and hence is bad with respect to $\mu \llcorner \{ s \in \cS : r_s < r_b \}$ as well.
\end{lemma}

\subsection{Proof of Theorem \ref{thm:main-packing} given Lemma \ref{lem:main-covering}}\label{sec_proof_main}

Before proving the covering lemma, we show that with it we can prove our main Theorem \ref{thm:main-packing}. We postpone the proof of Lemma \ref{lem:main-covering} to Section \ref{sec:proof-covering}.

We first observe that if suffices to prove Theorem \ref{thm:main-packing} when
\begin{gather}\label{eq_M=delta}
 M = \delta^2 = \delta_0^2(k)\, .
\end{gather}
For otherwise, if $0 \neq M \neq \delta_0^2(k)$, we can simply replace $\mu$ with the measure $\delta_0^2 \mu / M$, and use the scaling of $\beta$.  Of course what secretly happens by scaling is that we are changing our definition of good/bad balls -- instead of scaling $\mu$ one could instead incorporate $M$ into Definition \ref{deph_bad_balls}. Note that the same idea has been used in the recent article \cite{miskiewicz}.

If $M = 0$, then Theorem \ref{thm:main-packing} is trivial: By Lemma \ref{lem:beta-zero} we can find an affine $k$-plane $p + V$ so that $\mu(X \setminus (p + V)) = 0$, and then we define $\cS'$ by the condition that $\{B_{r_{s'}}(s')\}_{s' \in \cS'}$ covers $\mu$-a.e. $B_1(0) \cap (p + V)$, while the balls $\{B_{r_{s'}/5}(s')\}_{s' \in \cS'}$ are disjoint.  The required measure estimate is vacuous, and the packing estimate follows from Lemma \ref{lemma_Banach_packing}.

\vspace{3mm}

We observe second that, in the language of Lemma \ref{lem:main-covering}, we have $\cS_+ = \cS$, and $\cS_z = \emptyset$.

We now demonstrate how the Reifenberg Covering Lemma \ref{lem:main-covering} can be used to prove Theorem \ref{thm:main-packing}.  The basic idea is that we can refine the covering on bad balls by applying inductively the covering lemma in order to obtain a finer and finer coverings.

\subsubsection{Inductive claim}
We claim we can find for each $i \geq 0$ a collection of bad balls $\cB_i$, and a subcollection $\cS_i \subset \cS$, with the following properties:
\begin{enumerate}
\item[A)] Measure estimate: if we let
\begin{gather}
F_i = \bigcup_{s \in \cS_i} B_{r_s}(s) \cup \bigcup_{b \in \cB_i} \left[ B_{r_b}(b) \cap \{ s \in \cS : r_s < r_b \} \right],
\end{gather}
then we have
\begin{gather}
\mu(B_1(0) \setminus F_i) \leq \sum_{j=0}^i 2^{-j}\, ,
\end{gather}

\item[B)] Packing estimates: 
\begin{gather}
\sum_{s \in \cS_i} r_s^k \leq 3^k c_2(k) \sum_{j=0}^i 2^{-j}, \quad \text{ and } \quad \sum_{b \in \cB_i} r_b^k \leq 2^{-i}\, , 
\end{gather}

\item[C)] We have $\cS_i \subset \cS_{i+1}$, and $\bigcup_{b \in \cB_i} B_{2r_b}(b) \supset \bigcup_{b \in \cB_{i+1}} B_{2r_b}(b)$.


\item[D)] For each $b \in \cB_i$, we have $r_b \leq \chi^i$ and the ball $B_{r_b}(b)$ is bad \emph{with respect to} $\mu \llcorner \{ s \in \cS : r_s < r_b \}$.
\end{enumerate}

Let us prove this claim by induction.  If $B_1(0)$ is a bad ball then let $\cB_0 = \{0\}$ with corresponding radius function $r_0 = 1$, and let $\cS_0 = \emptyset$.  Conditions A)-D) are vacuous.

Otherwise, if $B_1(0)$ is good, we let $\cB_0 = \cS_0 = \emptyset$, and start from $i = 1$.  To get $\cB_1, \cS_1$, we apply the Covering Lemma \ref{lem:main-covering} to $B_1(0)$ and $\mu$, obtaining a Lipschitz $k$-manifold $T_1$, a collection of bad balls $\cB_1$, and original balls $\cS_1 \subset \cS_+$.  Conditions A)-D) are then immediate, since $\cS_z= \emptyset$.

\subsubsection{Inductive step}\label{sec:packing-induct}
Assume by induction our claim is true for $i$.  Take $b \in \cB_i$.  We know $B_{r_b}(b)$ is bad for $\mu \llcorner \{ s \in \cS : r_s < r_b \}$.  Let us first estimate the ``bad ball excess.''  Set $p + V = p(b, r_b) + V^k_\mu(b, r_b)$, so that
\begin{align}\label{eqn:bad-excess-1}
\mu( B_{r_b}(b) \setminus B_{\chi r_b/30}(p + V)) 
&\leq [\chi r_b/30]^{-2} \int_{\B{r_b}{b}} d(z, p + V)^2 d\mu(z) \\
&\leq c(k,\chi)r_b^k\beta^k_\mu(x,r)^2 \\
&\leq c(k,\chi)\delta^2 r_b^k\, ,
\end{align}
where in the last inequality we used the bound \eqref{eqn:main-assumption} along with \eqref{eq_M=delta} and the estimate \eqref{eq_beta_pws}.

By virtue of being bad there is an affine $(k-1)$-plane $p + L^{k-1} \subset p + V^k$ so that, for any $y \in B_{r_b}(b) \setminus B_{10\chi r_b}(p + L^{k-1})$, we have
\begin{gather}\label{eqn:bad-excess-ball}
\mu ( \{ s \in \cS : r_s < r_b \} \cap B_{\chi r_b}(y) \cap B_{r_b}(b) ) \leq c_2^{-1} (\chi r_b)^k/10\, .
\end{gather}
If $k = 0$ then we interpret $p + L^{k-1} = \emptyset$. By choosing a maximal $\chi r_b/2$-net in
\begin{gather}
\cS \cap B_{r_b}(b) \cap B_{\chi r_b/30}(p + V^k) \setminus B_{10 \chi r_b}(p + L^{k-1})\, ,
\end{gather}
and combining Lemma \ref{lemma_Banach_packing} with \eqref{eqn:bad-excess-ball}, we obtain
\begin{gather}\label{eqn:bad-excess-2}
\mu( \{s \in \cS : r_s < r_b \} \cap B_{r_b}(b) \cap B_{\chi r_b/30}(p + V^k) \setminus B_{10\chi r_b}(p + L^{k-1}) ) \leq r_b^k/10\, .
\end{gather}

We need now only estimate ``lower-dimensional'' neighborhood
\begin{gather}
\{s \in \cS : r_s < r_b \} \cap B_{r_b}(b) \cap B_{\chi r_b/30}(p + V^k) \cap B_{10\chi r_b}(p + L^{k-1})\, .
\end{gather}
Let us define $\cS^b \subset \cS$ by the conditions that, first: 
\begin{gather}
\cS^b \subset \{s \in \cS \cap B_{2r_b}(b) \cap B_{\chi r_b/30}(p + V^k)  \text{ such that  } \chi r_b \leq r_s < r_b\}\, ;
\end{gather}
second: the balls $\{B_{r_s}(s)\}_{s \in \cS^b}$ cover
\begin{gather}
\bigcup \{ B_{r_s}(s) :  s \in \cS \cap B_{2r_b}(b) \cap B_{\chi r_b/30}(p + V^k)  \text{ and } \chi r_b \leq r_s < r_b \} \, ;
\end{gather}
and third: the balls $\{B_{r_s/5}(s) : s \in \cS^b\}$ are disjoint.  One can construct $\cS^b$ by the Vitali covering theorem.  By proximity to $V$ and disjointness, we have by Lemma \ref{lemma_Banach_packing}
\begin{gather}
\sum_{s \in \cS^b} r_s^k \leq 2^k c_2(k) r_b^k\, .
\end{gather}

Now define $\cJ^b$ to be a maximal $2\chi r_b/5$-net in
\begin{gather}\label{eqn:bad-ball-cover}
\cS \cap B_{r_b}(b) \cap B_{10\chi r_b}(L^{k-1}) \cap B_{\chi r_b/30}(V^k) \setminus \bigcup_{s \in \cS^b} B_{r_s}(s)\, .
\end{gather}
We observe that $\{B_{r_x}(x)\}_{x \in \cJ^b}$ covers \eqref{eqn:bad-ball-cover}, that the balls $\{B_{r_x/5}(x)\}_{x \in \cJ^b}$ are disjoint, and by Lemma \ref{lem:L-packing} that $\#\cJ^b \leq c_B(k) \chi^{1-k}$.  Moreover, it is clear from construction that if $x \in \cJ^b$ then
\begin{gather}
s \in \{ s' \in \cS : r_{s'} < r_b \} \cap B_{\chi r_b}(x) \setminus \bigcup_{s' \in\cS^b} B_{r_{s'}}(s') \implies r_s < \chi r_b\, .
\end{gather}

For each $x \in \cJ^b$, apply the Covering Lemma \ref{lem:main-covering} at scale $B_{\chi r_b}(x)$ to the measure $\mu \llcorner \{ s \in \cS : r_s < \chi r_b\}$, and cover $\{ s \in \cS : r_s < \chi r_b \}$ to obtain corresponding collections $\cS_x$, and $\cB_x$.

Now define
\begin{gather}
\cS_{i+1} = \cS_i \cup \bigcup_{b \in \cB_i} \left( \cS^b \cup \bigcup_{x \in \cJ^b} \cS_x \right), \quad \cB_{i+1} = \bigcup_{b \in \cB_i} \bigcup_{x \in \cJ^b} \cB_{x} \, .
\end{gather}

\subsubsection{Packing estimate}\label{sec:iterative-packing}

For each $b \in \cB_i$ we estimate, using our inductive hypothesis, 
\begin{gather}
\sum_{x \in \cJ^b} \left( \sum_{s \in \cS_{x}} r_s^k + \sum_{ b' \in \cB_{x}} r_{b'}^k \right) \leq c_5 \sum_{x \in \cJ^b} r_x^k \leq c_5 c_B \chi r_b^k\,  .
\end{gather}
Choose $\chi(k)$ so that $c_5 c_B \chi < 1/2$.  Then we have
\begin{align}
\sum_{s \in \cS_{i+1}} r_s^k  \leq \sum_{s \in \cS_i} r_s^k + (2^k c_2 + c_5 c_B \chi) \sum_{b \in\cB_i} r_b^k \leq 3^k c_2 \sum_{j=0}^i 2^{-j}\, ,
\end{align}
and
\begin{gather}
\sum_{b \in \cB_{i+1}} r_b^k \leq c_5 c_B \chi \sum_{b \in \cB_i} r_b^k \leq 2^{-i-1}\, .
\end{gather}

\subsubsection{Measure estimate}

By the Covering Lemma, and since $\cS_z = \emptyset$, for each $b \in \cB_i$ and $x \in \cJ^b$ we have
\begin{gather}
\mu \left( \{ s : r_s < \chi r_b \} \cap B_{\chi r_b}(x) \setminus \left( \bigcup_{s \in \cS_x} B_{r_s}(s) \cup \bigcup_{b' \in \cB_{x}} \left[ B_{r_{b'}}(b') \cap \{ s : r_s < r_{b'} \} \right] \right) \right) \leq c \delta^\alpha r_b^k\, .
\end{gather}

Therefore, using our inductive hypothesis,  bounds \eqref{eqn:bad-excess-1}, \eqref{eqn:bad-excess-2}, and ensuring $\delta(k, \rho_X, \chi)$ is sufficiently small, we obtain:
\begin{align}
\mu(B_1(0) \setminus F_{i+1})
&\leq \sum_{j=0}^i 2^{-j} + \sum_{b \in \cB_i} \mu( B_{r_b}(b) \cap \{ s : r_s < r_b \} \setminus F_{i+1}) \\
&\leq \sum_{j=0}^i 2^{-j} + \sum_{b \in \cB_i} (c\delta^\alpha + 1/10) r_b^k + \sum_{b \in \cB_i} \sum_{x \in \cJ^b} \mu(B_{\chi r_b}(x) \cap \{ s : r_s < \chi r_b \} \setminus F_{i+1}) \\
&\leq \sum_{j=0}^i 2^{-j} +  2^{-i} /5 + \sum_{b \in \cB_i} \sum_{x \in \cJ^b} c\delta^\alpha r_x^k \\
& \leq \sum_{j=0}^i 2^{-j} + 2^{-i}/5 + c(k,\chi)\delta^\alpha \chi 2^{-i} \\
&\leq \sum_{j=0}^{i+1} 2^{-j} \, .
\end{align}

\subsubsection{Finally}


Take $\cS' = \cup_i \cS_i$, and set
\begin{gather}
Z = \bigcap_i \bigcup_{b \in \cB_i} \left[ B_{2r_b}(b) \cap \{ s \in \cS : r_s < r_b\} \right]\,  .
\end{gather}
Then by the inclusions C) we have
\begin{gather}
\mu \left( B_1(0) \setminus \left( \bigcup_{s' \in \cS'} B_{r_{s'}}(s') \cup Z \right) \right) \leq \sum_{j=0}^\infty 2^{-j}\, .
\end{gather}
But since $Z \cap \cS_+ = \emptyset$ we have $\mu(Z) = 0$.

\vspace{1cm}

\section{Proof of Covering Lemma \ref{lem:main-covering}}\label{sec:proof-covering}

We build by induction on $i$ a sequence of $k$-dimensional Lipschitz manifolds $T_i$, Lipschitz mappings $\sigma_i : X \to X$, and almost-coverings of $\cS$ by ``good,'' ``bad,'' and ``original'' balls, written as $\cG_i$, $\cB_i$, $\cS_i$.  We also define a sequence of ``remainder sets'' $R_i$, and ``excess sets'' $E_i$.  It will hold that $\cS_i \subset \cS$, and $\cG_i \cup \cB_i \subset \cS \setminus (R_i \cup E_i)$.

As opposed to the construction carried out in Section \ref{sec:packing-induct}, where at every inductive step bad balls were covered in a finer and finer way, here we will stop our construction at the bad and original balls, and continue refining the construction inside good balls.

We shall prove that, for some fixed $\Lambda = \Lambda(k, \chi)$, our manifolds and coverings admit the following properties for every $i$:

\begin{enumerate}

\item\label{iit_1}  $T_0 = p(0, 1) + V(0, 1)$.

\item\label{iit_2} Graphicality of $T_i$: for any $y \in T_i$, there is a $k$-dimensional affine plane $p + V$ (depending on $y$), so that for any choice of almost-projection $\pi_V$ to $V$, we have
\begin{gather}
T_i \cap B_{2\er_i}(y) = \graph_{\Omega, \pi_V}(f)\, , \quad (2\er_i)^{-1} ||f|| + \lip(f) \leq \Lambda \delta\, , \quad B_{1.5 \er_i}(y) \cap (p + V) \subset \Omega \subset (p + V)\, .
\end{gather}
Moreover, if there exists some $g \in \cG_i \cap B_{10\er_i}(y)$, then we can take $p + V = p(g, \er_i) + V(g, \er_i)$.

\item\label{iit_3} Each map $\tau_i = \sigma_i \circ \cdots \circ \sigma_1 : p(0, 1) + V(0, 1) \to T_i$ is a $(1 + c(k,\rho_X, \chi)\delta^\alpha)$-bi-Lipschitz equivalence:
\begin{gather}
\left| \frac{||\tau_i(x) - \tau_i(y)||}{||x - y||} - 1 \right| \leq c(k, \rho_X, \chi) \delta^\alpha \quad \forall x, y \in p(0, 1) + V(0, 1).
\end{gather}

\item\label{iit_4} Ball control: The balls $\{B_{r_s/5}(s)\}_{s \in \cS_i} \cup \{B_{r_b/5}(b) \}_{b \in \cB_i} \cup \{B_{\er_i/5}(g)\}_{g \in \cG_i}$ are all pairwise-disjoint.  Moreover if $x \in \cB_i \cup \cS_i$, then $d(x, T_i) \leq r_x/20$.  Similarly, if $g \in \cG_j$ for $1 \leq j \leq i$, then $d(g, T_i) \leq r_j/20$.

\item\label{iit_5} Radius control: If $b \in \cB_i$ and $s \in \cS \cap B_{r_b}(b) \setminus (E_i \cup R_i)$, then $r_s < r_b$.  Similarly, if $g \in \cG_i$ and $s \in \cS \cap B_{\er_i}(g) \setminus (E_i \cup R_i)$, then $r_s < \er_i$.

\item\label{iit_6} Packing control: we have
\begin{gather}\label{eq_packing_proof}
\sum_{s \in \cS_i \subset \cS} r_s^k + \sum_{b \in\cB_i} r_b^k + \sum_{g \in \cG_i} \er_i^k \leq c_5(k)\, .
\end{gather}

\item\label{iit_7} Covering control: we have
\begin{gather}\label{eq_cov_con_proof}
\mu \qua{ B_1(0) \setminus \ton{\bigcup_{g \in \cG_i} \left[B_{\er_i}(g) \cap \{s \in \cS : r_s < \er_i \} \right] \cup \bigcup_{s \in \cS_i \subset \cS} B_{r_s}(s) \cup \bigcup_{b \in \cB_i} \left[ B_{r_b}(b) \cap \{ s \in \cS: r_s < r_b \} \right] }} \leq c(k, \chi) \delta^\alpha\,  .
\end{gather}

\end{enumerate}

An important consequence of $\cG_i \subset \cS \setminus (E_i \cup R_i)$, item \ref{iit_5} ``radius control,'' and lemma \ref{lem:basic-beta} (and our assumption \eqref{eqn:covering-hyp}), is that: whenever $y \in B_{20 \er_i}(\cG_i)$, and $r \geq \er_i$, then
\begin{gather}\label{eqn:beta-S-shifting}
\beta(y, r) \leq c(k) \delta, \quad \text{and} \quad \int_r^\infty \beta(y, s)^\alpha \frac{ds}{s} \leq c(k) \delta^\alpha.
\end{gather}

\subsection{Sketch of the proof}\label{sec_sketch_cov}
To aid the reader in navigating the proof and construction of the Covering Lemma, we give a rough and imprecise outline of how the manifolds $T_{i+1}$ and new covering at scale $\er_{i+1}$ are inductively built.  The detailed proof is carried out in Section \ref{sec:proof-covering}.

The basic idea is that we want to refine the covering at scale $i$ only on the set of good balls $\cG_i$, since only in these balls do we have tilting control.  We leave the scale $i$ bad and original balls $\cB_i$ and $\cS_{i}$ untouched.

Given a good ball $\B{\er_i}{g}$, we let $p(g,\er_i)+V(g,\er_i)$ be one of its approximate best subspace according to Definition \ref{deph_V}, i.e., a $k$-dimensional subspace almost minimizing the integral $\int_{B_{\er_i}(g)} d(x,p+V)^2 d\mu$.  We define the sets
\begin{gather}
\tilde E_{i+1} = \bigcup_{g \in \cG_i} B_{\er_i}(g) \setminus B_{\er_{i+1}/30}(p(g, \er_i) + V(g, \er_i))\, , 
\end{gather}
which are set of points that are scale-invariantly far from the approximating planes of the good balls.  Given the bounds on $\beta$ given by \eqref{eqn:covering-hyp}, we can infer that the measure of these points is small, see \eqref{eq_exc_a} for precise estimates. We do not refine our covering in the $\tilde E_{i+1}$ .  Neither do we refine our covering over the set of bad and original balls $\cB_i$ and $\cS_i$, which for convenience we denote by
\begin{gather}
 R_i = \bigcup_{b\in \cB_i} \B {r_b}{b} \cup \bigcup_{s\in \cS_i} \B {r_s}{s} \, .
\end{gather}

Thus we focus on the set 
\begin{gather}
 \ton{\bigcup_{g \in \cG_i} B_{\er_i}(g) \cap B_{\er_{i+1}/30}(p(g, \er_i) + V(g, \er_i))}\setminus R_i\, .
\end{gather}
We cover this set by a Vitali collection of balls of radius roughly $\er_{i+1}$, so that the balls with the same centers and $1/5$ of the radius are disjoint (see Section \ref{sec_con} for the precise construction). We classify these balls into three types: original balls $\{\B {r_s}{s}\}_{s\in \tilde \cS_{i+1}}$, bad balls $\{\B {r_b}{b}\}_{b\in \tilde \cB_{i+1}}$ and good balls $\{\B {r_g}{g}\}_{g\in \tilde \cG_{i+1}}$, and set
\begin{gather}
 \cS_{i+1}=\cS_i\cup \tilde \cS_{i+1}\, , \quad \cB_{i+1}=\cB_i\cup \tilde \cB_{i+1}\, , \quad \cG_{i+1}=\tilde \cG_{i+1}\quad \, .
\end{gather}
In other words, we forget about the old good balls, while original and bad balls are cumulative in $i$. Original balls are a subset of the original covering, and good and bad balls are chosen according to Definition \ref{deph_bad_balls}.

In this construction, some care is needed to ensure first, we don't refine inside original balls (item \eqref{iit_5} ``radius control'' of our inductive hypothesis); and second, the balls
\begin{gather}\label{eq_ab}
 \{B_{r_s/5}(s)\}_{s \in \cS_{i+1}} \cup \{B_{r_b/5}(b) \}_{b \in \cB_{i+1}} \cup \{B_{\er_{i+1}/5}(g)\}_{g \in \cG_{i+1}}\, 
\end{gather}
are pairwise-disjoint.

\vspace{2mm}

Now we define the map $\sigma_{i+1}$ and in turn the manifold $T_{i+1}=\sigma_{i+1}(T_i)$ using the constructions and estimates of Section \ref{sec_sigma}. 

In particular, the map $\sigma_{i+1}$ is going to be an interpolation of projection maps $\pi_g$ onto the approximating planes of $B_{\er_{i+1}}(g)$ with $g \in \cG_{i+1}$. These maps are glued together with a partition of unity subordinate to $\{B_{\er_{i+1}}(g)\}_{g \in \cG_{i+1}}$. We do not consider the planes associated to bad and original balls.

Since all $\{B_{\er_{i+1}}(g)\}_{g \in \cG_{i+1}}$ are good balls, we can apply Lemma \ref{lem:good-ball-tilting} to obtain tilting control over the best planes $p(g,\er_{i+1})+V(g,\er_{i+1})$ with $g \in \cG_{i+1}$.  Plugging these estimates into the Squash Lemma \ref{lem:squash}, we obtain that $\sigma_{i+1}$ is a bi-Lipschitz equivalence between $T_i$ and $\sigma_{i+1}(T_i)=T_{i+1}$.

By analyzing these estimates carefully, we prove also that the map $\tau_i=\sigma_i\circ \sigma_{i-1}\circ\cdots\circ \sigma_1: T_0\to T_{i+1}$ has uniform bi-Lipschitz estimates, thus the limit map $\tau=\lim_i \tau_i$ is still a bi-Lipschitz equivalence and $T=\tau(T_0)$ is a Lipschitz manifold with uniform volume bounds.

The importance of the manifold $T_{i+1}$ is that it provides a link among all the balls in the covering, in the sense that all the disjoint balls in \eqref{eq_ab} have quantitatively nonempty intersection with $T_{i+1}$ (see item \eqref{iit_4} in the construction for a more precise statement). This allows us to turn the uniform volume estimates into packing estimates for the covering.


As $i \to \infty$, the covering we constructed
\begin{gather}
 \bigcup_{s \in \cS_{i}}{B_{r_s}(s)} \cup \bigcup_{b \in \cB_{i}}{B_{r_b}(b)}\cup \bigcup_{g \in \cG_{i}}{B_{\er_i}(g)}
\end{gather}
will have three pieces in the limit: the set $\cup_i \cS_i$ of all original balls, the set $\cup_i \cB_i$ of all bad balls, and the set $\cap_i \cG_i$ of limits of good balls.  These pieces become $\cS_+'$, $\cB$, and $\cS_z'$ (respectively) in \eqref{eq_measure_control}.  The last part consists of the points where the refinement of the construction never stops.  Since this piece is contained in $\B {\er_i}{T_i}$ for all $i$, its limit is contained in the manifold $T=\tau(T_0)$ (and thus it is rectifiable).

\subsection{Construction}\label{sec_con}

Recall that we write $\er_i = \chi^i$.  By scaling we can assume $r = 1$ and $p = 0$.  We can also assume $B_1(0)$ is a good ball w.r.t. $\mu$, as otherwise simply take $\cG=\cS_+'=\emptyset$ and $\cB = \{0\}$.  Thus we start our inductive process by defining $\cG_0 = \{0\}$, and $\cS_0 = \cB_0 = \emptyset$ (so, no bad/original balls at scale $r = 1$), $E_0 = \emptyset$, $T_0 = V(0, 1)$ and $R_0=\emptyset$.

Suppose we have defined good/bad/original balls down to scale $\er_i$.  Let us detail the $i+1$ stage of the construction.  Let
\begin{gather}
\tilde E_{i+1} = \bigcup_{g \in \cG_i} B_{\er_i}(g) \setminus B_{\er_{i+1}/30}(p(g, \er_i) + V(g, \er_i))
\end{gather}
be the ``excess set,'' and define for convenience the cumulative excess set by $E_{i+1} = E_i \cup \tilde E_{i+1}$.

We define $\tilde \cS_{i+1}$ by the following three conditions: first,
\begin{gather}
\tilde \cS_{i+1} \subset \cur{ s \in \cS \cap \bigcup_{g \in \cG_i} \left[B_{1.5\er_i}(g) \cap B_{\er_{i+1}/30}(p(g, \er_i) + V(g, \er_i))\right] \setminus R_i \text{ and } \er_{i+1} \leq r_s < \er_i }\, ;
\end{gather}
second, we ask that the balls $\{B_{r_s}(s) \}_{s \in \tilde \cS_{i+1}}$ cover the set
\begin{gather}\label{eqn:S_i-defn}
\bigcup \cur{ B_{r_s/5}(s) : s \in \cS \cap \bigcup_{g \in \cG_i} \left[B_{1.5\er_i}(g) \cap B_{\er_{i+1}/30}(p(g, \er_i) + V(g, \er_i))\right] \setminus R_i \text{ and } \er_{i+1} \leq r_s < \er_i }\, ;
\end{gather}
and third, we require that the balls $\{B_{r_s/5}(s)\}_{s \in \tilde \cS_{i+1}}$ be disjoint.  One can construct $\tilde \cS_{i+1}$ by taking an appropriate Vitali cover.

In order to define $\cG_{i+1}$ and $\cB_{i+1}$, let $\cJ_{i+1}$ be a maximal $2\er_{i+1}/5$-net in
\begin{gather}
\left( \cS \cap B_1(0) \cap \bigcup_{g \in \cG_i} \left[B_{r_g}(g) \cap B_{\er_{i+1}/30}(p(g, \er_i) + V(g, \er_i))\right] \right) \setminus \left( R_i \cup \bigcup_{s \in \tilde \cS_{i+1}} B_{r_s}(s) \right)\, .
\end{gather}
It is easy to see that
\begin{gather}
\cS \cap B_1(0) \setminus (E_{i+1} \cup R_i) \subset \bigcup_{s \in \tilde \cS_{i+1}} B_{r_s}(s) \cup \bigcup_{x \in \cJ_{i+1}} B_{r_x}(x)\, ,
\end{gather}
and the balls $\{B_{r_x/5}(x)\}_{x \in \cJ_{i+1}}$ are disjoint.

We split $\cJ_{i+1}$ into $\tilde \cG_{i+1}$ and $\tilde \cB_{i+1}$ depending on whether $B_{r_x}(x)$ is good or bad w.r.t. $\mu$ and $\chi$ according to Definition \ref{deph_bad_balls}.  We set also
\begin{gather}
R_{i+1} = R_i \cup \bigcup_{s \in \tilde\cS_{i+1}} B_{r_s}(s) \cup \bigcup_{b \in \tilde\cB_{i+1}}  B_{\er_{i+1}}(b)\, ,
\end{gather}
and
\begin{gather}
\cS_{i+1} = \cS_i \cup \tilde\cS_{i+1}, \quad \cG_{i+1} = \tilde \cG_{i+1}, \quad \cB_{i+1} = \cB_i \cup \tilde \cB_{i+1}\, .
\end{gather}
Notice again that while the sets $\cS_{i+1},\cB_{i+1},R_{i+1}$ are ``cumulative'' wrt $i$, the set $\cG_{i+1}$ is not. Moreover,
it is easy to see that
\begin{gather}
B_1(0) \subset E_i \cup R_i \cup \bigcup_{g \in \cG_{i+1}} B_{\er_{i+1}}(g)\, .
\end{gather}

We define $\sigma_{i+1}$ as follows.  Let $\{\phi_g\}_{g \in \cG_{i+1}}$ be the truncated partition of unity subordinate to $\{B_{\er_{i+1}}(g)\}_{g \in \cG_{i+1}}$, as per Lemma \ref{lem:pou}.  For a given $g \in \cG_{i+1}$, let $p_g  + V_g = p(g, \er_{i+1}) + V(g, \er_{i+1})$ be the $L^2$-approximate plane for $B_{\er_{i+1}}(g)$ of Definition \ref{deph_V}.\\
Let also $\pi_g : X \to V_g \equiv V(g, \er_{i+1})$ be a choice of almost-projection.  If $X$ is a Hilbert space, take $\pi_g$ to be the orthogonal projection; if $X$ is uniformly smooth, and $k = 1$, take $\pi_g$ to be the $J$-projection.

We now set
\begin{gather}
\sigma_{i+1}(x) = x - \sum_{g \in \cG_{i+1}} \phi_i(g) \pi_g^\perp(x - p_g)\, ,
\end{gather}
and let $T_{i+1} = \sigma_{i+1}(T_i)$.

\vspace{3mm}

This completes the inductive construction.  In the following subsections we prove the inductive properties asserted in 1)-7).  On can easily check that all properties hold trivially when $i = 0$, and therefore in the rest of this section we can assume by inductive hypothesis that properties 1)-7) hold for all scales between $\er_0$ and $\er_i$.

\subsection{Item \ref{iit_2}: Graphicality}\label{sec:graph}

Fix $y \in T_{i+1}$.  If $y \not\in B_{10\er_{i+1}}(\cG_{i+1})$, then by construction $\sigma_{i+1}$ is the identity on $B_{2\er_{i+1}}(y)$, and item 2 follows by induction.  We can assume that $y \in B_{10\er_{i+1}}(g)$ for some $g \in \cG_{i+1}$.  In the following $c$ denotes a generic constant depending only on $(k, \chi)$, and which is \emph{independent} of $\Lambda$, and we shall assume $\delta_0(k, \chi, \Lambda)$ is small enough so that $c (1+\Lambda) \delta_0 \leq \eps_1$ (the constant from the squash Lemma \ref{lem:squash}).

\vspace{5mm}

First suppose $i = 0$, so that so that $T_i \equiv T_0 = p(0, 1) + V(0, 1) \equiv p_0 + V_0$.  By the tilting Lemma \ref{lem:good-ball-tilting-2}, and by construction, we have for any $\tilde g \in \cG_1 \cap B_{9 r_1}(g)$ the estimates
\begin{gather}\label{eqn:case-0-tilting}
d(\tilde g, p_0 + V_0) < r_1/10, \quad r_1^{-1} d(p_{\tilde g}, p_0 + V_0) + d_G(V(\tilde g, r_1), V_0) \leq c(k, \chi) \beta(0, 5) \leq c \delta\, ,
\end{gather}
where $p_0+V_0=p(0,1)+V(0,1)$ is an approximating subspace on $\B 1 0$. Set $\pi_{\tilde g}$ to be a projection onto $p(\tilde g,\er_1)+V(\tilde g,\er_1)$, and observe that if $x \in B_{6r_1}(y)$, then $\sigma_1(x)$ takes the form
\begin{gather}\label{eqn:case-0-sigma}
\sigma_1(x) = x - \sum_{\tilde g \in \cG_1 \cap B_{9 r_1}(y)} \phi_{\tilde g} \pi_{\tilde g}(x - p_{\tilde g})\, .
\end{gather}

In light of \eqref{eqn:case-0-tilting} and \eqref{eqn:case-0-sigma}, $\sigma_1|_{B_{6r_1}(y)}$ satisfies the hypotheses of the squash lemma at scale $B_{2r_1}(y)$.  Since $T_0 \equiv p_0 + V_0$, we can apply the squash lemma part B) to deduce
\begin{gather}\label{eqn:case-0-graph}
T_1 \cap B_{4r_1}(y) = \graph_{\Omega, \pi'}(f)\, , \quad \er_i^{-1} ||f|| + \lip(f) \leq c(k, \chi) \beta(0,5)\, , \quad B_{3r_1}(y) \cap (p_0 + V_0) \subset \Omega\, .
\end{gather}
Finally, using estimates \eqref{eqn:case-0-tilting} we apply the regraphing Lemma \ref{lem:regraph} to \eqref{eqn:case-0-graph} at scale $B_{2r_1}(y)$ to deduce item 2 when $i = 0$.  Let us mention also that the squash lemma part A) gives the bound
\begin{gather}
||\sigma_1(x) - x || \leq c(k,\chi)\delta r_1 \quad \forall x \in T_0\, .
\end{gather}

\vspace{5mm}

Now suppose $i \geq 1$.  By construction there is a $g' \in \cG_i$ so that $g \in B_{\er_i}(g')$, and a $g'' \in \cG_{i-1}$ so that $g' \in B_{\er_{i-1}}(g'')$.  Let us fix almost-projections $\pi$, $\pi'$, and $\pi''$ to $V_g$, $V_{g'}$, and $V_{g''}$ respectively.

We have by induction $d(g', T_{i-1}) \leq \er_i/30 + \Lambda\delta \er_{i-1} \leq \er_{i-1}/10$, and by construction $B_{3\er_i}(g') \subset B_{1.1\er_{i-1}}(g'')$, and therefore we can write
\begin{gather}
T_{i-1} \cap B_{3\er_i}(g') = \graph_{\Omega, \pi''}(f)\, , \quad \er_i^{-1} ||f|| + \lip(f) \leq c \Lambda \delta\, , \quad B_{2.5 \er_i}(g') \cap (p_{g''} + V_{g''}) \subset \Omega\, .
\end{gather}

From tilting Lemma \ref{lem:good-ball-tilting-2} and construction we have for any $\tilde g \in B_{6\er_i}(g') \cap \cG_i$ the estimates:
\begin{align}\label{eqn:good-tilting-1}
&d(\tilde g, p_{g''} + V_{g''}) \leq \er_i/30 + c\beta(g'', 3\er_{i-1}) \er_i < \er_i/10\, , 
\end{align}
and
\begin{align}\label{eqn:good-tilting-2}
\er_i^{-1} d(p_{\tilde g}, p_{g''} + V_{g''}) + d_G(V(\tilde g, \er_i), V_{g''}) \leq c \beta(g'', 3\er_{i-1}) \leq c \delta\,  .
\end{align}

And similarly, for any $\tilde g \in B_{9\er_{i+1}}(y) \cap \cG_{i+1}$:
\begin{align}\label{eqn:good-tilting-3}
d(\tilde g, p_{g''} + V_{g''}) < \er_{i+1}/10 \, , \quad \er_{i+1}^{-1} d(p_{\tilde g}, p_{g''} + V_{g''}) + d_G(V(\tilde g, \er_{i+1}), V_{g''}) \leq c \beta(g'', 3\er_{i-1}) \leq c \delta\, .
\end{align}

We now observe that for $x \in B_{3\er_i}(g')$ we have
\begin{gather}
\sigma_i(x) = x - \sum_{\tilde g \in \cG_i \cap B_{6\er_i}(g') } \phi_{\tilde g} \pi^\perp_{\tilde g}(x - p_{\tilde g}) \quad \text{ and } \quad \sum_{\tilde g \in \cG_i \cap B_{6\er_i}(g') } \phi_{\tilde g} \equiv 1 \text{ on } B_{2.5 \er_i}(g')\, .
\end{gather}
Therefore, by estimates \eqref{eqn:good-tilting-1} and \eqref{eqn:good-tilting-2}, $\sigma_i|_{B_{3\er_i}}(g)$ satisfies the hypothesis of the squash Lemma \ref{lem:squash} at scale $B_{\er_i}(g)$.  We are justified in applying the squash lemma parts B), C)  to deduce
\begin{gather}\label{eqn:good-graph}
T_i \cap B_{2\er_i}(g') = \graph_{\Omega, \pi''}(f)\, , \quad \er_i^{-1}|| f|| + \lip(f) \leq c(k,\chi) \beta(g'', 3\er_{i-1})\, , \quad B_{1.5\er_i}(g') \cap (p_{g''} + V_{g''}) \subset \Omega\, ,
\end{gather}
with $c$ independent of $\Lambda$.

Since $B_{6\er_{i+1}}(y) \subset B_{1.1\er_i}(g')$ we can use \eqref{eqn:good-graph} to write
\begin{gather}
T_i \cap B_{6\er_{i+1}}(y) = \graph_{\Omega, \pi''}(f), \quad \er_{i+1}^{-1} ||f|| + \lip(f) \leq c(k,\chi) \beta(g'', 3\er_{i-1}), \quad B_{5\er_{i+1}}(y) \cap (p_{g''} + V_{g''}) \subset \Omega\, .
\end{gather}
As above, by construction for $x \in B_{6\er_{i+1}}(y)$, the map $\sigma_{i+1}$ takes the form
\begin{gather}\label{eqn:good-tilting-4}
\sigma_{i+1}(x) = x - \sum_{\tilde g \in \cG_{i+1} \cap B_{9\er_{i+1}}(y)} \phi_{\tilde g} \pi_{\tilde g}(x - p_{\tilde g})\, ,
\end{gather}
though notice we \emph{do not} anymore have equality $\sum_{\tilde g} \phi_{\tilde g} = 1$ in the partition of unity.  By estimates \eqref{eqn:good-tilting-3} we can apply the squash lemma part B) at scale $B_{2\er_{i+1}}(y)$ to obtain 
\begin{gather}
T_{i+1} \cap B_{4\er_{i+1}}(y) = \graph_{\Omega,\pi''}(f), \quad \er_{i+1}^{-1} ||f|| + \lip(f) \leq c(k,\chi) \beta(g'', 3\er_{i-1}), \quad B_{3\er_{i+1}}(y) \cap (p_{g''} + V_{g''}) \subset \Omega\, .
\end{gather}
Finally, again from estimates \eqref{eqn:good-tilting-3} we can apply the regraphing Lemma \ref{lem:regraph} at scale $B_{2\er_{i+1}}(y)$ to prove item 2.

Let us observe further that, by applying the squash lemma part A) to \eqref{eqn:good-graph} at scale $B_{2\er_{i+1}}(y)$ we can obtain the estimate
\begin{gather}\label{eqn:uniform-C0-est}
||\sigma_{i+1}(x) - x|| \leq c(k, \chi) \delta \er_{i+1} \quad \forall x \in T_i\, .
\end{gather}

\subsection{Item \ref{iit_3}: bi-Lipschitz estimates}\label{sec:bi-lip}

The bi-Lipschitz estimates are the core of the covering lemma, and they basically follow from the corresponding estimates in the squash lemma.  First, let us remark that from the uniform estimate \eqref{eqn:uniform-C0-est}, we immediately obtain
\begin{gather}\label{eqn:small-movement}
||\tau_\ell(y) - \tau_j(y)|| \leq c(k, \chi) \delta \er_j \quad \forall 0 \leq j < \ell \leq i+1, \quad \forall y \in p(0,1) + V(0, 1) \, .
\end{gather}

Fix any $x, y \in B_3(0) \cap (p(0, 1) + V(0, 1))$.  Note that wlog we can suppose $x \in B_3(0)$ since every $\sigma_j$ is the identity outside $B_3(0)$.  

Choose a maximal, non-negative $m \leq i + 1$ so that $6 \er_j > ||\tau_j(x) - \tau_j(y)||$ for all $j \leq m$, and $\tau_j(x) \in B_{10 \er_{j+1}}(\cG_{j+1})$ for all $j \leq m-1$.  Notice that when $\tau_{m}(x) \in B_{10 \er_{m+1}}(\cG_{m+1})$, then since necessarily $||\tau_{m+1}(x) - \tau_{m+1}(y)|| \geq 6\er_{m+1}$, we have for $\delta(k,\chi)$ sufficiently small that
\begin{gather}\label{eqn:crude-lip-bound}
||\tau_m(x) - \tau_m(y)|| \geq 6\er_{m+1} - c\delta \er_{m+1} \geq \chi \er_m\, .
\end{gather}
If no such non-negative $m$ exists, take $m = 0$, and then \eqref{eqn:crude-lip-bound} trivially holds.

\textbf{We claim that} for each $j \leq m-1$, we have
\begin{gather}\label{eqn:lip-bound-big-r}
\left| \frac{||\tau_{j+1}(x) - \tau_{j+1}(y)||}{||\tau_j(x) - \tau_j(y)||} - 1 \right| \leq c(k, \rho_X, \chi) \beta(\tau_j(x), 5\er_{j-1})^\alpha\, .
\end{gather}

It will then follow, by \eqref{eqn:beta-S-shifting} and for $\delta(k, \rho_X, \chi)$ sufficiently small, that for any $j \leq m$ we have the bounds
\begin{align}
||\tau_j(x) - \tau_j(y)|| \leq \prod_{\ell=1}^j (1 + c \beta(\tau_j(x), 10 \er_{\ell-1})^\alpha) ||x - y|| \leq \exp\left( c \int_{\er_j}^{10} \beta(\tau_j(x), r)^\alpha \frac{dr}{r} \right) ||x - y||, \quad \text{ and } \label{eqn:lip-bound-calc-1} \\
||\tau_j(x) - \tau_j(y)|| \geq \prod_{\ell=1}^j (1 - c \beta(\tau_j(x), 10 \er_{\ell-1})^\alpha) ||x - y|| \geq \exp\left( -c \int_{\er_j}^{10} \beta(\tau_j(x), r)^\alpha \frac{dr}{r} \right) ||x - y ||,\label{eqn:lip-bound-calc-2}
\end{align}
for $c = c(k, \rho_X, \chi)$ \emph{independent} of $j$ and $m$.

We shall see that claim \eqref{eqn:lip-bound-big-r} is a direct consequence of the graphical estimates from Section \ref{sec:graph}, and the squash Lemma \ref{lem:squash}.  Take $j \leq m-1$.  Like in the proof of item 2, we can find a $g'' \in \cG_{j-1}$ with $\tau_j(x) \in B_{1.1\er_{j-1}}(g'')$, and
\begin{gather}\label{eqn:bilip-graph}
T_j \cap B_{6\er_j}(\tau_j(x)) = \graph_{\Omega, \pi_V}(f)\, , \quad \er_j^{-1} ||f|| + \lip(f) \leq c(k,\chi) \beta(g'', 3\er_{j-1})\, , \quad B_{2.5\er_j}(\tau_j(x)) \cap (p_{g''} + V_{g''}) \subset \Omega\, ,
\end{gather}
for some choice of almost-projection $\pi''$ to $V_{g''}$ (if $j = 0$, then \eqref{eqn:bilip-graph} vacuously holds with $g'' = 0$).  Using estimates \eqref{eqn:good-tilting-3} and relation \eqref{eqn:good-tilting-4} (respectively \eqref{eqn:case-0-tilting}, \eqref{eqn:case-0-sigma} when $j = 0$), $\sigma|_{B_{6\er_{j+1}}(\tau_j(x))}$ satisfies the hypotheses of the squash Lemma \ref{lem:squash} at scale $B_{2\er_{j+1}}(\tau_j(y))$.  Therefore, since $\tau_j(y) \in B_{6\er_{j+1}}(\tau_j(x))$ by definition of $m$, we can apply the squash Lemma \ref{lem:squash} part D) if $X$ is uniformly smooth, or part A) for general $X$, in order to deduce
\begin{gather}
\left| \frac{||\sigma_{j+1}(\tau_j(x)) - \sigma_{j+1}(\tau_j(y))||}{||\tau_j(x) - \tau_j(y)||} - 1 \right| \leq c \beta(g'', 3\er_{j-1})^\alpha \leq c(k,\chi)\beta(\tau_j(x), 5\er_{j-1})^\alpha\, .
\end{gather}
This proves \eqref{eqn:lip-bound-big-r}.

\textbf{We now prove the following estimate: }for any $j \geq m$, we have
\begin{gather}\label{eqn:lip-bound-small-r}
\left| \frac{||\tau_j(x) - \tau_j(y)||}{||\tau_m(x) - \tau_m(y)||} - 1 \right| \leq c(k, \rho_X, \chi) \delta^\alpha\, ,
\end{gather}
with $c$ independent of $j$, $m$, as before.  This clearly completes the bi-Lipschitz estimates.

\vspace{3mm}

First, we notice that if $\tau_m(x) \not \in B_{10\er_{m+1}}(\cG_{m+1})$, then $\sigma_j$ is the identity on $\tau_m(x)$, $\tau_m(y)$ for all $j \geq m + 1$, and hence there is nothing to show.  We henceforth assume that \eqref{eqn:crude-lip-bound} holds.

For $\alpha = 1$ (i.e., if $X$ is a generic Banach space), the estimate is straightforward. Indeed, using \eqref{eqn:small-movement} and \eqref{eqn:crude-lip-bound} we get:
\begin{align}
\Big| ||\tau_j(x) - \tau_j(y)|| - ||\tau_m(x) - \tau_m(y)|| \Big|
&\leq ||\tau_j(x) - \tau_m(x)|| + ||\tau_j(y) - \tau_m(y)|| \\
&\leq c \delta \er_m \\
&\leq c(k,\chi) \delta ||\tau_m(x) - \tau_m(y)||\, .
\end{align}

For $\alpha > 1$ we proceed as follows.  Let us fix $p_m + V_m$ and $\pi_m$ a choice of plane and almost-projection so that, as per \eqref{eqn:bilip-graph}, we have
\begin{gather}\label{eqn:lip-starting-graph}
T_m \cap B_{6r_m}(\tau_m(x)) = \graph_{\Omega, \pi_m}(f)\, , \quad r_m^{-1} ||f|| + \lip(f) \leq c(k, \chi) \delta\, , \quad (p_m + V_m) \cap B_{2.5 r_m}(\tau_m(x)) \subset \Omega\, .
\end{gather}
We first prove the auxiliary estimate
\begin{lemma}\label{lem:improved-tau}
For any $z \in B_{6r_m}(\tau_m(x)) \cap T_m$, and $j > m$, we have
\begin{gather}\label{eq_improved-tau}
||\pi_m(\tau_j(z) - z)|| \leq c(k, \rho_X, \chi) \delta^\alpha \er_m\, .
\end{gather}
\end{lemma}

\begin{proof}
If $\alpha = 1$ then this follows trivially from \eqref{eqn:uniform-C0-est}.  Let us assume $\alpha > 1$. We can assume wlog that $\tau_{t}(z) \in B_{5\er_{{t}+1}}(\cG_{{t}+1})$ for all ${t} \leq j$.

For each ${t}$ with $m < {t} \leq j$, choose a plane $p_{t} + V_{t}$, and almost-projection $\pi_{t}$ to $V_{t}$, so that $T_{t} \cap B_{2\er_{t}}(\tau_{t}(z))$ is graphical over $p_{t} + V_{t}$ as per item \ref{iit_2}.  Moreover, we can choose $V_{t} = V(g_{t}, \er_{t})$ for some $g_{t} \in \cG_{t} \cap B_{2\er_{t}}(\tau_{t}(z))$. 

From the squash lemma part D) we have
\begin{gather}
||\pi_{t}(\sigma_{t}(\tau_{{t}-1}(z)) - \tau_{{t}-1}(z) )|| \leq c(k,\rho_X, \chi) \delta^\alpha \er_{t}\, ,
\end{gather}
and by the tilting Lemma \ref{lem:good-ball-tilting-2} and Lemma \ref{lem:operator-diff}, we have
\begin{gather}
||\pi_{t} - \pi_{{t}-1}|| \leq c(k, \rho_X, \chi) \delta^{\alpha-1}\, .
\end{gather}

Now for each such ${t}$ we compute
\begin{align}
||\pi_m(\tau_{t}(z) - \tau_{{t}-1}(z))||
&\leq ||\pi_{t}(\sigma_{t}(\tau_{{t}-1}(z)) - \tau_{{t}-1}(z))|| + \left( \sum_{\ell=m+1}^{t} ||\pi_{t} - \pi_{{t}-1}|| \right) ||\sigma_{t}(\tau_{{t}-1}(z)) - \tau_{{t}-1}(z)|| \\
&\leq c\delta^\alpha \er_{t} + c {t} \delta^\alpha \er_{t}\, ,
\end{align}
and therefore
\begin{align}
\sum_{{t}=m+1}^j ||\pi_m(\tau_{t}(z) - \tau_{{t}-1}(z))|| \leq \sum_{{t}=m+1}^j c \delta^\alpha {t} \er_{t} \leq c(k, \rho_X, \chi) \delta^\alpha \er_m\, .
\end{align}
This proves \eqref{eq_improved-tau}.
\end{proof}

We proceed to prove \eqref{eqn:lip-bound-small-r}.  We use Lemma \ref{lem:improved-tau} to bound
\begin{align}
\big| ||\pi_m(\tau_j(x) - \tau_j(y))|| - ||\pi_m(\tau_m(x) - \tau_m(y))|| \big| 
&\leq ||\pi_m(\tau_j(x) - \tau_m(x)|| + ||\pi_m(\tau_j(y) - \tau_m(y))|| \\
&\leq c \delta^\alpha \er_m \\
&\leq c \delta^\alpha ||\tau_m(x) - \tau_m(y)||,
\end{align}
and Proposition \ref{prop:graph-est} and \eqref{eqn:crude-lip-bound} to bound
\begin{gather}
\big| ||\pi_m(\tau_m(x) - \tau_m(y))|| - ||\tau_m(x) - \tau_m(y)|| \big| \leq c \delta^\alpha ||\tau_m(x) - \tau_m(y)|| .
\end{gather}
These together imply
\begin{gather}\label{eqn:lip-bound-pi_m}
\Big| ||\pi_m(\tau_j(x) - \tau_j(y))||^2 -  ||\tau_m(x) - \tau_m(y)||^2 \Big| \leq c(k, \rho_X, \chi) \delta^\alpha ||\tau_m(x) - \tau_m(y)||^2.
\end{gather}

By \eqref{eqn:small-movement} and our choice of $m$, we have the coarse bound
\begin{gather}
||\tau_j(x) - \tau_j(y)|| \leq 10 \er_m 
\end{gather}
for small $\delta(k, \chi)$.  Therefore, by the Pythagorean theorem \ref{lem:pythag}, we obtain
\begin{align}
\Big| ||\tau_j(x) - \tau_j(y)||^2 - ||\pi_m(\tau_j(x) - \tau_j(y))||^2 \Big| 
&\leq c(k, \rho_X, \chi) ||\pi_m^\perp(\tau_j(x) - \tau_j(y))||^\alpha \er_m^{2-\alpha} .
\end{align}

Finally, using \eqref{eqn:lip-starting-graph}, and estimates \eqref{eqn:crude-lip-bound}, \eqref{eqn:uniform-C0-est}, we have
\begin{align}
||\pi^\perp_m(\tau_j(x) - \tau_j(y))||^\alpha \er_m^{2-\alpha} 
&\leq \left( ||\tau_j(x) - \tau_m(x)|| + ||\tau_j(x) - \tau_m(x)|| + ||\pi^\perp(\tau_m(x) - \tau_m(y))|| \right)^\alpha \er_m^{2-\alpha} \\
&\leq c(k, \chi) \delta^\alpha \er_m^2 \\
&\leq c(k, \chi) \delta^\alpha ||\tau_m(x) - \tau_m(y)||^2 \, ,
\end{align}
which completes the proof of \ref{eqn:lip-bound-small-r}, and thus the proof of item 3.

\subsection{Item \ref{iit_4}: Ball control}

It is clear that $\{B_{r_s/5}(s)\}_{s \in \tilde \cS_{i+1}} \cup \{ B_{\er_{i+1}/5}(x)\}_{x \in \cJ_{i+1}}$ are pairwise-disjoint, since for each such $s$ we have $r_s \geq \er_{i+1}$.  Now take some $x \in \tilde \cS_{i+1} \cup \cJ_{i+1}$ and $y \in \cS_i \cup \cB_i$.  By construction we have $(\tilde \cS_{i+1} \cup \cJ_{i+1}) \cap B_{r_y}(y) = \emptyset$, and $r_y \geq \er_i$, and $r_x < \er_i$.  It is then immediate that $B_{r_x/5}(x) \cap B_{r_y/5}(y) = \emptyset$.

Let us prove the second assertion.  Given $x \in \tilde \cB_j \cup \tilde \cS_j \cup \cG_j$, for $j \geq 1$, then by construction and item 2 ``graphicality'' there exists an $x'(x) \in T_{i-1}$ so that $\norm{x'(x) - x} \leq r_j/30 + c(k,\rho_X, \chi) \delta \er_{j-1}$.  The uniform estimates \eqref{eqn:uniform-C0-est} on the $\sigma_i$ imply that $x'' = \sigma_{i+1} \circ \sigma_i \circ \cdots  \circ \sigma_j(x') \in T_{i+1}$ satisfies $\norm{x''(x) - x} \leq \er_j/30 + c \delta \er_{j-1}$.  Therefore, for $\delta(k, \rho_X, \chi)$ sufficiently small we obtain $d(x, T_{i+1}) \leq \er_j/20 \leq r_x/20$.

\subsection{Item \ref{iit_5}: Radius control}

Since $R_i \cup E_i \subset R_{i+1} \cup E_{i+1}$, by our inductive hypothesis it suffices to prove ``radius control'' when $r_b = \er_{i+1}$.  Suppose $s \in \cS \cap B_{r_b}(b)$ and $r_s \geq r_b \equiv \er_{i+1}$.  If $r_s \geq \er_i$, then by induction $s \in E_i \cup R_i$ and we are done.  Otherwise, $\er_{i+1} \leq r_s < \er_i$, and we can WLOG assume $s \not\in E_{i+1}$.

By construction, $b \in B_{r_g}(g)$ for some $g \in \cG_{i-1}$, and therefore $s \in B_{r_b}(b) \subset B_{1.5r_g}(g)$.  Therefore by definition \eqref{eqn:S_i-defn} and our assumptions on $s$, we must have $s \in \bigcup_{s' \in \cS_{i+1}} B_{r_{s'}}(s') \subset R_{i+1}$.

This establishes item \ref{iit_5}, since the proof for good balls is verbatim.

\subsection{Item \ref{iit_6}: Packing control}

For the packing control, we are going to use the bi-Lipschitz estimates on the manifolds $T_j$ and the disjointness properties of balls in our construction. In particular, we know that $T_{i+1}$ is $(1+c(k,\rho_X,\chi)\delta^\alpha)$ bi-Lipschitz to $V(0,1)$.  Moreover, since we have the uniform estimates \eqref{eqn:uniform-C0-est}, we also know that $T_{i+1} \cap \B 1 0$ is bi-Lipschitz to a subset of $V(0,1)\cap \B 2 0$. 

For all $s\in \cS_{i+1}$, let $s'(s) \in T_{i+1}$ be a point satisfying $\norm{ s'(s) - s } \leq r_s/20$, and in a similar way $b'(b) \in T_{i+1}$ satisfies $\norm{ b' - b } \leq r_b/20$, and $g'(g) \in T_{i+1}$ satisfies $\norm{ g' - g} \leq \er_{i+1}/20$. By construction, all the balls in the collection
\begin{gather}
 \cur{\B{r_s/7}{s'(s)}}_{s\in \cS_{i+1}}\cup\cur{\B{r_b/7}{b'(b)}}_{b\in \cB_{i+1}}\cup \cur{\B{\er_{i+1}/7}{g'(g)}}_{g\in \cG_{i+1}}\, 
\end{gather}
are pairwise disjoint.

Using the map $\tau_{i+1}^{-1}$ and its bi-Lipschitz estimates, we obtain that all the balls in the collection
\begin{gather}
 \cur{\B{r_s/10}{\tau^{-1}_{i+1}(s'(s))}}_{s\in \cS_i}\cup\cur{\B{r_b/10}{\tau^{-1}_{i+1}(b'(b))}}_{b\in \cB_i}\cup \cur{\B{\er_{i+1}/10}{\tau^{-1}_{i+1}(g'(g))}}_{g\in \cG_i} 
\end{gather}
are pairwise disjoint inside the $k$-dimensional affine ball $T_0 \cap B_3(0)$, and now the desired packing control is a corollary of Lemma \ref{lemma_Banach_packing}.

\subsection{Item \ref{iit_7}: Covering control}

It is clear from item ``radius control'' that
\begin{gather}
B_1(0) \cap \cS \subset E_i \cup R_i \cup \bigcup_{g \in \cG_i} \left[ B_{\er_i}(g) \cap \{ s : r_s < \er_i \} \right]\, .
\end{gather}
To prove item ``covering control'' it will therefore suffice to establish
\begin{gather}\label{eqn:covering-suffice}
\mu(E_i) \leq c(k, \chi) \delta^{\alpha}\, .
\end{gather}

First of all, note that by definition \eqref{eqn:beta-defn} of $\beta$, we get that for all $j \geq 0$:
\begin{gather}\label{eq_exc_a}
 \mu\ton{\tilde E_{j+1}} \leq c(k,\chi) \er_j^k\sum_{g \in \cG_j} \beta(g,\er_j)^2\, .
\end{gather}

We want to control the RHS with an integral wrt the $\cH^k$ Hausdorff measure on $T_{i+1}$.  For each fixed $0 \leq j \leq i$, using item 4 ``ball control'' we know the balls $\{B_{\er_j/5}(g)\}_{g \in \cG_j}$ are pairwise disjoint, and for each $g \in \cG_j$ we have a $g'(g) \in T_{i+1}$ with $\norm{g'(g) - g} \leq \er_j/20$.  Therefore, the collection $\cur{\B{\er_i/7}{g'(g)}}_{g\in \cG_j}$ are pairwise disjoint also.

Since $T_{i+1}$ is $(1+c\delta^\alpha)$-bi-Lipschitz to a $k$-dimensional plane, and by Lemma \ref{lemma_Hk_balls}, we have that for all $g \in \cG_j$
\begin{gather}
 c(k)^{-1}\er_j^k\leq \cH^k\ton{\B{\er_j/7}{g'(g)} \cap T_{i+1}}\leq c(k) \er_j^k\, .
\end{gather}
Moreover, by \eqref{eq_beta_yx}, we know that for all $y\in \B{\er_i/7}{g'(g)}\subset \B{\er_i/5}{g}$
\begin{gather}
 \beta^k_\mu(g,\er_j)\leq c(k)\beta^k_\mu(y,2\er_j)\, .
\end{gather}
Summing up all of these estimates, we get that
\begin{gather}
 \mu \ton{\tilde E_{j+1}} \leq c(k,\chi) \int_{\bigcup_{g\in \cG_j} \B{\er_j/7}{g'(g)}\cap T_{i+1}} \beta(y,2\er_j)^2 d\cH^k(y)\leq  c(k,\chi) \int_{B_{2\er_j}(\cG_j) \cap T_{i+1}} \beta(y,2\er_j)^2 d\cH^k(y)\, .
\end{gather}

Take a $y \in B_2(0)$, and let $m$ be the maximal integer $\leq i$ for which $y \in B_{2\er_m}(\cG_m)$.  Since $B_{2\er_0}(\cG_0) = B_2(0)$, $m \geq 0$.  Then from \eqref{eqn:beta-S-shifting} and our assumption \eqref{eqn:covering-hyp} we have
\begin{gather}\label{eqn:staggered-beta-sum}
\sum_{j=0}^i 1_{B_{2\er_j}(\cG_j)}(y) \beta(y, 2\er_j)^2 \leq c(k, \chi) \int_{\er_m}^\infty \beta(y, s)^2 \frac{ds}{s} \leq c(k, \chi) \int_{\er_m}^\infty \beta(y, s)^\alpha \frac{ds}{s} \leq c(k, \chi) \delta^\alpha.
\end{gather}

We can use \eqref{eqn:staggered-beta-sum} to sum contributions to $E_{i+1}$ over scales, and end up with
\begin{align}
 \mu\ton{E_{i+1}} \leq \sum_{j=0}^{i} \mu \ton{\tilde E_{j+1}} &\leq c(k,\chi) \int_{T_{i+1} \cap \B 2 0 } \ton{ \sum_{j=0}^i 1_{B_{2\er_j}(\cG_j)}(y) \beta(y, 2\er_j)^2 } d\haus^k(y) \\
 &\leq c(k, \chi) \delta^\alpha \haus^k(T_{i+1} \cap \B 2 0)\, .
\end{align}
Using that $T_{i+1}$ is $(1+c\delta^\alpha)$-bi-Lipschitz to $V(0, 1)$, and Lemma \ref{lemma_Hk_balls}, we conclude \eqref{eqn:covering-suffice}.  This establishes item ``covering control.''

\subsection{Finishing the proof of Lemma \ref{lem:main-covering}}
The proof of the lemma is now just a corollary of the inductive covering.  We can define
\begin{gather}
 \cS_+' = \bigcup_{i=0}^\infty \cS_i\, , \quad \cB =\bigcup_{i=0}^\infty \cB_i\, ,\quad \tau=\lim_i\tau_i \, ,
\end{gather}
where the last limit exists as the $\tau_i$ are uniformly Cauchy (by e.g. \eqref{eqn:small-movement}).  We obtain that $\tau$ is a bi-Lipschitz map with the desired estimates because the bi-Lipschitz estimates in item \eqref{iit_3} are independent of $i$, and packing control of \eqref{eq_packing} follows directly from the estimate \eqref{eq_packing_proof} of item \eqref{iit_6}. The bad ball structure is simply the definition of a bad ball in \ref{deph_bad_balls}. 

We just need to establish the measure bound \eqref{eq_measure_control}.  By ``ball control'' (item \eqref{iit_4}), we know that for all $i$,
\begin{gather}
 \bigcup_{g \in \cG_i} \left[B_{\er_i}(g) \cap \{s \in \cS : r_s < \er_i \} \right] \subseteq \B {2\er_i}{T_i}\cap \{s \in \cS : r_s < \er_i \} \, .
\end{gather}
Therefore, by ``covering control'' (item \eqref{iit_7}), we get for every $i$:
\begin{gather}\label{eqn:finishing-covering-mu}
\mu \qua{ B_1(0) \setminus \ton{ \left[ B_{2\er_i}(T_i) \cap \{s \in \cS : r_s < \er_i \} \right] \cup \bigcup_{s \in \cS_i \subset \cS} B_{r_s}(s) \cup \bigcup_{b \in \cB_i \subset \cB} \left[ B_{r_b}(b) \cap \{ s \in \cS: r_s < r_b \} \right] }} \leq c(k, \chi) \delta^\alpha\,  .
\end{gather}

Since this estimate is independent of $i$, and 
\begin{gather}
 \bigcap_{i=0}^\infty \B {2\er_i}{T_i} \subset \tau(B_3(0) \cap \left[ p(0, 1) + V(0, 1) \right] ) \, , \quad \bigcap_{i=0}^\infty \{s \in \cS : r_s < \er_i \} = \cS_z\, ,
\end{gather}
we get the desired result.

\section{Corollaries}

In this Section we complete the proofs of the various corollaries of the Main Theorem \ref{thm:main-packing}. 

We start with Corollary \ref{cor:discrete}. Here we basically choose the radius function $r_s$ for the covering $\cS_+$ in a clever way and apply the Main Theorem.
\begin{proof}[Proof of Corollary \ref{cor:discrete}]

Fix an $\overline{r} \in (0, 1)$.  For case A) define $\cS_{\overline{r}} = \{ s : r_s \geq \overline{r} \}$ and $\mu_{\overline{r}} = \mu \llcorner \cS_r$.  We claim that $\mu_{\overline{r}}$ is finite.  From the definition of $\beta^k_\mu$ we have a $k$-plane $p + V^k$ so that
\begin{gather}\label{eqn:bounded-far-from-V}
\mu_{\overline{r}}( B_1(0) \setminus B_{\overline{r}/20}(p + V)) \leq c(\overline{r}, k)M\, .
\end{gather}
On the other hand, using the definition of $\cS_{\overline{r}}$ we have by Lemma \ref{lemma_Banach_packing} that
\begin{gather}
\mu_{\overline{r}}(B_1(0) \cap B_{\overline{r}/20}(p + V)) \leq \sum \{ r_s^k : s \in \cS_{\overline{r}} \text{ and } d(s, p + V) < r_s/10 \} \leq c(k)\, .
\end{gather}

So $\mu_{\overline{r}}$ is finite, and thus Borel regular (see for example \cite[theorem II, 1.2, pag 27]{parthasarathy}).   This and the monotonicity of $\beta_\mu$ wrt $\mu$ imply that we can find a Borel set $U$ so that
\begin{gather}\label{eq_mu_U}
\int_0^2 \beta_{\mu_{\bar r}}^k(x, r)^\alpha \frac{dr}{r} \leq M^{\alpha/2} \quad \forall x \in U\, ,
\end{gather}
and $\mu_{\bar r}(B_1(0) \setminus U) \leq \Gamma$. 

By monotonicity of $\beta$, $\mu_{\overline{r}}\llcorner U$ and $\cS_{\overline{r}}$ satisfy the requirements of Theorem \ref{thm:main-packing}.  Therefore we have some $\cS'_{\overline{r}}$ so that
\begin{gather}\label{eq_mu_bound}
\mu_{\overline{r}}(B_1(0)) \leq b \sum_{s' \in \cS'_{\overline{r}}} r_{s'}^k + c(k,\rho_X)M +\Gamma \leq c(k, \rho_X)(M + b)+\Gamma \, .
\end{gather}
Since $\cup_{\overline{r} > 0} \cS_{\overline{r}}$ covers $\mu$-a.e., the required bound follows taking $\overline{r} \to 0$.

Similarly, in case B) define
\begin{gather}
\cS = \{ x : \Theta_*^k(\mu, x) \leq b \}\,  , 
\end{gather}
and set $r_s\in (0,1)$ to be any choice of radius for which $\mu(B_{5r_s}(s)) \leq 20^k b r_s^k$.  Take $p + V$, $\cS_{\overline{r}}$, and $\mu_{\overline{r}}$ as in part A).  By assumption, we have $\cup_{\overline{r}} \cS_{\overline{r}}$ covers $\mu$-a.e.

We must demonstrate $\mu_{\overline{r}}$ is finite.  Let $\{B_{r_s}(s)\}_{\tilde \cS_{\overline{r}}}$ be a Vitali cover of $\{B_{r_s}(s) : s \in \cS_{\overline{r}} \cap B_1(0) \cap B_{\overline{r}/20}(p + V) \}$.  Then, using the definition of $r_s$ and Lemma \ref{lemma_Banach_packing}, we have
\begin{gather}\label{eqn:mu-bound-from-packing}
\mu_{\overline{r}}(B_1(0) \cap B_{\overline{r}/20}(p + V)) \leq \sum_{s \in \tilde \cS_{\overline{r}}} \mu(B_{5r_s}(s)) \leq \sum_{s \in \tilde \cS_{\overline{r}}} 20^k b r_s^k \leq c(k) b\, .
\end{gather}
By the same argument as in \eqref{eqn:bounded-far-from-V} we have $\mu_{\overline{r}}(B_1(0)) < \infty$, and thus Borel-regular. So, as in part A), we can find a set $U$ with \eqref{eq_mu_U} and $\mu_{\bar r}\ton{B_1(0)\setminus U}\leq \Gamma$.

So $\mu_{\overline{r}}\llcorner U$ and $\cS_{\overline{r}}$ satisfy the requirements of Theorem \ref{thm:main-packing}, by an analogous computation to \eqref{eq_mu_bound} we deduce the required bound for $\mu_{\overline{r}}$.  Since this bound is independent of $\bar r$ and $\mu_{\overline{r}}\nearrow \mu$, we obtain the claim.

We prove case C).  Fix $p + V$ as above, and now in this case define
\begin{gather}
\mu_{\overline{r}} = \mu \llcorner (B_1(0) \setminus B_{\overline{r}}(p + V)) \leq b\haus^k \llcorner (S \setminus B_{\overline{r}}(p + V))\, .
\end{gather}
From \eqref{eqn:bounded-far-from-V} each $\mu_{\overline{r}}$ is finite, and hence Borel-regular.  A standard argument (see e.g. chapter 1 in \cite{simon:gmt}) shows that if
\begin{gather}
A = \{ x : \Theta^{*,k}(\mu_{\overline{r}}, x) > t \},
\end{gather}
then $t \haus^k(A) \leq \mu_{\overline{r}}(A)$.  Therefore we must have the density bounds
\begin{gather}
\Theta^{*, k}(\mu_{\overline{r}}, x) \leq b \quad \text{ for $\mu_{\overline{r}}$-a.e. $x$}\, ,
\end{gather}
Using part B), then taking $\overline{r} \to 0$, we deduce
\begin{gather}
\mu(B_1(0) \setminus (p + V)) \leq c(k,\rho_X)(b + M)\, .
\end{gather}
Since from Lemma \ref{lemma_Hk_balls} we have $\mu(B_1(0) \cap (p + V)) \leq b \haus^k(B_1(0) \cap (p + V)) \leq c(k) b$, we conclude.
\end{proof}

\vspace{5mm}

\subsection{Rectifiability}
Now we are ready to prove Theorem \ref{thm:rect} about rectifiability criteria for the measure $\mu$. This proof follows from the Covering Lemma and some considerations. First of all, fix any $\B r x\subset \B 1 0$. We can consider the trivial covering $\cS=\cS_z=\B r x$ for this ball, and the Covering Lemma \ref{lem:main-covering} tells us that if we define \eqref{eq_measure_control}: 
\begin{gather}
F = \left( \cS_z \cap \tau(B_3(0^k)) \right) \cup \bigcup_{b \in \cB} \left[ B_{r_b}(b) \cap \{ s \in \cS : r_s < r_b \} \right]\, ,
\end{gather}
this set covers most all of $\B r x$ up to a set of small $\mu$ measure. 

The point now is to make sure that a fixed portion of the measure $\mu$ in $\B 1 0$ will be covered by the first part of the covering, i.e., by the set $\tau(B_3(0^k))$, which is clearly rectifiable. In other words, we need to make sure that the ``bad balls'' $B_{r_b}(b)$ and the set not covered by $F$ do not carry much portion of the measure. This is the main part of this proof, and it requires lower density bounds to ensure that we can pick balls $\B r x$ that have enough measure $\mu$. Once this is done a standard inductive procedure can be used to cover a set of full measure with countably many Lipschitz images.

 Notice that this is the only place where the lower bound on the upper density $\Theta^{*, k}(\mu, x) > 0$ plays a role. Notice also that this assumption is necessary to ensure rectifiability. Indeed, consider for example the $n$-dimensional Lebesgue measure $\lambda^n$ in $\R^n$. For $k<n$, this measure clearly satisfies 
 \begin{gather}
  \int_0^2 \beta^k_{\lambda^n}(x, r)^2 \frac{dr}{r} < \infty\, , \quad \Theta^k_*(\lambda^n, x) < \infty\, 
 \end{gather}
for all $x\in \R^n$. Indeed for all $x$, $\Theta^k_*(\lambda^n, x)=\Theta^{*,k}(\lambda^n, x)=\Theta^k(\lambda^n, x)=0$, but clearly $\lambda^n$ is not $k$-rectifiable.

\begin{proof}[Proof of Theorem \ref{thm:rect}]
The argument is very similar to the ones in \cite[section 10]{ENV}.  For the reader's convenience we sketch the argument here.  

First, we prove our theorem under the stronger assumptions that $\mu$ is finite and 
\begin{gather}\label{eqn:strong-rect-hyp}
\int_0^2 \beta^k_\mu(x, r)^\alpha \frac{dr}{r} \leq M^{\alpha/2}\,  , \quad \Theta^k_*(\mu, x) \leq b, \quad \Theta^{*,k}(\mu, x) \geq a\, , 
\end{gather}
with $a,b,M$ positive and finite. We will turn to the general case afterwards. 

Applying Corollary \ref{cor:discrete} at every scale we deduce
\begin{gather}\label{eqn:rect-upper-bound}
\mu(B_r(x)) \leq c(k, \rho_X)(M + b) r^k =: \Gamma r^k \quad \forall x \text{ and } \forall r < 1\, .
\end{gather}
Note this implies $\mu << \haus^k$.

By Lemma \ref{lemma_technical}, given any $\delta > 0$, then for $\mu$-a.e. $x$ there is a scale $R_x$ so that
\begin{gather}\label{eq_claim_rect}
\mu \left( z \in B_r(x) : \int_0^\infty \beta^k_{\mu \llcorner B_r(x)}(z, s)^\alpha \frac{ds}{s} > \delta \right) \leq \delta r^k \quad \forall 0 < r < R_x\, .
\end{gather}

Let us take any such $x$ and $r < R_x$, and by the above we can find a Borel set $A \subset B_r(x)$ so that
\begin{gather}
\int_0^\infty \beta^k_{\mu \llcorner A}(z, s)^\alpha \frac{ds}{s} < \delta, \quad \mu(B_r(x) \setminus A) \leq \delta r^k\, .
\end{gather}
Ensuring $\delta \leq \delta_0(k,\rho_X, \chi)$, we can apply the Covering Lemma \ref{lem:main-covering} to $\mu \llcorner A$, with cover $\cS_z = A$, $\cS_+ = \emptyset$, to obtain a Lipschitz mapping $\tau : B_3 \to X$ and a family of bad balls $\cB$, so that
\begin{gather}\label{eqn:rect-measure-packing}
\mu \left[ A \setminus \left( \tau(B_3) \cup \bigcup_{b \in \cB} B_{r_b}(b) \right) \right] \leq c(k, \chi, \rho_X) \delta\, , \quad \text{ and } \quad \sum_{b \in \cB} r_b^k \leq c(k)\, .
\end{gather}

For each bad ball $B_{r_b}(b)$, we can follow the argument from Section \ref{sec:packing-induct}, and use upper bound \eqref{eqn:rect-upper-bound}, to obtain
\begin{gather}\label{eqn:rect-small-bad}
\mu(B_{r_b}(b)) \leq (c(k,\rho_X,\chi) \delta^2 + 1/10 + c_B(k) \chi \Gamma) r_b^k\, .
\end{gather}
Choose $\chi = \min(1/100, 1/\Gamma)$, then taking $\delta(k, \rho_X, \chi)$ sufficiently small, we can combine \eqref{eqn:rect-measure-packing} with \eqref{eqn:rect-small-bad} and our definition of $A$ to obtain
\begin{gather}
\mu(B_r(x) \setminus \tau(B_3)) \leq c_6(k) r^k\, ,
\end{gather}
for some constant $c_6(k)$ which is \emph{independent} of $M$, $b$, $a$.

In particular, by scaling $\mu$ and correspondingly readjusting $\chi$, $\delta$, we can assume $a \geq 10 c_6$.  Then a straightforward argument using the above conclusions shows that, for any closed set $C$, we can find finitely many Lipschitz mappings $\tau_1, \ldots, \tau_N : B_3(0) \subset \R^k \to X$, so that
\begin{gather}
\mu(B_1(0) \setminus (C \cup \tau_1(B_3) \cup \cdots \cup \tau_N(B_3))) \leq \frac{1}{2} \mu(B_1(0) \setminus C)\, .
\end{gather}
Rectifiability for $\mu$ satisfying \eqref{eqn:strong-rect-hyp} now follows directly.

\vspace{3mm}
In order to conclude the proof, we show that the assumptions that $\mu$ is finite and \eqref{eqn:strong-rect-hyp} holds instead of \eqref{eqn:rect-hyp} are not restrictive.

First, we show that we can assume wlog that $\mu$ is finite, and thus also Borel-regular since $X$ is a metric space. Indeed, let $x\in \B 1 0$ be such that 
\begin{gather}
 \beta^k_\mu(x,2)^\alpha<\infty\, ,
\end{gather}
and consider a $k$-dimensional affine plane $p+V$ with 
\begin{gather}
 \int_{\B 2 x} d(y,p+V)^\alpha =\int_{\B 1 0 } d(y,p+V)^\alpha <\infty\, .
\end{gather}
Then automatically for all $\bar r>0$ the measure $\mu$ restricted to the open set $O_{\bar r}=\overline{\B {\bar r}{p+V}}^C$ has finite mass. Moreover, by monotonicity of $\beta$ wrt $\mu$ and since $O_{\bar r}$ is open, $\mu\llcorner O_{\bar r}$ satisfies all the assumptions of \eqref{eqn:rect-hyp} and it is finite.

Note also that the measure $\mu\llcorner (p+V)$ is rectifiable. Indeed, let 
\begin{gather}
 A_i=\cur{\Theta_*^k(\mu,x)<i}\cap (p+V)\cap \B 1 0\, , \quad \mu_i =\mu\llcorner (p+V)\cap A_i\, .
\end{gather}
We claim that $\mu_i(\B r x)\leq ci r^k$ for all $x,r$, and thus $\mu\llcorner (p+V) = \lim_i \mu_i <<\cH^k\llcorner (p+V)$. In order to show that $\mu_i(\B r x)\leq ci r^k$, let $\B { r_j}{x_j}$ be a covering of $\B r x\cap A_i$ with $x_j\in A_i$, $\mu_i \ton{\B { r_j}{x_j}}\leq 2\omega_k i r_j^k$ and $\B { r_j/5}{x_j}$ pairwise disjoint. Since $x_j\in (p+V)\cap \B 1 0$, $\sum_j (r_j/5)^k\leq c$, and so $\mu_i(\B r x)\leq ci r^k$ as wanted.

Thus we can write 
\begin{gather}
 \mu = \mu\llcorner (p+V) +\lim_{i\to \infty} \mu\llcorner O_{i^{-1}}\, ,
\end{gather}
and so if the finite measure $\mu\llcorner O_{\bar r}$ is rectifiable for all $\bar r>0$, we obtain that the original $\mu$ is countably rectifiable also.

\vspace{3mm}

As for the stronger hypothesis \eqref{eqn:strong-rect-hyp}, we have the following.  Given a finite $\mu$, for any integer $i$, define
\begin{gather}
 U_i = \left\{ x \in B_1(0) : \int_0^2 \beta^k_\mu(x, r)^\alpha \frac{dr}{r} \leq i, \quad \Theta^k_*(\mu, x) \leq i, \quad \Theta^{*, k}(\mu, x) \geq i^{-1} \right\}\, .
\end{gather}
By assumption, $\cup_i U_i$ covers $\mu$-a.e. $x$. Moreover, $\mu\llcorner U_i$ obviously satisfies
\begin{gather}\label{eq_muibounds}
 \int_0^2 \beta^k_{\mu\llcorner U_i}(x, r)^\alpha \frac{dr}{r} \leq i, \quad \Theta^k_*(\mu\llcorner U_i, x) \leq i\, .
\end{gather}
We claim that $\Theta^{*,k}(\mu \llcorner U_i, x) \geq 10^{-k} i^{-1}$ for $\mu$-a.e. $x \in U_i$.  Given this claim and the previous bounds \eqref{eq_muibounds}, our initial proof will show that each $\mu\llcorner U_i$ is $k$-rectifiable, and hence $\mu$ is $k$-rectifiable also.

Let us prove our claim.  The proof is standard, but we include it for the reader's convenience.  When $k = 0$ the claim is trivial.  Otherwise, set
\begin{gather}
A = \{ x \in U_i : \Theta^{*, k}(\mu \llcorner U_i, x) < 10^{-k} i^{-1} \}\, .
\end{gather}

Suppose, towards a contradiction, that $\mu(A) > 0$.  Since $\mu$ is finite Borel-regular, we can choose an open $V \supset A$ so that $\mu(V) \leq (11/10)\mu(A)$.  For $\mu$-a.e. $x \in A$, pick a radius $r_x$ so that:
\begin{gather}
 B_{r_x}(x) \subset V, \quad \frac{\mu(B_{r_x}(x) \cap A)}{\omega_k r_x^k} \leq 10^{-k} i^{-1}, \quad \frac{\mu(B_{r_x/5}(x))}{\omega_k (r_x/5)^k} \geq (9/10) i^{-1}\, .
\end{gather}

Let $\{B_{r_{x_i}}(x_i)\}_i$ be a Vitali cover of $\{B_{r_x}(x)\}_{x \in A}$, so that the $r_{x_i}/5$-balls are disjoint.  This collection is countable, since each ball has a positive amount of measure.  Then we have the contradiction
\begin{align}
\mu(A)
&\leq \sum_i 10^{-k} i^{-1} \omega_k r_{x_i}^k \leq 2^{-k} (10/9) \sum_i \mu(B_{r_{x_i}/5}(x_i)) \leq 2^{-k} (10/9) \mu(V) < \mu(A)\,  .
\end{align}
Therefore we must have $\mu(A) = 0$.

\vspace{3mm}

  This completes the proof of our claim, and in turn the proof of Theorem \ref{thm:rect}.
\end{proof}

Now we turn our attention to Corollary \ref{cor:rect}, which is just a special case of the previous Theorem \ref{thm:rect}.
\begin{proof}[Proof of Corollary \ref{cor:rect}]
Take $\overline{r} > 0$.  
By our assumption there is an affine $k$-plane $p + V$ so that
\begin{gather}
\haus^k(S \setminus B_{\overline{r}}(p + V)) < \infty\, .
\end{gather}
Define
\begin{gather}
S_{\overline{r}} = \left\{ x \in B_1(0) : d(x, p + V) \geq \overline{r} \quad \text{and} \quad  \int_0^\infty \beta^k_\mu(x, r)\frac{dr}{r} \leq 1/\overline{r} \right\}\, .
\end{gather}

Then $\haus^k \llcorner S_{\overline{r}}$ is finite, and hence we have density bounds
\begin{gather}
2^{-k} \leq \Theta^{*, k}(\haus^k \llcorner S_{\overline{r}}, x) \leq 1 \quad \text{ for $\haus^k$-a.e. $x \in S_{\overline{r}}$}\, .
\end{gather}
By construction and monotonicity of $\beta$, $\haus^k \llcorner S_{\overline{r}}$ satisfies the requirements of Theorem \ref{thm:rect}, and so we deduce $S_{\overline{r}}$ is $k$-rectifiable.

From our hypotheses $\cup_{\overline{r}} S_{\overline{r}} = S \setminus (p + V)$ up to a set of $\haus^k$-measure $0$.  Since $p + V$ is trivially $k$-rectifiable, we finish the proof taking $\overline{r} \to 0$.
\end{proof}

\subsection{Proof of Proposition \ref{thm:improved-reif}}


Now we turn to Proposition \ref{thm:improved-reif}, which is a corollary of the proof of the main Theorem. Actually, the construction is much simplified in this case.
\begin{remark}\label{rem_cite}
 Before we sketch the proof of this result, it is worth noticing that up to making sure that the constants involved in the estimates are independent of the ambient dimension $n$, and up to using the notion of almost projections/canonical projections on Banach spaces and the relative estimates studied in Section \ref{sec_preliminaries}, the proof of this theorem is very similar to the proof of \cite[main theorem]{toro:reifenberg},\cite{davidtoro}. In the language of our proofs, the Reifenberg flat condition allows us to completely skip the good balls - bad balls construction and makes the inductive covering of Lemma \ref{lem:main-covering} technically less involved.
\end{remark}

For this entire section, let us fix $S$ to be a $(k, \delta)$-Reifenberg flat set having $0 \in S$, as per Theorem \ref{thm:improved-reif}.  The proof is essentially standard.

Let us review some basic properties of the $\beta_\infty$.  First, we trivially have $\beta_\infty(x, r) \leq \delta$ for any $x \in S$, by the Reifenberg-flat assumption.  Second, if $B_r(x) \subset B_R(y)$, then $\beta_\infty(x, r) \leq (R/r) \beta_\infty(y, R)$.  In particular, we have
\begin{gather}
\beta^k_{S, \infty}(x, r) \leq c(k) \int_r^{2r} \beta^k_{S, \infty}(x, s) \frac{ds}{s}\, .
\end{gather}

\begin{definition}
Given $x \in S$, let us define $V_\infty(x, r)$ to be any $k$-plane for which
\begin{gather}
S \cap B_r(x) \subset B_{2\beta_\infty(x, r) r}(x+V_\infty(x, r) )\, .
\end{gather}
\end{definition}

Similar to how the $L^2$-$\beta$-numbers control tilting between nearby good balls, the $L^\infty$-$\beta$-numbers control tilting between nearby Reifenberg-flat balls.  The proof is identical, except we use the Reifenberg-flat condition to obtain points in $S$ in general position, and require no lower mass bounds.

\begin{lemma}\label{lem:tilting-reif}
Let $x, x', y \in S$, and suppose $B_r(x) \cup B_{r'}(x') \subset B_{R/2}(y)$, and $B_R(y) \subset B_2(0)$.  Then we have
\begin{gather}\label{eqn:tilting-dH-reif}
d_H( (x + V_\infty(x, r)) \cap B_R(y), (x' + V_\infty(x', r')) \cap B_R(y)) \leq c(k, r/R, r'/R) \beta^k_{S, \infty}(y, R) R\, ,
\end{gather}
and
\begin{gather}\label{eqn:tilting-dG-reif}
d_G(V_\infty(x, r), V_\infty(x', r')) \leq c(k, r/R, r'/R) \beta^k_{S, \infty}(y, R)\, .
\end{gather}

Similarly, we have
\begin{gather}\label{eqn:beta-plane-reif}
d_H( (x + V_\infty(x, r)) \cap B_R(y), S \cap B_R(y)) \leq c(k, r/R) \delta R\, .
\end{gather}
\end{lemma}
\begin{remark}
 Although phrased differently, a similar lemma is present in the proof of \cite[lemma 3.1]{toro:reifenberg}.
\end{remark}

\begin{proof}
Provided $\delta(k)$ is sufficiently small, the Reifenberg-flat condition and stability Lemma \ref{lemma_GP_stable} imply we can find points $x_0=x$ and $x_1, \ldots, x_k \in S \cap B_{9r/10}(x)$ so that the vectors $\{x_i - x_0\}_{i=1}^k$ lie in $r/2$-general position.

For each $i = 0, \ldots, k$ we have
\begin{gather}
d(x_i, (x + V_\infty(x, r))) \leq \beta_\infty(x_i, r) r,\quad \text{and} \quad d(x_i, (y + V_\infty(y, R))) \leq \beta_\infty(y, R) R\, .
\end{gather}
Therefore, ensuring $\delta(k)$ is sufficiently small, we can use the stability Lemma \ref{lemma_GP_stable} to find $z_i \in (x + V_\infty(x, r)) \cap B_r(x)$ such that $||x_i - z_i|| \leq 2\delta r$, and the vectors $\{z_i - z_0\}_{i=1}^k$ lie in $r/4$-general position.  Lemma \ref{lemma_GP_bounds} implies that
\begin{gather}
d(z, y + V_\infty(y, R)) \leq c(k,r/R) \beta_\infty(y, R) R \quad \forall z \in (x + V_\infty(x, r))\cap B_R(y)\, .
\end{gather}
Now use Lemma \ref{lemma_hdv}, and repeat the argument with $B_{r'}(x')$, and the desired estimates \eqref{eqn:tilting-dH-reif}, \eqref{eqn:tilting-dG-reif} follow from the triangle inequality.

Let us prove \eqref{eqn:beta-plane-reif}.  Fix a $k$-plane $W$ so that $d_H((y + W) \cap B_R(y), S \cap B_R(y)) < 2\delta R$.  We have by our choice of $z_i$ that
\begin{gather}
d(z_i, y + W) \leq 4 \delta R \quad i = 0, \ldots, k\, .
\end{gather}
Therefore, as above, lemmas \ref{lemma_GP_bounds} and \ref{lemma_hdv} imply that
\begin{gather}
d_H( (x + V_\infty(x, r)) \cap B_R(y), (y + W) \cap B_R(y)) \leq c(k, r/R) \delta R\, ,
\end{gather}
and \eqref{eqn:beta-plane-reif} follows by the triangle inequality.
\end{proof}

\vspace{5mm}

\subsubsection{Construction}

We build the map $\tau$ as a limit of maps $\tau_i$, constructed in a very similar manner to Section \ref{sec:proof-covering}.  The proof that each $\tau_i$ has the required bi-H\"older/bi-Lipschitz bounds is essentially verbatim to items 2 and 3 in Section \ref{sec:proof-covering}.

We shall inductively define a sequence of mappings $\tau_i : V(0,1) \to X$, and manifolds $T_i = \tau_i(V(0, 1))$, which admit the following properties:
\begin{enumerate}

\item \label{it_1} $T_0 = V(0, 1)$.

\item \label{it_2} Graphicality of $T_i$: for any $y \in T_i$, there is an $k$-dimensional affine plane $p + V$ (depending on $y$), so that for any choice of almost-projection $\pi_V$ to $V$, we have
\begin{gather}
T_i \cap B_{2\er_i}(y) = \graph_{\Omega, \pi_V}(f), \quad (2\er_i)^{-1} ||f|| + \lip(f) \leq \Lambda \delta, \quad B_{1.5 \er_i}(y) \cap (p + V) \subset \Omega \subset (p + V)
\end{gather}
Moreover, if there exists some $g \in \cG_i \cap B_{10\er_i}(y)$, then we can take $p + V = p(g, \er_i) + V(g, \er_i)$.

\item \label{it_3} Each map $\tau_i : V(0, 1) \to T_i$ is a $(1+c(k,\chi)\delta)$-bi-H\"older equivalence.

\item \label{it_4} Given summability condition \eqref{eqn:reif-hyp}, then in fact each $\tau_i$ is a bi-Lipschitz equivalence, with bound
\begin{gather}
e^{-c(k,\rho_X)Q^\alpha} ||x - y|| \leq ||\tau_i(x) - \tau_i(y)|| \leq e^{c(k,\rho_X) Q^\alpha}||x - y||.
\end{gather}

\item \label{it_5} Covering control: We have $d_H(S \cap B_{1+\er_i/2}, T_i \cap B_{1+\er_i/2}) \leq \er_i$.

\end{enumerate}
Given items \ref{it_1})-\ref{it_5}), the Reifenberg Theorem \ref{thm:improved-reif} will follows directly.

\vspace{5mm}
Let us detail the construction of the $\tau_i$ and $T_i$.  Recall that $\er_i = \chi^i$, where here we shall fix $\chi = 1/100$.

For each $i$ define $\cG_i$ to be a maximal $2\er_i/5$-net in $S \cap B_1(0)$, so that the balls $\{B_{\er_i}(g)\}_{g \in\cG_i}$ cover $S \cap B_1(0)$, and the balls $\{B_{\er_i/5}(g)\}_{g \in \cG_i}$ are disjoint.   Given $g \in \cG_i$, let $V_g = V_\infty(g, \er_i)$, and $\pi_g$ be a choice of almost-projection to $V_g$.  Let $\{\phi_g\}_{g \in \cG_i}$ be the truncated partition of unity subordinate to $\{B_{\er_i}(g)\}_{g \in \cG_i}$, as per Lemma \ref{lem:pou}.

We now define
\begin{gather}
\sigma_i = x - \sum_{g \in \cG_i} \phi_g(x) \pi_g(x - g)\, ,
\end{gather}
and set $\tau_i = \sigma_i \circ \cdots \circ \sigma_1$, and $T_i = \tau_i(T_0) \equiv \tau_i(V_\infty(0, 1))$.

This completes the construction of the $\tau_i$ and $T_i$.  In the following subsections we prove by induction the properties 2)-5).  We can assume by inductive hypotheses that items 2)-5) hold for scales $r_0, \ldots, \er_i$.

\subsubsection{Item \ref{it_2}: Graphicality}

The proof is the same as Section \ref{sec:graph}, except we use the $\beta_\infty$ instead of $\beta$, and tilting Lemma \ref{lem:tilting-reif} in place of Lemma \ref{lem:good-ball-tilting-2}.  Let us sketch the proof.  In this section $c$ denotes a constant depending on $k$, but independent of $\Lambda$, and we assume $\delta(k)$ is sufficiently small so that $c(1+\Lambda)\delta < \eps_1(k)$.

Fix $y \in T_i$, and we can assume $y \in B_{10\er_{i+1}}(g)$ for some $g \in \cG_{i+1}$, since otherwise $\sigma_{i+1}$ is the identity on $B_{2\er_{i+1}}(y)$.  If $i = 0$, then we have for any $\tilde g \in \cG_1 \cap B_{9r_1}(y)$ the estimates
\begin{gather}
d(\tilde g, V_\infty(0, 1)) \leq 5 \beta_\infty(0, 5)\, , \quad d_G(V_\infty(\tilde g, r_1), V_\infty(0, 1)) \leq c(k) \beta_\infty(0, 5)\, .
\end{gather}
Since $T_0 \equiv V_\infty(0, 1)$, we can apply the squash lemma at scale $B_{2r_1}(y)$, then the regraphing lemma at scale $B_{r_1}(y)$, to deduce item \ref{it_2}.

Suppose $i \geq 1$.  By construction there is a $g' \in \cG_i$ so that $g \in B_{\er_i}(g')$, and a $g'' \in \cG_{i-1}$ so that $g' \in B_{\er_{i-1}}(g'')$.  Let us fix almost-projection $\pi''$ to $V_{g''}$.

By tilting Lemma \ref{lem:tilting-reif} and by construction, we have for any $\tilde g \in B_{6\er_i}(g)$ the estimates
\begin{gather}
d(\tilde g, g'' + V_{g''}) \leq c \beta_\infty(g'', 3\er_{i-1}) \er_i\, , \quad d_G(V_\infty(\tilde g, \er_i), V_{g''}) \leq c \beta_\infty(g'', 3\er_{i-1})\, ,
\end{gather}
and similarly, for any $\tilde g \in \cG_{i+1} \cap B_{9\er_{i+1}}(y)$, 
\begin{gather}
d(\tilde g, g'' + V_{g''}) \leq c \beta_\infty(g'', 3\er_{i-1}) \er_{i+1}\, , \quad d_G(V_\infty(\tilde g, \er_{i+1}), V_{g''}) \leq c \beta_\infty(g'', 3\er_{i-1})\, .
\end{gather}

We can then use our inductive hypothesis, the structure of $\sigma_i$, and the squash lemma part C) at scale $B_{\er_i}(g')$, to obtain
\begin{gather}\label{eqn:graph-reif}
T_i \cap B_{2\er_i}(g') = \graph_{\Omega, \pi''}(f)\, , \quad \er_i^{-1} ||f|| + \lip(f) \leq c \beta_\infty(g'', 3\er_{i-1})\, , \quad B_{1.5\er_i}(g) \cap (g'' + V_{g''}) \subset \Omega\, ,
\end{gather}
where $c$ is \emph{independent} of $\Lambda$.  Since $B_{6\er_{i+1}}(y) \subset B_{1.1\er_i}(g')$, we can now use the squash lemma part B) at scale $B_{2\er_{i+1}}(y)$, then the regraphing lemma at scale $B_{\er_{i+1}}(y)$, to deduce item \ref{it_2}.

As before, we can apply the squash lemma part A) to obtain the estimate
\begin{gather}\label{eqn:C0-est-reif}
||\sigma_{i+1}(x) - x|| \leq c(k) \delta \er_{i+1} \quad \forall x \in T_i\, .
\end{gather}

Moreover, part A) also gives the estimate
\begin{gather}\label{eqn:C1-est-reif}
||(\sigma_{i+1}(x) - \sigma_{i+1}(y)) - (x - y)|| \leq c(k) \delta ||x - y|| \quad \forall x, y \in T_i\, .
\end{gather}
We explain.  When $||x - y|| < 2\er_i$, then we can use \eqref{eqn:graph-reif} and the squash lemma to obtain \eqref{eqn:C1-est-reif}.  Otherwise, when $||x - y|| \geq 2\er_i$, then we can use \eqref{eqn:C0-est-reif} to get
\begin{gather}
||(\sigma_{i+1}(x) - \sigma_{i+1}(y)) - (x - y)|| \leq 2 c \delta \er_{i+1} \leq c(k) \delta ||x - y||\, .
\end{gather}

\subsubsection{Item \ref{it_3}: bi-H\"older estimates}

Let us fix an $x, y \in B_3 \cap V_\infty(0, 1)$.  Set $m$ be the maximal integer so that $||\tau_i(x) - \tau_i(y)|| \leq 6\er_i$ for all $i \leq m$.  We have by estimate \eqref{eqn:C1-est-reif} the bound
\begin{gather}
||\tau_m(x) - \tau_m(y)|| \geq (1-c(k)\delta)^m ||x - y||\, ,
\end{gather}
and so, provided $1-c(k)\delta \geq 1/2$, we have $m \leq a(10\log(6) - \log(||x - y||))$ for some absolute constant $a$.

Therefore, using \eqref{eqn:C1-est-reif}, we have for any $i \leq m$ the bounds
\begin{align}
&||\tau_i(x) - \tau_i(y)|| \leq (1+c\delta)^m ||x - y|| \leq (1+c(k)\delta) ||x - y||^{1-a \log(1+c(k)\delta)}, \quad \text{ and } \label{eqn:holder-bound-1} \\
&\quad ||\tau_i(x) - \tau_i(y)|| \geq (1-c\delta)^m ||x - y|| \geq (1-c(k)\delta) ||x - y||^{1-a\log(1-c(k)\delta)}\, . \label{eqn:holder-bound-2}
\end{align}
As in Section \ref{sec:bi-lip}, we can use \eqref{eqn:C0-est-reif} deduce for any $i \geq m$ the bound
\begin{gather}
\Big| ||\tau_i(x) - \tau_i(y)|| - ||\tau_m(x) - \tau_m(y)|| \Big| \leq c(k) \delta r_m \leq c(k) \delta ||\tau_m(x) - \tau_m(y)||\, .
\end{gather}
Combining this with \eqref{eqn:holder-bound-1}, \eqref{eqn:holder-bound-2}, and ensuring $\delta(k, \gamma)$ is sufficiently small, we obtain the required bi-H\"older estimate.

\subsubsection{Item \ref{it_4}: bi-Lipschitz estimates}

Let us assume the summability condition \eqref{eqn:reif-hyp}.  The proof is identical to Section \ref{sec:bi-lip}.  Fix $x, y \in B_3 \cap V_\infty(0, 1)$, and choose $m$ maximal so that $||\tau_i(x) - \tau_i(y)|| \leq 6\er_i$ for all $i \leq m$.  Using \eqref{eqn:graph-reif} and the squash Lemma \ref{lem:squash} part D), we obtain
\begin{gather}
\left| \frac{||\tau_{i+1}(x) - \tau_{i+1}(y)||}{||\tau_i(x) - \tau_i(y)||}  - 1 \right| \leq c(k, \rho_X)\beta_\infty(\tau_i(x), 5\er_{i-1})^\alpha \quad \forall i \leq m - 1\, .
\end{gather}
By the same computation as \eqref{eqn:lip-bound-calc-1}, \eqref{eqn:lip-bound-calc-2}, we deduce, ensuring $\delta(k,\rho_X)$ is sufficiently small, 
\begin{gather}
e^{-c(k,\rho_X)Q^\alpha} ||x - y|| \leq ||\tau_i(x) - \tau_i(y)|| \leq e^{c(k,\rho_X) Q^\alpha} ||x - y|| \quad \forall i \leq m\, .
\end{gather}

On the other hand, again by the same argument as in Section \ref{sec:bi-lip}, we have
\begin{gather}
\Big| ||\tau_i(x) - \tau_i(y)|| - ||\tau_m(x) - \tau_m(y)|| \Big| \leq c(k, \rho_X) \delta^\alpha ||\tau_m(x) - \tau_m(y)|| \quad \forall i \geq m\, .
\end{gather}
Since we can clearly assume $\delta \leq Q$, this establishes the required bi-Lipschitz bound.

\subsubsection{Item \ref{it_5}: covering control}

By inductive hypothesis we have
\begin{gather}\label{eqn:covering-reif-1}
d_H(T_i \cap B_{1+\er_i/2}, S \cap B_{1+\er_i/2}) < \er_i\, ,
\end{gather}
and therefore by item ``graphicality'' and estimate \eqref{eqn:beta-plane-reif}, we have
\begin{gather}\label{eqn:covering-reif-2}
d_H(T_i \cap B_{1+\er_i/2}, S \cap B_{1+\er_i/2}) \leq c(k) \delta \er_i\, .
\end{gather}

We elaborate.  Given any $y \in T_i \cap B_{1+\er_i/2}$, by \eqref{eqn:covering-reif-1} and construction we can find a $g \in \cG_i \cap B_{6\er_i/5}(y)$.  Graphicality and estimate \eqref{eqn:beta-plane-reif} imply that
\begin{gather}
d(y, S) \leq d(y, g + V_g) + d_H( (g + V_g) \cap B_{5\er_i}(y), S \cap B_{5\er_i}(y)) \leq \Lambda \delta \er_i + c(k) \delta \er_i\, .
\end{gather}
Conversely, given $z \in S \cap B_{1+\er_i/2}$, we can pick a $g \in \cG_i \cap B_{\er_i}(z)$ and $y \in T_i \cap B_{\er_i}(z)$.  Then using graphicality an the definition of $\beta_\infty$ we obtain
\begin{gather}
d(z, T_i) \leq d(z, g + V_g) + d_H( (g + V_g) \cap B_{2\er_i}(y), T_i \cap B_{2\er_i}(y)) \leq 2\delta \er_i + \Lambda \delta \er_i\, .
\end{gather}
This establishes \eqref{eqn:covering-reif-2}.

Now using the $C^0$ estimate \eqref{eqn:C0-est-reif} with \eqref{eqn:covering-reif-2} we deduce
\begin{gather}
d_H(T_{i+1} \cap B_{1+\er_{i+1}/2}, S \cap B_{1+\er_{i+1}/2}) \leq c(k) \delta \er_i < \er_{i+1}\, ,
\end{gather}
provided $\delta(k)$ is sufficiently small.  This proves item \ref{it_5}.

\bibliographystyle{aomalpha}
\bibliography{ENV_Reifenberg}

\end{document}